%% file: grassmann-arcs-161105.tex
\documentclass[a4paper,11pt,reqno]{amsart}

\expandafter\let\csname ver@amsthm.sty\endcsname\relax
\let\theoremstyle\relax

\usepackage[utf8]{inputenc}

\usepackage{etex}

\usepackage[
    breaklinks,
    colorlinks,
    citecolor=teal,
    linkcolor=teal,
    urlcolor=teal,
    pagebackref=true,
    hyperindex
]{hyperref}

\usepackage{fancyhdr}

\usepackage[
    hscale=0.7,
    vscale=0.75,
    headheight=13pt,
    centering,
]{geometry}

\usepackage{amsmath}
\usepackage{amsthm}
\usepackage{amssymb}
\usepackage{mathtools}
\usepackage{stmaryrd}
\usepackage{mathdots}
\usepackage[all,cmtip]{xy}
\usepackage{xcolor}
\usepackage{framed}
\usepackage[capitalize]{cleveref}
\usepackage{array}
\usepackage{tikz}
\usetikzlibrary{
    arrows,
    backgrounds,
    intersections,
    decorations.pathmorphing,
    calc,
    shadings,
    shapes
}

\theoremstyle{plain}

\newtheorem{theorem}{Theorem}[section]
\newtheorem{proposition}[theorem]{Proposition}
\newtheorem{lemma}[theorem]{Lemma}
\newtheorem{corollary}[theorem]{Corollary}
\newtheorem{problem}[theorem]{Problem}

\theoremstyle{definition}

\newtheorem{definition}[theorem]{Definition}
\newtheorem{notation}[theorem]{Notation}

\newtheorem{remark}[theorem]{Remark}

\numberwithin{figure}{section}

\numberwithin{equation}{section}

\input{article-style.tex}

\newcommand{\C}{{\mathbb{C}}}
\newcommand{\largewedge}{\mbox{\Large $\wedge$}}
\newcommand{\psr}{{\mathbb C\llbracket t\rrbracket}}
\newcommand{\lsr}{{\mathbb C(\!( t )\!)}}
\newcommand{\GL}{{\rm GL}}
\newcommand{\Cont}{{\rm Cont}}
\newcommand{\codim}{\operatorname{codim}}
\newcommand{\ord}{\operatorname{ord}}
\newcommand{\val}{\operatorname{val}}
\newcommand{\lct}{\operatorname{lct}}

\begin{document}

\title{The arc space of the Grassmannian}

\author{Roi Docampo}

\address{
    Instituto de Ciencias Matemáticas (ICMAT)\\
    c/ Nicolás Cabrera, 13--15\\
    Campus de Cantoblanco, UAM\\
    28049 Madrid, Spain
}

\email{roi.docampo@icmat.es}

\author{Antonio Nigro}

\address{
    Instituto de Matemática e Estatística\\
    Universidade Federal Fluminense\\
    Rua Mário Santos Braga, s/n\\
    24020--140 Niterói, RJ, Brasil
}

\email{nigro@impa.br}

\subjclass[2010]{Primary 14E18, 14M15; Secondary 14B05.}
\keywords{Arc spaces, jet schemes, Grassmannian, Schubert varieties, plane
partitions, planar networks, log canonical threshold, Nash problem.}

\begin{abstract}
We study the arc space of the Grassmannian from the point of view of the
singularities of Schubert varieties. Our main tool is a decomposition of the
arc space of the Grassmannian that resembles the Schubert cell decomposition of
the Grassmannian itself. Just as the combinatorics of Schubert cells is
controlled by partitions, the combinatorics in the arc space is controlled by
plane partitions (sometimes also called 3d partitions). A combination of a
geometric analysis of the pieces in the decomposition and a combinatorial
analysis of plane partitions leads to invariants of the singularities. As an
application we reduce the computation of log canonical thresholds of pairs
involving Schubert varieties to an easy linear programming problem. We also
study the Nash problem for Schubert varieties, showing that the Nash map is
always bijective in this case.
\end{abstract}

\maketitle

\section*{Introduction}

Given a variety $X$, an arc on $X$ is a germ of a parametrized curve, and the
arc space $J_\infty X$ is a natural geometric object parametrizing arcs. Arc
spaces have been featured repeatedly in recent years in algebraic geometry,
from several points of view. They appear in singularity theory, mainly via the
study of Nash-type problems \cite{Nas95}, a tool to understand resolutions of
singularities. They are a key ingredient in the theory of motivic integration,
introduced by Kontsevich \cite{Kon95,DL99} and with many applications to the
study of invariants of varieties and of singularities. And they are used in
birational geometry in the study of singularities of pairs \cite{Mus02,EM09}.

The purpose of this paper is the study of the arc space of the Grassmannian
and of its Schubert varieties.

Recall that the Grassmannian $G(k,n)$ is the space parametrizing
$k$-dimensional vector subspaces in $\mathbb C^n$. This is a fundamental object
in geometry, a source of many examples, and used often as the starting point in
the construction of other varieties and invariants. The Schubert varieties
appear as natural sub-objects inside the Grassmannian, and they provide a rich
collection of singular varieties. These singularities have been studied
thoroughly in the literature, in many different contexts. 

The main tool that we propose is a decomposition of the arc space of the
Grassmannian into pieces that we call \emph{contact strata}. This
stratification can be defined in two equivalent ways: either using orders of
contact of arcs with respect to Schubert varieties, or using invariant factors
of lattices naturally associated to arcs. The resulting pieces, the contact
strata, should be thought as the arc space analogue of the Schubert cells of
the Grassmannian itself.
All subsets in the arc space which are relevant for the study of Schubert
varieties (for example contact loci of Schubert varieties) can be decomposed
using contact strata. This essentially reduces the computation of invariants to the
understanding of contact strata. 

For the study of contact strata we were inspired by previous work on the
totally positive Grassmannian (mainly \cite{FZ00}). In particular,
\emph{weighted planar networks} appear repeatedly in the paper. We use them to
construct arcs, and to control (in a combinatorial way) the orders of contact
of an arc with respect to Schubert varieties. Some special planar networks lead
to a very explicit description of contact strata, which is useful throughout
the paper. These combinatorial techniques allow us to classify contact strata,
understand their basic geometry, and determine their position with respect to
each other. 

As a main application of our study, we give an effective algorithm to compute
log canonical thresholds of pairs involving Schubert varieties. This is done by
reducing the problem to maximizing a linear function on a explicit rational
convex polytope, which we call the \emph{polytope of normalized Schubert
valuations}. After this is done, techniques from the theory of linear
programming provide fast algorithms for the actual computation of the log
canonical threshold.

For completeness, we also include the solution to the Nash problem for Schubert
varieties in the Grassmannian. This turns out not to need a deep understanding
of the arc space. We show that there exist resolutions of singularities for
which the exceptional components are in bijection with the irreducible
components of the singular locus. This immediately implies that the Nash map is
bijective, and that the Nash families are also in bijection with the components
of the singular locus.

In the remainder of the introduction we give a more precise overview of the
main results of the paper.

\subsection*{The stratification of the arc space}

The $\mathbb C$-valued points of the arc space $J_\infty G(k,n)$ can be
described using the defining universal property of the Grassmannian. They
correspond to lattices $\Lambda \subset \psr^n$ for which the quotient
$\psr^n/\Lambda$ is a free module of rank $n-k$. To get our stratification we
classify these lattices according to their position with respect to a flag.

More precisely, start with a full flag in $\mathbb C^n$ and consider the
corresponding flag of lattices $0 = F_0 \subsetneq F_1 \subsetneq \cdots
\subsetneq F_n = \psr^n$. For a given lattice $\Lambda \subset \psr^n$, the
isomorphism type of the quotient module $\psr^n/(\Lambda + F_i)$ is determined
by its invariant factors:
\[
    \frac{\psr^n}{\Lambda + F_i}
    \,\simeq\,
    \bigoplus_{j=1}^{n-k}
    \frac{\psr}{\left(t^{b_{i,j}}\right)},
\]
where the numbers $b_{i,j} \in \{ 0, 1, \ldots, \infty \}$ verify
$b_{i,j} \geq b_{i,j+1}$, and we use the convention
$t^\infty = 0$. We consider the numbers $\beta_{i,j}$ given by
\[
    \beta_{i+k-n+j,j} = 
    b_{i,j}.
\]
In \cref{schubert-cond-lattices} we see that $\beta_{i,j} = \infty$ for $i \leq
0$ and $\beta_{i,j}=0$ for $i>k$, so the only relevant values of $\beta_{i,j}$
occur for $1 \leq i \leq k$ and $1 \leq j \leq n-k$. In this way we obtain a
matrix $\beta = (\beta_{i,j})$ of size $k \times (n-k)$, which we call the
\emph{invariant factor profile} of the arc $\Lambda$. See
\cref{fig:schubert-cond-lattices} for a diagram explaining the meaning of the
numbers $\beta_{i,j}$ in the particular case of $G(2,5)$.

\begin{figure}[ht]
    \centering
    \includegraphics{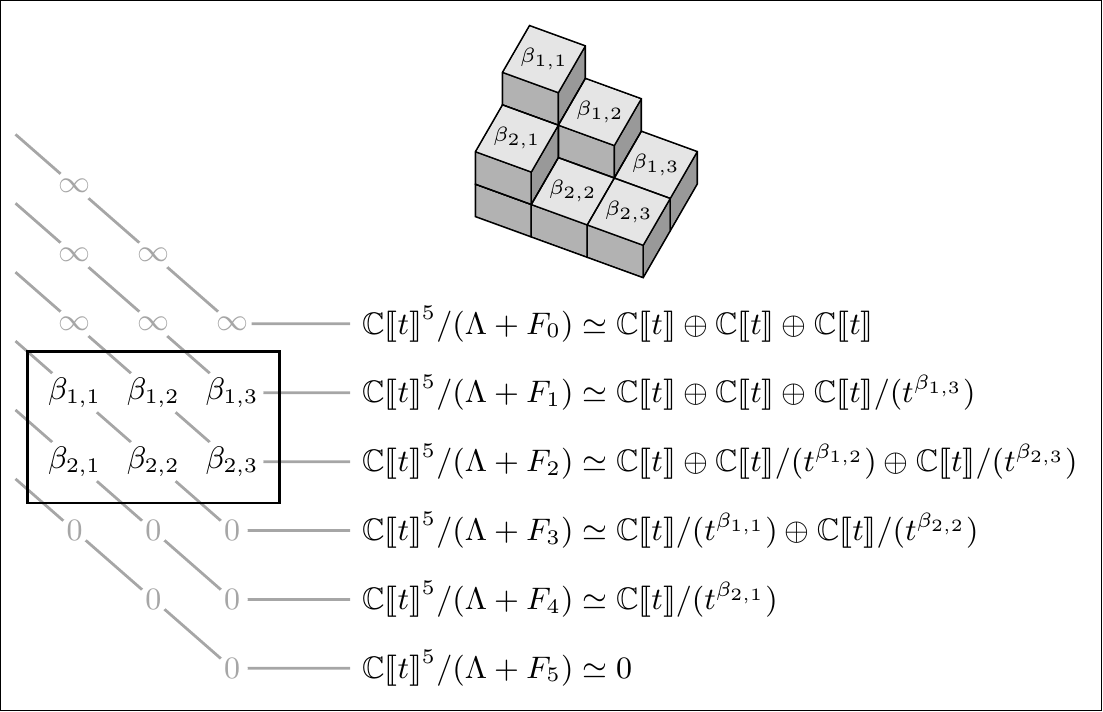}
    \caption{The meaning of the invariant factor profile for $G(2,5)$.}
    \label{fig:schubert-cond-lattices}
\end{figure}

We have a decomposition
\[
    J_\infty G(k,n)  = \bigcup_\beta \mathcal C_\beta,
\]
where $\mathcal C_\beta$ is the collection of arcs with invariant factor
profile $\beta$. The pieces $\mathcal C_\beta$ are called \emph{contact
strata}. The main facts about invariant factor profiles and contact strata are
summarized in the following \lcnamecref{intro-main}.

\begin{theorem}\label{intro-main}
\begin{enumerate}

    \item[{}]

    \item\label{fact-1} The invariant factor profile of an arc is determined by
        the orders of contact of the arc with respect to the Schubert
        varieties, and vice versa (\cref{schubert-cond-lattices}). 

    \item\label{fact-3} The invariant factor profile is a plane
        partition\footnote{In the literature, plane partitions are sometimes
        called 3d partitions.}, i.e., $\beta_{i,j} \geq \beta_{i+1,j}$ and
        $\beta_{i,j} \geq \beta_{i,j+1}$ (\cref{A-restr,if-prof}).

    \item\label{fact-4} Every plane partition is the invariant factor profile
        of some arc (\cref{if-prof,c-works}).

    \item\label{fact-5} Contact strata are irreducible (\cref{irred}).

\end{enumerate}
\end{theorem}

Because of fact (\ref{fact-1}), the decomposition into contact strata is very
relevant for the study of the singularities of Schubert varieties. For example,
for the computation of log canonical thresholds (\cref{intro-lct}) we use that
contact loci of Schubert varieties are unions of contact strata.

Among the above facts, the most delicate is (\ref{fact-4}). For its proof, we
need to produce arcs with prescribed invariant factor profile, and we do not
know of a simple way of achieving this. \cref{sec:planar-networks} is devoted
to this issue. Here is where we start using weighted planar networks (as
mentioned above, inspired by \cite{FZ00}). With this construction, we are able
to use combinatorial techniques to control the orders of contact with respect
to Schubert varieties. The resulting description of contact strata is very
explicit. For example, to prove fact (\ref{fact-5}) we use planar networks to
describe the generic point of each contact stratum.

There is another natural stratification of $J_\infty G(k,n)$, considering
orbits of the action of the group of arcs $J_\infty B$, where $B \subset \GL_n$
is the Borel subgroup. An analysis of this orbit decomposition would be in the
lines of previous approaches to the study of arc spaces. For example, this is
the main idea used in the cases of toric varieties \cite{IK03, Ish04} and of
determinantal varieties \cite{Doc13}. But in our case we found that the
structure of the orbits is too complex for our study (see \cref{sec:orbits}).
Contact strata provide a coarser decomposition, simpler to understand, and
enough for our purposes.

\subsection*{Geometry of contact strata}

In order to compute invariants it is not enough to just stratify the arc space,
we need to understand the geometry of each of the strata, and to study how
these pieces are placed with respect to each other. From our point of view, we
consider the following to be the main question.

\begin{problem}[Nash problem for contact strata]
    \label{problem:intro-nash}
    Given two plane partitions $\beta$ and $\beta'$, determine whether there is
    a containment $\overline{\mathcal C}_\beta \supset \overline{\mathcal
    C}_{\beta'}$.
\end{problem}

\cref{sec:gen-nash} is devoted to the study of \cref{problem:intro-nash}. We
are able to give an answer in several cases by analyzing the combinatorics of
plane partitions. The main results are \cref{nash-trivial}, which gives a
necessary condition for a containment to exist, \cref{nash-hard,nash-c}, which
give sufficient conditions, and \cref{nash-2}, which gives a complete answer
for $G(2,4)$. Again, for these results we often use planar networks to
transform geometric questions into combinatorics. A general answer to
\cref{problem:intro-nash} seems very difficult. 

Despite the fact that our answer to \cref{problem:intro-nash} is only partial,
we are able to use it effectively to compute invariants of contact strata.
Namely, we prove the following result.

\begin{theorem}
    \label{intro-codim}
    The codimension of a contact stratum $\mathcal C_\beta$ in $J_\infty
    G(k,n)$ is the number of boxes in the plane partition $\beta$:
    \[
        \codim(\mathcal C_\beta, J_\infty G(k,n)) = \sum_{i,j} \beta_{i,j}
    \]
\end{theorem}

This \lcnamecref{intro-codim} is proven studying chains of containments of
closures of contact strata, and these are provided by our answers to
\cref{problem:intro-nash}. The codimensions of contact strata immediately give
log discrepancies of the corresponding valuations (which we call Schubert
valuations, see \cref{sec:valuations}), and the computation of log canonical
thresholds gets reduced to the analysis of the combinatorics of plane
partitions.

\subsection*{Log canonical thresholds of Schubert varieties}

The Schubert varieties inside $G(k,n)$ are indexed by partitions $\lambda =
(\lambda_1 \lambda_2 \cdots)$ with at most $k$ parts of size at most $n-k$. We
denote them $\Omega_\lambda$. A partition of the form $\lambda = (b^a) = (b
\,\overset{a}\dots\, b)$ is called rectangular. The following result is proven
in \cref{sec:lct}. 

\begin{theorem}
    \label{intro-lct-sq}
    Let $\Omega_\lambda$ be a Schubert variety in $G(k,n)$, and assume that
    $\lambda = (b^a)$ is rectangular. Consider $\lambda^s = ((b+s)^{a+s})$, and
    let $|\lambda^s| = (a+s)(b+s)$ denote the number of boxes in $\lambda^s$.
    Let $r = \min\{k-a, n-k-b\}$. Then the log canonical threshold of the pair
    $(G(k,n), \Omega_\lambda)$ is
    \[
        \lct(G(k,n), \Omega_\lambda) 
        = \min_{s=0 \ldots r} 
        \left\{
            \frac {|\lambda^s|}{s+1}
        \right\}
        = \min_{s=0 \ldots r} 
        \left\{
            \frac {(a+s)(b+s)}{s+1}
        \right\}.
    \]
\end{theorem}

\begin{figure}[ht]
    \centering
    \includegraphics{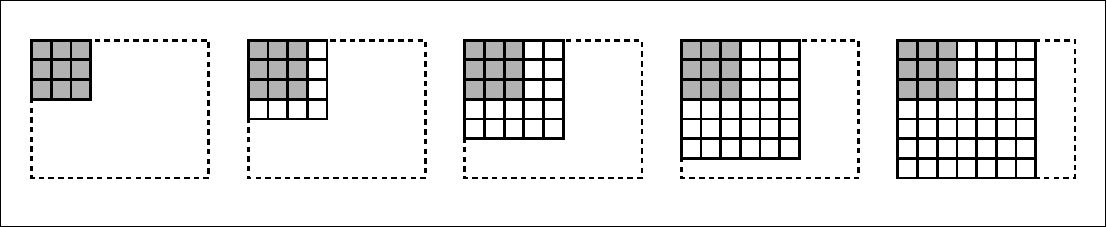}
    \caption{The partitions $\lambda^s$ in \cref{intro-lct-sq} for
    $\lambda=(333)$ in $G(7,16)$.}
    \label{fig:thick-sq}
\end{figure}

It should be noted that Schubert varieties corresponding to rectangular
$\lambda$ are essentially generic determinantal varieties, and therefore their
log canonical thresholds were already known, see \cite{Joh03,Doc13}. But the
proof that we provide is new, and also gives a natural combinatorial
interpretation for the numbers appearing in the formula.

The partitions $\lambda^s$ in \cref{intro-lct-sq} are obtained from $\lambda$
by adding rims of boxes, without exceeding the maximal allowed size (the
rectangle with $k$ rows and $(n-k)$ columns). For example, the case of
$\lambda=(333)$ in $G(7,16)$ appears in \cref{fig:thick-sq}. For the log
canonical threshold we get:
\[
    \lct(G(7,16),\Omega_{(333)})
    =
    \min\left\{
        \frac{9}{1},
        \frac{16}{2},
        \frac{25}{3},
        \frac{36}{4},
        \frac{49}{5}
    \right\}
    =
    \frac{16}{2} = 8.
\]

For more general Schubert varieties (when $\lambda$ is not necessarily
rectangular), we have an analogue version of \cref{intro-lct-sq} which
expresses the Arnold multiplicity (the reciprocal of the log canonical
threshold) as the maximum of a linear function on (the extremal points of) a
rational convex polytope.

Let $\mathbb R{\rm PP}(k,n)$ the convex hull of the set of plane partitions
inside $\mathbb R^{k(n-k)}$. Then $\mathbb R{\rm PP}(k,n)$ is a pointed
rational convex polyhedral cone with vertex at the origin. For a point $\beta
\in \mathbb R{\rm PP}(k,n)$, we denote by $|\beta| = \sum \beta_{i,j}$ the
\emph{volume} of $\beta$, and we let ${\rm SV}(k,n)$ be the subset of $\mathbb
R{\rm PP}(k,n)$ containing elements of volume $1$:
\[
    {\rm SV}(k,n)
    =
    \{\,\,
    \beta \in \mathbb R{\rm PP}(k,n) 
    \,\,\mid\,\,
    |\beta| = 1 
    \,\,\}.
\]
We call ${\rm SV}(k,n)$ the \emph{polytope of normalized Schubert valuations}.
The structure of ${\rm SV}(k,n)$ is very explicit. It is a bounded rational
convex polytope whose vertices are in bijection with non-empty partitions with
at most $k$ parts of size at most $n-k$. It has a natural simplicial structure,
where the $r$-dimensional simplices correspond to chains of partitions
$\lambda^0 \subsetneq \lambda^1 \subsetneq \cdots \subsetneq \lambda^r$. See
\cref{fig:polytope} for the example of $G(2,4)$.

\begin{figure}[ht]
    \centering
    \includegraphics{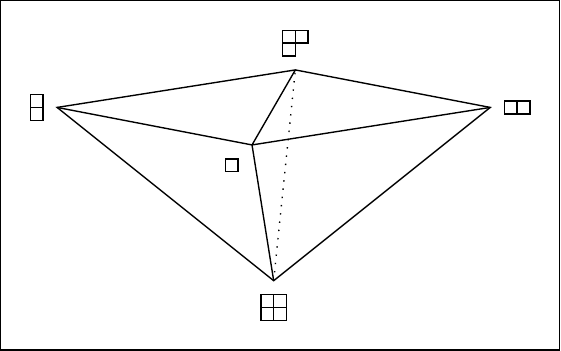}
    \caption{The simplicial structure on ${\rm SV}(2,4)$.} 
    \label{fig:polytope}
\end{figure}

Given a partition $\lambda$, the order of contact with respect to
$\Omega_\lambda$ induces a function on ${\rm SV}(k,n)$:
\[
    \ord(\lambda) \colon {\rm SV}(k,n)  \to \mathbb R,
    \qquad
    \beta \mapsto \ord(\lambda)(\beta) = \ord_\beta(\Omega_\lambda).
\]
The function $\ord(\lambda)$ can be described explicitly, by considering the
corners of the partition and the half diagonals emanating from these corners.
We refer to \cref{sec:lct} for details; see \cref{fig:diagonals} for some
examples. From this description it follows that $\ord(\lambda)$ is a concave
piecewise-linear function. We denote by $H_\lambda \subset \mathbb R^{k(n-k)}$
the linear subspace obtained as the zero locus of the linear equations defining
$\ord(\lambda)$. Notice that $H_\lambda$ is the biggest linear space contained
in the corner locus of $\ord(\lambda)$, and in particular $\ord(\lambda)$ is
linear on $H_\lambda$.

\begin{figure}[ht]
    \centering
    \includegraphics{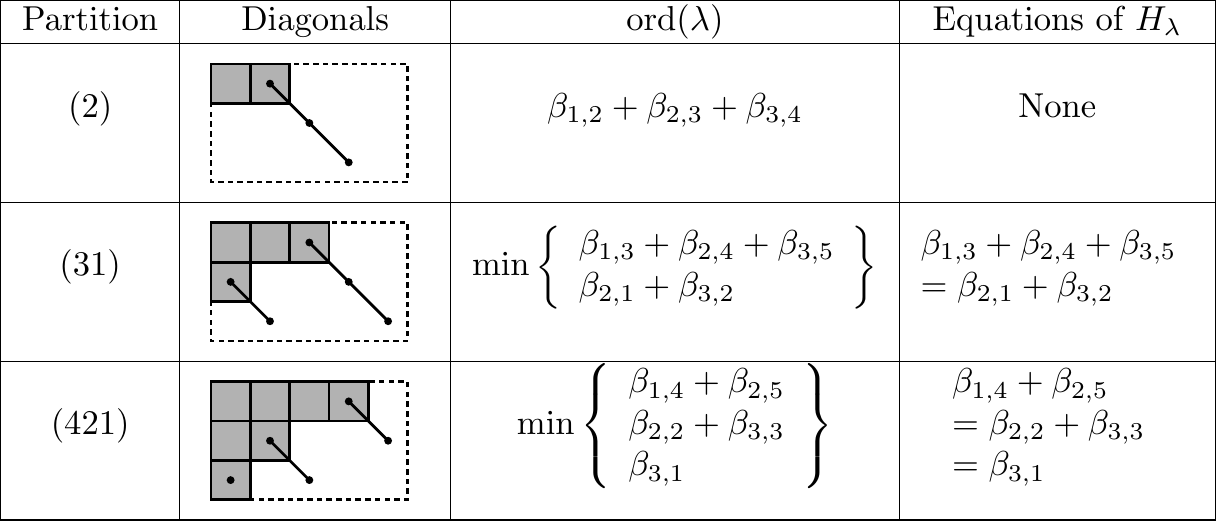}
    \caption{Examples of $\ord(\lambda)$ and $H_\lambda$ in $G(3,8)$.} 
    \label{fig:diagonals}
\end{figure}

\begin{theorem}
    \label{intro-lct}
    Let $\Omega_\lambda$ be a Schubert variety in $G(k,n)$. Then the Arnold
    multiplicity of the pair $(G(k,n),\Omega_\lambda)$ is the maximum of
    $\ord(\lambda)$ on ${\rm SV}(k,n) \cap H_\lambda$.
\end{theorem}

Notice that ${\rm SV}(k,n) \cap H_\lambda$ is a rational convex polytope, and
in particular the maximum of $\ord(\lambda)$, which is linear on the polytope,
is achieved on an extremal point. For example, when $\lambda = (b^a)$ is
rectangular, the partitions $\lambda^0, \ldots, \lambda^r$ appearing in
\cref{intro-lct-sq} give some of the extremal points of ${\rm SV}(k,n) \cap
H_\lambda$ (in the rectangular case, the other extremal points are easy to
discard).

From \cref{intro-lct}, to obtain an actual value for the Arnold multiplicity
(and hence for the log canonical threshold) one would normally use a computer.
The problem of maximizing a linear function on a convex polytope is the subject
of \emph{linear programming}. This is a highly developed theory, providing
several very efficient algorithms to calculate both approximate and exact
solutions. Both the polytope ${\rm SV}(k,n) \cap H_\lambda$ and the linear
function $\ord(\lambda)$ are straightforward to describe to a computer, and in
practice we found that the standard libraries dedicated to linear programming
are very fast at computing log canonical thresholds, even for large values of
$k$ and $n$ and complicated partitions $\lambda$.

\section{Generalities on arc spaces}

\label{sec:arcs}

In this section we review the theory of arc spaces. For a full treatment,
including proofs, we refer the reader to
\cite{ELM04,Voj07,Ish08,dFEI08,Mor09,EM09}.

\subsection*{Basic conventions} We work over the complex numbers $\mathbb C$,
although most of our results would be valid after replacing $\mathbb C$ with an
arbitrary algebraically closed field. All schemes are quasi-compact,
quasi-separated, and defined over $\mathbb C$, but not necessarily Noetherian.
By variety we mean a separated, reduced, and irreducible scheme of finite type
over $\mathbb C$. All morphisms of schemes are defined over $\mathbb C$.

\subsection*{Arcs and jets} Fix a scheme $X$. An \emph{arc} $\gamma$ on $X$ is
a morphism
\[
    \gamma \colon {\rm Spec}\,\psr \to X.
\]
Similarly, for a non-negative integer $m$, a \emph{jet} $\gamma$ on $X$ of
order $m$ is a morphism
\[
    \gamma \colon {\rm Spec}\,\mathbb C[t]/(t^{m+1}) \to X.
\]
Notice that a jet of order $1$ is a tangent vector. More generally, for a
$\mathbb C$-algebra $A$, morphisms of the type
\[
    {\rm Spec}\,A\llbracket t \rrbracket \to X
    \qquad\text{and}\qquad
    {\rm Spec}\,A[t]/(t^{m+1}) \to X
\]
are called $A$-valued arcs and jets on $X$. An $A$-valued arc/jet should be
thought as a family of arcs/jets parametrized by ${\rm Spec}\,A$.

We denote by $0$ the closed points of ${\rm Spec}\,\psr$ and ${\rm
Spec}\,\mathbb C[t]/(t^{m+1})$, and by $\eta$ the generic point of ${\rm
Spec}\,\psr$. For an arc or jet $\gamma$, the point $\gamma(0)$ is called the
\emph{center}, \emph{origin}, or \emph{special point} of $\gamma$. If $\gamma$
is an arc, $\gamma(\eta)$ is called the \emph{generic point} of $\gamma$. This
terminology is also used for $K$-valued arcs and jets, where $K$ is a field
extension of $\mathbb C$. A $K$-valued $\gamma$ arc is called \emph{fat} if its
generic point $\gamma(\eta)$ is the generic point of $X$; otherwise it is
called \emph{thin}.

Let $\mathfrak a \subseteq \mathcal O_X$ be a sheaf of ideals in $X$. For an
arc or jet $\gamma$, the inverse image $\gamma^{-1}(\mathfrak a)$ is an ideal
of the form $(t^e)$, for some number $e \in \{0,1,\ldots,\infty\}$. Here we use
the convention $t^\infty = 0$ to cover the case where the inverse image is the
zero ideal. This number $e$ is called the \emph{order of contact} of $\gamma$
along the ideal $\mathfrak a$, and denoted $\ord_\gamma(\mathfrak a)$. If $Y$
is the closed subscheme of $X$ defined by $\mathfrak a$, we also write
$\ord_\gamma(Y)$ for this order.

\subsection*{Arc spaces and jet schemes} The \emph{arc space} of $X$ is the
universal object parametrizing families of arcs on $X$. It is denoted $J_\infty
X$ and it is characterized\footnote{Usually the arc space is defined as the
projective limit of the jet schemes, so its functor of points is in principle
different from the one written here. But it follows from
\cite[Corollary~1.2]{Bha} that both definitions agree.} by its functor of
points:
\[
    J_\infty X(A)
    =
    {\rm Hom}_{\mathbb C-{\rm Schemes}}
    \left(
        {\rm Spec}\,A\llbracket t \rrbracket,
        \,
        X
    \right).
\]
Similarly, the \emph{jet scheme} of order $m$ of $X$, denoted $J_m X$ is given
by
\[
    J_m X(A)
    =
    {\rm Hom}_{\mathbb C-{\rm Schemes}}
    \left(
        {\rm Spec}\,A[t]/(t^{m+1}),
        \,
        X
    \right).
\]
It can be shown that the arc space and the jet schemes are schemes. If $X$ is
of finite type, the jet schemes are also of finite type, but the arc space is
not (unless $X$ is zero-dimensional).

The natural quotient maps at the level of algebras
\[
    A\llbracket t \rrbracket 
    \to
    A[t]/(t^{m+1})
    \to
    A
\]
induce morphisms of schemes:
\[
    J_\infty X \to J_m X \to X
\]
These morphisms are affine, and are called the \emph{truncation maps}. The arc
space is the projective limit of the jet schemes via the truncation maps.

There are natural sections of the truncation maps at level zero:
\[
    X \to J_\infty X
    \qquad\text{and}\qquad
    X \to J_m X.
\]
The images of these sections are called the \emph{constant} arcs and jets. In
general there are no natural sections of the truncations $J_\infty X \to J_m X$
for $m \geq 1$.

The construction of arc spaces and jet schemes is functorial. Given a morphism
of schemes $f \colon X \to Y$, composition with $f$ induces natural morphisms
at the level of arc spaces and jet schemes. These morphisms are compatible with
the truncation maps:
\[
    \xymatrix{
        J_\infty X \ar[r] \ar[d]
        & J_\infty Y \ar[d]
        \\ J_m X \ar[r] \ar[d]
        & J_m Y \ar[d]
        \\ X \ar[r]
        & Y
    }
\]

As a consequence of functoriality, if $G$ is a group scheme, the arc space
$J_\infty G$ and jet schemes $J_m G$ are also groups. Moreover, if $G$ acts on
$X$, we get induced actions of $J_\infty G$ on $J_\infty X$, and of $J_m G$ on
$J_m X$. All these groups structures and actions are compatible with the
truncation maps.

\subsection*{Constructible sets and contact loci} From now on we assume that
$X$ is a variety. For the general definition of constructible set in a scheme,
we refer the reader to \cite[$0_{\rm III}$, \S9.1]{EGA3.1}. In the finite type
case (for the variety $X$ and for the jet schemes $J_m X$) this is the familiar
notion: a set is constructible if it is a finite boolean combination of Zariski
closed subsets. For the arc space a constructible set turns out to be the same
as a \emph{cylinder} \cite{ELM04}: the inverse image via a truncation map of a
constructible set in some jet scheme $J_m X$.

The most important examples of constructible sets in the arc space are contact
loci. Given a closed subscheme $Y \subset X$ and a number $p \in
\{0,1,\ldots\}$, we define
\[
    \Cont^{\geq p}(Y)
    =
    \left\{
        \,
        \gamma \in J_\infty X
        \,\mid\,
        \ord_\gamma(Y) \geq p
        \,
    \right\}.
\]
We also define $\Cont^{=p}(Y)$ in the obvious way, and analogous versions in
the jet schemes: $\Cont^{\geq p}_m(Y)$ and $\Cont^{=p}_m(Y)$. We call these
types of sets \emph{contact loci}. Notice that a contact locus in the arc space
is the inverse image of a contact locus in a jet scheme of high enough order
(the order of an arc is determined by the order of a high enough truncation).
In particular, contact loci are constructible.

\subsection*{Valuations} Let $R$ be an integral domain containing $\mathbb C$.
A (discrete, rank at most one) \emph{semi-valuation}
on $R$ is a function $v \colon R \to \mathbb \{0,1,\ldots,\infty\}$ satisfying the
following properties:
\begin{enumerate}
    \item $v(f g) = v(f) + v(g)$ for all elements $f,g \in R$,
    \item $v(f + g) \geq \min\{v(f), v(g)\}$ for all elements $f,g \in R$,
    \item $v(z) = 0$ for all non-zero constants $z \in \mathbb
        C\setminus\{0\}$, and
    \item $v(0) = \infty$.
\end{enumerate}
We say that $v$ is a \emph{valuation} if furthermore:
\begin{enumerate}
    \setcounter{enumi}{4}
    \item $v(f) = \infty$ if and only if $f=0$.
\end{enumerate}
We say that $v$ is \emph{trivial} if its only values are $0$ and $\infty$. The
greatest common divisor of the non-zero values of a non-trivial semi-valuation
$v$ is called its \emph{multiplicity}, and denoted $q_v$. The prime ideals
\[
    \mathfrak b_v = \{ f \in R \mid v(f) = \infty \}
    \qquad\text{and}\qquad
    \mathfrak c_v = \{ f \in R \mid v(f) > 0 \}
\]
are called the \emph{home} and \emph{center} of $v$. A semi-valuation is a
valuation precisely when its home is zero. A semi-valuation $v$ induces a
valuation in the standard sense in the field of fractions ${\rm
Frac}(R/\mathfrak b_v)$. The corresponding \emph{valuation ring} is denoted by
$\mathcal O_v \subset {\rm Frac}(R/\mathfrak b_v)$. Notice that $\mathcal O_v$
is either a field (if $v$ is trivial) or a discrete valuation ring of rank one.

Let $R_f$ be the localization of $R$ obtained by inverting $f$. Then the set of
semi-valuations of $R_f$ is in natural bijection with the set of
semi-valuations of $R$ for which $f$ has value zero. This allows us to glue
this construction, and talk about \emph{semi-valuations} and \emph{valuations}
on a variety $X$. Geometrically, a semi-valuation on $X$ can be thought as a
choice of a subvariety $Y \subset X$ (the home of the semi-valuation) and a
valuation (in the standard sense) on $Y$.

Semi-valuations and arcs are closely related. Let $\gamma$ be a point of
$J_\infty X$ in the sense of schemes, and let $K_\gamma$ be its residue field.
It corresponds to a $K_\gamma$-valued arc on $X$:
\[
    \gamma \colon {\rm Spec}\,K_\gamma\llbracket t \rrbracket
    \to X.
\]
Then $\ord_\gamma$ is a semi-valuation on $X$. Its home is $\gamma(\eta)$, the
generic point of the arc. Its center in the sense of semi-valuations agrees
with $\gamma(0)$, the center of $\gamma$ in the sense of arcs. It is trivial if
and only if $\gamma$ is a constant arc, and it is a valuation if and only if
$\gamma$ is fat. More geometrically, this construction can be reduced to use
only $\mathbb C$-valued arcs. We consider the closure of $\gamma$ in the arc
space, denoted $\mathcal C = \overline{\{\gamma\}} \subset J_\infty X$. Then
the semi-valuation $\ord_\gamma$ can be recovered from the semi-valuations of
the arcs in the family $\mathcal C$:
\[
    \ord_\gamma
    =
    \ord_{\mathcal C}
    =
    \min \{\, \ord_\alpha \,\mid\, \alpha \in \mathcal C \,\}.
\]

Conversely, every semi-valuation is induced by some arc. Let $v$ be a
non-trivial semi-valuation on $X$ with multiplicity $q_v$, and consider its
valuation ring $\mathcal O_v$. The completion $\widehat{\mathcal O}_v$ is
isomorphic to the power series ring $\widehat K_v\llbracket t \rrbracket$,
where $\widehat K_v$ is the residue field of $\widehat{\mathcal O}_v$. For a
choice $\varphi$ of any such isomorphism we get a $\widehat K_v$-valued arc
$\gamma_{v,\varphi}$:
\[
    \xymatrix{
        *[l]{
            {\rm Spec}\,\widehat K_v\llbracket t \rrbracket
        }
        \ar[r]^-{t\,\mapsto\,t^{q_v}}
        \ar@(dr,dl)[rrrr]_{\gamma_{v,\varphi}}
        & {\rm Spec}\,\widehat K_v\llbracket t \rrbracket
        \ar[r]^-\varphi
        & {\rm Spec}\,\widehat{\mathcal O}_v
        \ar[r]
        & {\rm Spec}\,\mathcal O_v
        \ar[r]
        & X.
        \\
    }
\]
It is straightforward to check that ${\ord_{\gamma_{v,\varphi}}} = v$. Trivial
valuations can be written as $v = \ord_{\gamma}$, where $\gamma$ is any
constant arc for which $\gamma(0)$ is the home of $v$.

Among all arcs giving the same semi-valuation $v$, there is a distinguished
one, characterized by being maximal with respect to specialization in the arc
space. Namely, we consider the family
\[
    \mathcal C_v
    =
    \overline{\{\,
        \gamma \in J_\infty X
        \,\mid\,
        \ord_\gamma = v
    \,\}}.
\]
Then $\mathcal C_v$ is irreducible \cite{ELM04,Ish08,dFEI08,Mor09}, and its
generic point $\gamma_v$ verifies $\ord_{\gamma_v}=v$. Any other arc inducing
$v$ is a specialization of $\gamma_v$. Following the terminology of
\cite{Mor09}, we call $\mathcal C_v$ the \emph{maximal arc set} associated to
$v$. If $v$ is trivial, we have that $\mathcal C_v = J_\infty Y$, where $Y$ is
the home of $v$.

\subsection*{Divisorial valuations} Among all valuations on a variety $X$, the
\emph{divisorial} ones are of particular importance. Let $f \colon Y \to X$ be
a proper birational map with $Y$ smooth, and let $E$ be a prime divisor on $Y$.
Then computing orders of vanishing along $E$ gives a valuation on $X$, which we
denote $\val_E$. Any valuation of the form $q\cdot\val_E$, where $q$ is a
positive integer, is called a \emph{divisorial valuation} on $X$. The maximal
arc sets associated to divisorial valuations are called \emph{maximal
divisorial sets}.

One of the main results of \cite{ELM04} and \cite{dFEI08} is a characterization
of divisorial valuations among all semi-valuations using contact loci. In
precise terms, they prove that the following are equivalent:
\begin{enumerate}
    \item $v$ is a divisorial valuation;
    \item there exists a contact locus $\mathcal C$ such that $v =
        \ord_{\mathcal C}$; and
    \item there exists a constructible set $\mathcal C$ such that $v =
        \ord_{\mathcal C}$.
\end{enumerate}
Moreover, for a subset $\mathcal C \subset J_\infty X$, the following are also
equivalent:
\begin{enumerate}
    \item $\mathcal C$ is a maximal divisorial set; and
    \item $\mathcal C$ is a fat irreducible component of a contact locus.
\end{enumerate}

For a divisorial valuation $v = q \cdot \val_E$, the corresponding maximal
divisorial set $\mathcal C_v$ has an explicit geometric interpretation. It is
the closure of the set of arcs whose lift to $Y$ is tangent to $E$ with order
$\geq q$. In symbols:
\[
    \mathcal C_v =
    \overline{\{
        f(\gamma)
        \,\mid\,
        \gamma \in J_\infty Y,\,
        \ord_\gamma(E) \geq q
    \}}.
\]

\subsection*{Discrepancies and log canonical thresholds} The importance of arc
spaces from the point of view of the minimal model program resides in a formula
that computes discrepancies of divisorial valuations in terms of arcs. We
restrict ourselves to the smooth case, which will be enough for our purposes.
Let $X$ be a smooth variety, and consider a divisorial valuation $v = q_v \cdot
\val_E$, where $E$ is a prime divisor in some smooth birational model $f \colon
Y \to X$. Then the \emph{discrepancies} of $E$ and $v$ are defined as
\[
    k_{E}(X) = \ord_E(K_{Y/X}),
    \qquad\text{and}\qquad
    k_{v}(X) = q_v \cdot \ord_E(K_{Y/X}),
\]
where $K_{Y/X} \sim K_Y - f^*(K_X)$ is the relative canonical divisor. A
standard computation shows that $k_v(X)$ does not depend on the choice of model
$Y$. We have the following formula \cite{Mus01,ELM04,dFEI08}:
\begin{equation}
    \label{eq:discrepancy-formula}
    q_v + k_v(X) = \codim(\mathcal C_v, J_\infty X).
\end{equation}
Since we assume that $X$ is smooth, the codimension in the above formula can be
computed either in the sense of the Zariski topology of $J_\infty X$, or in the
sense of cylinders (as the codimension of a high enough truncation).

Using discrepancies we can define the \emph{log canonical threshold}, an
invariant of the singularities of a pair which is central in the minimal model
program. Let $Z \subset X$ be a subscheme, and consider a log resolution of the
pair $(X,Z)$. This consists of a proper birational map $f \colon Y \to X$ where
$Y$ is smooth, the scheme theoretic inverse image of $Z$ is a divisor $A$, and
$A + {\rm Ex}(f)$ is a divisor with simple normal crossings. Then the \emph{log
canonical threshold} of the pair $(X,Z)$ is defined as
\[
    \lct(X,Z) = \min_E \left\{ \frac{1 + k_E(X)}{\val_E(Z)} \right\}.
\]
In this formula $E$ ranges among the prime exceptional divisors of $f$. As
above, one can show that $\lct(X,Z)$ does not depend on the choice of log
resolution $Y$. Using arc spaces we can express the formula for the log
canonical threshold in the following way:
\begin{equation}
    \label{eq:lct-formula}
    \lct(X,Z) = \min_{\mathcal C}
    \left\{
        \frac{\codim(\mathcal C, J_\infty X)}{\ord_{\mathcal C}(Z)}
    \right\}.
\end{equation}
Here $\mathcal C$ ranges in principle among all maximal divisorial sets of
$J_\infty X$, but one can easily show that it is enough to consider only the
fat irreducible components of all the contact loci $\Cont^{\geq p}(Z)$.

For us it will sometimes be more convenient to deal with the \emph{Arnold
multiplicity}, which is just the reciprocal of the log canonical threshold:
\begin{equation}
    \label{eq:arnold-formula}
    \operatorname{Arnold-mult}(X,Z)
    = \max_E \left\{ \frac{\val_E(Z)}{1 + k_E(X)} \right\}
    = \max_{\mathcal C}
    \left\{
        \frac{\ord_{\mathcal C}(Z)}{\codim(\mathcal C, J_\infty X)}
    \right\}.
\end{equation}

\subsection*{Nash-type problems} Let $X$ be a variety, and denote by ${\rm
Sing}(X)$ its singular locus. The fat irreducible components of the contact
locus
\[
    \Cont^{\geq 1}({\rm Sing}(X)) \subset J_\infty X
\]
are called the \emph{Nash families} of arcs of $X$. From the above discussion,
we see that the Nash families are the maximal divisorial sets associated to
some divisorial valuations, which we call the \emph{Nash valuations} of $X$.
The Nash valuations, apart from being divisorial, are also \emph{essential}, in
the sense that they appear as irreducible components of the exceptional locus
of every resolution of singularities of $X$. This is what is known as the
\emph{Nash map}:
\[
    \{\, \text{Nash valuations of $X$} \,\}
    \subseteq
    \{\, \text{essential valuations of $X$} \,\}.
\]
The \emph{Nash problem}, in its more general form, asks for a geometric
characterization of the image of the Nash map. The \emph{Nash conjecture}
asserts that the Nash map is a bijection.

The Nash problem has a long history. The Nash conjecture turns out to be true
for curves, for surfaces \cite{FdBPP12,dFD15}, and for several special families
of singularities in higher dimensions, including toric varieties
\cite{IK03,Ish05,Ish06,GP07,PPP08,LJR12,LA11,LA16}. But there are
counterexamples to the Nash conjecture in all dimensions $\geq 3$
\cite{IK03,dF13,JK13}. For an approach to the Nash problem in higher dimensions
using the minimal model program, see \cite{dFD15}.

We will also be interested in a variant of the Nash problem that we call the
\emph{generalized Nash problem}. The above construction of the maximal arc set
associated to a semi-valuation can be thought as an inclusion:
\[
    \{\, \text{semi-valuations on $X$} \,\}
    \subseteq
    J_\infty X,
\]
where a semi-valuation $v$ gets sent to the generic arc in $\mathcal C_v$. This
endows the set of semi-valuations with a geometric structure. As the Nash
problem and the formula for discrepancies show, this structure is relevant from
the point of view of singularity theory. A basic question in this context is
the following: given two semi-valuations $v_1$ and $v_2$, determine whether
there is an inclusion $\mathcal C_{v_1} \supseteq \mathcal C_{v_2}$. This is
what we call the \emph{generalized Nash problem}.

We understand the generalized Nash problem for invariant valuations on toric
varieties \cite{Ish08} and on determinantal varieties \cite{Doc13}. But beyond
this, very little is known, even for valuations on the plane
\cite{Ish08,FdBPPPP}.

\section{The Grassmannian and its Schubert varieties}

\label{sec:grass}

In this section we discuss generalities about Grassmannians and Schubert
varieties. Our main purpose is to fix notation and recall basic results that
will be used in the rest of the paper. All results are well-known, and we
mostly enumerate them without proof. For details we refer the reader to any of
the standard texts in the subject, for example \cite[Chapter~II]{ACGH85},
\cite{BV88}, or \cite{Ful97}.

\subsection*{Grassmannians}

Fix integers $0 < k < n$. The \emph{Grassmannian} of $k$-planes in $\C^n$ is
denoted by $G(k,n)$. A point $V \in G(k,n)$ can be described as the row span of
a full-rank matrix with $k$ rows and $n$ columns:
\[
    V = 
    \operatorname{row\ span}
    \begin{pmatrix}
    v_{11} & v_{12} & \hdots & v_{1n} \\
    v_{21} & v_{22} & \hdots & v_{2n} \\
    \vdots & \vdots & \ddots & \vdots \\
    v_{k1} & v_{k2} & \hdots & v_{kn}
    \end{pmatrix},
    \qquad
    v_{ij} \in \C.
\]
Such a matrix is determined by $V$ only up to left multiplication by an element
of $\GL_k$. This way we obtain an identification of $G(k,n)$ with the GIT
quotient $\GL_k\bbslash {\rm Mat}_{k \times n}$.

The group $\GL_n$ has a natural right action on $G(k,n)$. This identifies the
Grassmannian with the quotient $P_{k,n} \backslash \GL_n$, where $P_{k,n}$ is
the parabolic subgroup of $\GL_n$ whose elements have zeros in the lower-left
block of size $(k) \times (n-k)$.

Three subgroups of $\GL_n$ will be featured prominently in the rest of the
paper. The first one is the torus $T = (\C^*)^n$, the subgroup of diagonal
matrices. The other two are the Borel subgroup $B = B^+$ and the opposite Borel
subgroup $B^-$, containing, respectively, the upper- an lower-triangular
matrices. Also relevant is the Weyl group $W = S_n$, the symmetric group on $n$
letters, naturally embedded in $\GL_n$ as the group of permutation matrices.

We denote by $\{e_1, \ldots, e_n\}$ the standard basis for $\C^n$. The
torus-fixed points of $G(k,n)$ are determined by the $k$-element subsets of
$\{e_1, \ldots, e_n\}$. More precisely, given a multi-index $I = [i_1 \dots
i_k]$, where $1 \leq i_1 < \dots < i_k \leq n$, we can consider the following
point in $G(k,n)$:
\[
    V_I = \langle e_{i_1}, \ldots, e_{i_k} \rangle.
\]
Then the $V_I$ are all the torus-fixed points in $G(k,n)$.

\subsection*{Schubert varieties}

The \emph{Schubert cells} are the Borel orbits in $G(k,n)$. The \emph{Schubert
varieties} are the closures of the Schubert cells. Each Borel orbit contains
exactly one torus-fixed point. For a multi-index $I = [i_1 \dots i_k]$, we
denote by $\Omega^\circ_I$ the Schubert cell containing $V_I$, and by
$\Omega_I$ the closure of $\Omega^\circ_I$. The Schubert cell $\Omega^\circ_{[1
\ldots k]}$ is called the \emph{big cell}.

Schubert varieties can be described more explicitly as follows. We consider the
flag
\[
    F_1 \subset F_2 \subset \dots \subset F_n = \C^n,
\]
where $F_i$ is spanned by the last $i$ vectors in the standard basis of $\C^n$.
Notice that the Borel subgroup $B$ is the stabilizer of $F_\bullet$. For a
multi-index $I = [i_1 \dots i_k]$, the Schubert variety associated to $I$ is
the subset of $G(k,n)$ given by
\[
    \Omega_I = 
    \{ 
        V \in G(k,n) 
    \mid
        \dim V \cap F_{n+1-i_s} \geq k+1-s, \,\, 1 \leq s \leq k
    \}.
\]

\subsection*{Bruhat order}

To a multi-index $I = [i_1 \dots i_k]$ we associate the partition $\lambda =
(\lambda_1 \dots \lambda_k)$ given by
\[
    i_s = s + \lambda_{k+1-s}.
\]
Notice that $n-k \geq \lambda_1 \geq \dots \geq \lambda_k \geq 0$. This
association induces a bijection between multi-indexes of length $k$ in the
range $\{1, \dots, n\}$, and partitions with at most $k$ parts of size at most
$n-k$. It is helpful to visualize partitions via the associated Ferrers-Young
diagrams; some examples in $G(3,6)$ are given in \cref{fig:young}.

\begin{figure}[ht]
    \centering
    \includegraphics{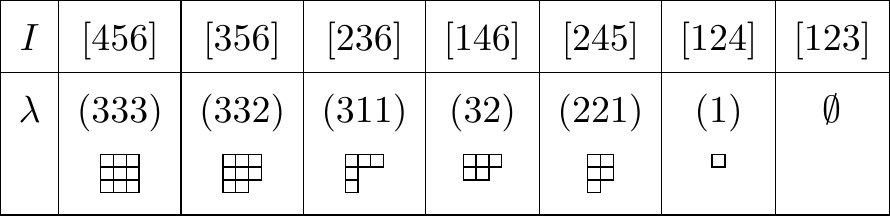}
    \caption{Some multi-indexes, partitions, and diagrams in $G(3,6)$.} 
    \label{fig:young}
\end{figure}

If $\lambda$ is the partition associated to a multi-index $I$, we also use the
notations $\Omega_\lambda = \Omega_I$ and $\Omega_\lambda^\circ =
\Omega_I^\circ$. Given two partitions $\lambda$ and $\mu$, the containment of
the Schubert varieties $\Omega_\lambda \subseteq \Omega_\mu$ corresponds to the
reversed containment of the (Ferrers-Young diagrams of the) partitions $\lambda
\supseteq \mu$. In terms of multi-indexes, given $I=[i_1 \dots i_k]$ and
$J=[j_1 \dots j_k]$, the containment  $\Omega_I \subseteq \Omega_J$ corresponds
to $i_s \geq j_s$; in this situation we write $I \geq J$. See \cref{fig:poset}
for an example: the first two diagrams show the poset of Schubert varieties in
$G(2,4)$ using multi-indexes and partitions.

\begin{figure}[ht]
    \centering
    \includegraphics{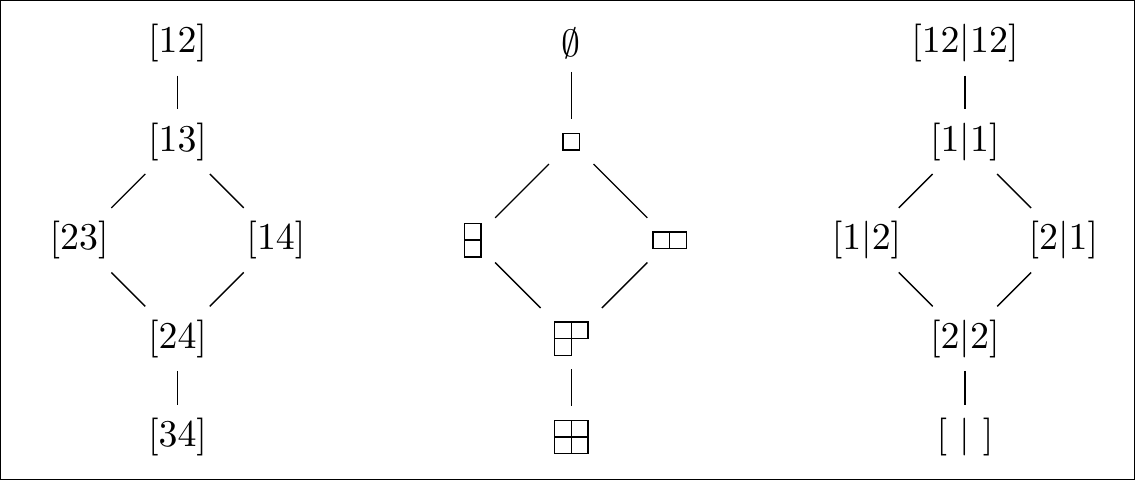}
    \caption{The poset of Schubert varieties in $G(2,4)$.} 
    \label{fig:poset}
\end{figure}

The codimension of $\Omega_\lambda$ in $G(k,n)$ is $|\lambda| = \lambda_1 +
\dots + \lambda_k$, that is, the number of boxes in the diagram of $\lambda$.
Moreover, $\Omega_\lambda^\circ$ is isomorphic to $\mathbb
A^{k(n-k)-|\lambda|}$.

\subsection*{Plücker coordinates}

We consider a matrix
\[
    X = 
    \begin{pmatrix}
        X_{1 1} & \cdots & X_{1 n} \\
        \vdots  & \ddots & \vdots  \\
        X_{k 1} & \cdots & X_{k n}
    \end{pmatrix}
\]
where the entries are indeterminates. The $k \times k$ minors of this matrix
are called the \emph{Plücker coordinates} of $G(k,n)$. Given a tuple of indexes
$I = [i_1 \dots i_k]$ (not necessarily distinct or in increasing order), the
minor determined by the columns in $I$ will also be denoted by $[i_1 \dots
i_k]$. If $I = [i_1 \dots i_k]$ is a multi-index, then $[i_1 \dots i_k]$ is a
Plücker coordinate. This abuse of notation (using the same symbols to denote a
multi-index and a Plücker coordinate) will not cause problems.

Given a $k$-plane $V \subset \C^n$, we obtain a line $\largewedge^k V \subset
\largewedge^k \C^n$. This induces the Plücker embedding $G(k,n) \hookrightarrow
\mathbb P(\largewedge^k\C^n)$. The homogeneous coordinate ring of $G(k,n)$
corresponding to the Plücker embedding will be denoted by $\C[G(k,n)]$; it is
isomorphic to the subring of the polynomial ring $\C[X_{ij}]$ generated by the
Plücker coordinates. 

For a multi-index $I$ with associated partition $\lambda$, we denote by
$\mathcal I_I = \mathcal I_\lambda$ the ideal of $\Omega_I$ in $G(k,n)$. We
think of $\mathcal I_I$ as an ideal in $\C[G(k,n)]$. The following result is
classic%
\footnote{%
    The proof can be found in many places, for example in
    \cite[Theorem~1.4]{BV88}. But notice that the notation in \cite{BV88} for
    Schubert varieties differs from ours. What they denote $\Omega(a_1, \ldots,
    a_k)$ corresponds to our $\Omega_I$, where $I = [i_1 \ldots i_k]$ is given
    by $i_s = n+1-a_{k+1-s}$.
}.

\begin{theorem}
    \label{ideal-gens}
    Let $[i_1 \dots i_k]$ be a Plücker coordinate. Then the Plücker coordinates
    $[j_1 \dots j_k]$ such that $[j_1 \dots j_k] \ngeq [i_1 \dots i_k]$
    generate the ideal $\mathcal I_{[i_1 \dots i_k]}$.
\end{theorem}

In particular, $\Omega_\boxempty$ is the divisor with equation $[1 \ldots k]$,
the determinant of the first $k$ columns of $X$. The big cell
$\Omega^\circ_\emptyset$ is given by the non-vanishing of $[1 \ldots k]$.

\subsection*{Plücker relations}

For our analysis of the arc space of $G(k,n)$ we will need some understanding
of the structure of products of ideals of Schubert varieties. In this study,
the Plücker relations play an important role. For our purposes it will be
enough to consider the following special case. For a proof, we refer the reader
to \cite[Lemma~4.4]{BV88}.

\begin{theorem}
    \label{plucker}
    Consider tuples of indexes $[i_1 \dots i_k]$ and $[j_1 \dots j_k]$, and let
    $u$ be an integer such that $1 \leq u \leq k$. Then
    \[
        [i_1 \dots i_k] \cdot [j_1 \dots j_k]
        =
        \sum_{v=1}^k \pm \,
        [i_1 \dots i_{u-1} j_v i_{u+1} \dots i_k] \cdot
        [j_1 \dots j_{v-1} i_u j_{v+1} \dots j_k].
    \]
\end{theorem}

\subsection*{The opposite big cell}

Using the opposite Borel $B^-$, instead of $B$, we define \emph{opposite
Schubert cells} and \emph{opposite Schubert varieties}. We denote them with an
inverted circumflex, like $\check\Omega^\circ_I$ and $\check\Omega_\lambda$.

We are mainly interested in the \emph{opposite big cell}, which we will denote
by $\mathcal U = \check\Omega^\circ_{[n-k+1 \ldots n]}$. It is given by the
non-vanishing of the Plücker coordinate $[n-k+1 \ldots n]$, the determinant of
the last $k$ columns of $X$. A point in $\mathcal U$ is uniquely represented by
a matrix of the form $(X_{\mathcal U} | \Delta')$, where:
\[
    X_{\mathcal U} =
    \begin{pmatrix}
        x_{11} & x_{12} & \cdots & x_{1(n-k)} \\
        x_{21} & x_{22} & \cdots & x_{2(n-k)} \\
        \vdots & \vdots & \ddots & \vdots     \\
        x_{k1} & x_{k2} & \cdots & x_{k(n-k)} \\
    \end{pmatrix}
    \quad\text{and}\qquad
    \Delta' = 
    \begin{pmatrix}
        0      & \cdots  & 0      & 1 \\
        0      & \cdots  & 1      & 0 \\
        \vdots & \iddots & \vdots & \vdots \\
        1      & \cdots  & 0      & 0 \\
    \end{pmatrix}.
\]

We think of the entries $x_{ij}$ of $X_{\mathcal U}$ as variables. The
corresponding polynomial ring is the coordinate ring of $\mathcal U \simeq
\mathbb A^{k(n-k)}$, and will be denoted by $\C[\mathcal U] = \C[x_{ij}]$.

$\mathcal U$ and its Weyl translates form a natural system of affine charts for
the projective variety $G(k,n)$, and there is a natural ``de-homogenization
process'' from $\C[G(k,n)]$ to $\C[\mathcal U]$. Explicitly, this consists in
substituting the matrix $X$ with the matrix $(X_{\mathcal U} | \Delta')$.

We use the notation $[i_1 \ldots i_r | j_1 \ldots j_r ]$ for the minor of
$X_{\mathcal U}$ corresponding to the rows $i_s$ and the columns $j_s$. By
convention we set $[~|~] = 1$. Via the ``de-homogenizing'' substitution
mentioned above (and ignoring signs), we obtain a bijection between Plücker
coordinates $[i_1 \ldots i_k]$ and minors $[i_1 \ldots i_r | j_1 \ldots j_r ]$
(of arbitrary size) of $X_{\mathcal U}$. Some examples of this process are
shown in \cref{fig:de-hom}.

\begin{figure}[ht]
    \centering
    \includegraphics{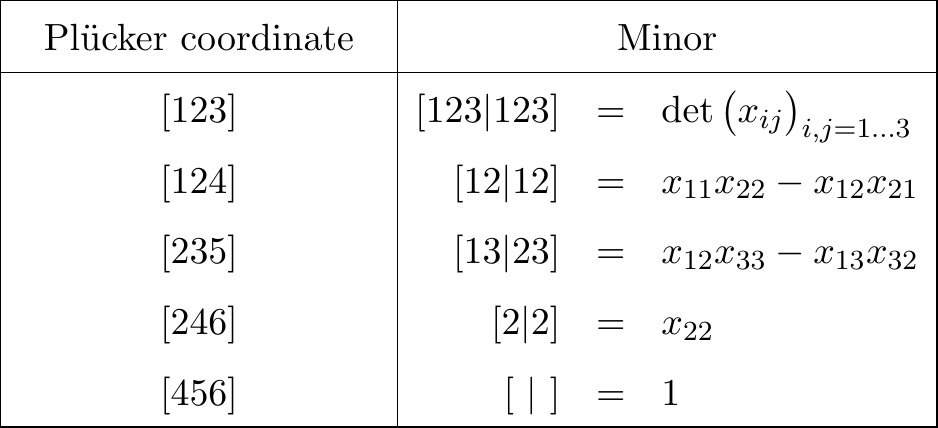}
    \caption{Examples of de-homogenizations of Plücker coordinates in $G(3,6)$.
    Signs have been ignored.}
    \label{fig:de-hom}
\end{figure}

Using the bijection with Plücker coordinates we endow the set of minors with an
order. If $M_1, M_2$ are minors, with corresponding Plücker coordinates $I_1,
I_2$, and corresponding partitions $\lambda_1, \lambda_2$, we have:
\[
    M_1 \leq M_2
    \qquad\Leftrightarrow\qquad
    I_1 \leq I_2
    \qquad\Leftrightarrow\qquad
    \lambda_1 \subseteq \lambda_2.
\]
Explicitly, if $M_1 = [i_1 \ldots i_r | j_1 \ldots j_r]$ and 
$M_2 = [a_1 \ldots a_s | b_1 \ldots b_s]$, then:
\begin{equation}
    \label{eq:minor-order}
    M_1 \leq M_2
    \qquad\Leftrightarrow\qquad
    \begin{array}{l}
        r \geq s \\
        i_u \leq a_u \\
        j_u \leq b_u \\
    \end{array}
    \quad\text{for $1\leq u \leq s$}.
\end{equation}
See the last diagram in \cref{fig:poset} for the example of $G(2,4)$.

All Schubert varieties intersect the opposite big cell, and the ideal of this
intersection can be determined using \cref{ideal-gens}. Let $I$ be a
multi-index, with corresponding minor $M$. Then the ideal of $\Omega_I \cap
\mathcal U$ in the ring $\C[\mathcal U] = \C[x_{ij}]$ is generated by the
minors $M_1$ such that $M_1 \ngeq M$.

\subsection*{Single Schubert conditions}

We denote by $(b^a) = (b \,\overset{a}\dots\, b)$ the rectangular partition
with $a$ rows and $b$ columns. A \emph{single-condition Schubert variety} is a
Schubert variety $\Omega_\lambda$ whose associated partition $\lambda$ has
rectangular shape. In such case, if $\lambda = (b^a)$, we have
\[
    \Omega_{(b^a)} = 
    \{ V \in G(k,n) \mid \dim V \cap F_{n-k+a-b} \geq a \}.
\]

Given an arbitrary partition $\lambda$, we define the \emph{Schubert
conditions} of $\lambda$ to be the maximal rectangular partitions contained in
$\lambda$. This definition is perhaps best illustrated by examples: see
\cref{fig:schubert-conditions}.

\begin{figure}[ht]
    \centering
    \includegraphics{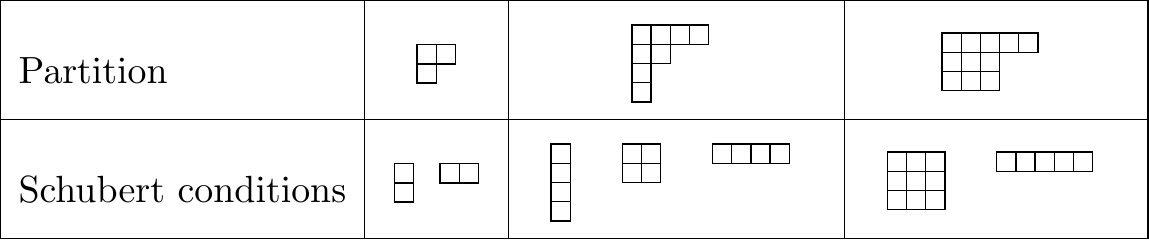}
    \caption{Schubert conditions}
    \label{fig:schubert-conditions}
\end{figure}

\begin{proposition}
    \label{schubert-ideals}
    Let $\lambda$ be a partition, and let $\mu_1, \dots, \mu_r$ be the Schubert
    conditions of $\lambda$. Then
    \[
        \Omega_\lambda = \Omega_{\mu_1} \cap \dots \cap \Omega_{\mu_r}
        \qquad\text{and}\qquad
        \mathcal I_\lambda = \mathcal I_{\mu_1} + \dots + \mathcal I_{\mu_r}.
    \]
\end{proposition}

\begin{proof}
This is an immediate consequence of \cref{ideal-gens}. Let $[j_1 \dots j_k]$ be
a multi-index with associated partition $\mu$. Then $[j_1 \dots j_k]$ is a
generator of $\mathcal I_\lambda$ if and only if $\mu \not\supseteq \lambda$.
But $\mu \not\supseteq \lambda$ if and only if $\mu \not\supseteq \mu_s$ for
some $1\leq s \leq r$, and the result follows.
\end{proof}

The ideals defining single-condition Schubert varieties have a particularly
simple structure, especially in the opposite big cell. Given $1 \leq a \leq k$
and $1 \leq b \leq n-k$, we denote by $M_{a,b}$ the minor with the biggest size
having first row $a$ and first column $b$. It is easy to see that $M_{a,b} = [a
\ldots a+r|b \ldots b+r]$, where $r = \min\{k-a, n-k-b\}$.  We will refer to
the minors obtained in this way as \emph{final minors}.

Given a final minor $M_{a,b}$, we consider the corresponding multi-index
$I_{a,b}$ and partition $\lambda_{a,b}$. They are given by
\begin{equation}
    \label{eq:final-plucker}
    I_{a,b} = \big[ ~ (b) \ldots (b+k-a) ~ (n-a+2) \ldots (n) ~ \big]
\end{equation}
and
\[
    \lambda_{a,b} = \big(~ (n-k) \ldots (n-k) ~ (b-1) \ldots (b-1) ~\big),
\]
where $(n-k)$ shows up $(a-1)$ times in $\lambda_{a,b}$, and $(b-1)$ shows up
$(k-a+1)$ times. More visually, $\lambda_{a,b}$ is the biggest partition not
containing the box in position $(a,b)$.

\begin{proposition}
    \label{gens-rect}
    Let $\lambda = (b^a)$ be a rectangular partition, and consider $I_{a,b}$
    and $M_{a,b}$ as above. Let $r = \min\{k-a,n-k-b\}$, so that the size of
    $M_{a,b}$ is $r+1$. Then:
    \begin{enumerate}
        \item The ideal $\mathcal I_\lambda \subset \C[G(k,n)]$ is generated by
        the Plücker coordinates $J$ that verify $J \leq I_{a,b}$.

        \item In the opposite big cell $\mathcal U$, the ideal of
        $\Omega_{\lambda}\cap\mathcal U$ is generated by the minors $M$ that
        verify $M \leq M_{a,b}$. Moreover, it is enough to consider only minors
        $M$ of size $r+1$.

        More precisely, if $(k-a)\leq(n-k-b)$ the generators are the minors of
        size $r+1$ in the first $b+r$ columns of $X_{\mathcal U}$. If $(k-a)
        \geq (n-k-b)$ the generators are the minors of size $r+1$ in the first
        $a+r$ rows of $X_{\mathcal U}$.
    \end{enumerate}
\end{proposition}

\begin{proof}
Let $I$ be the Plücker coordinate corresponding to $\lambda$. By
\cref{ideal-gens}, the ideal $\mathcal I_\lambda$ is generated by the Plücker
coordinates $J$ such that $J \ngeq I$. This is equivalent to $\mu \nsupseteq
\lambda$, where $\mu$ is the partition associated to $J$. Since $\lambda$ is
rectangular, this happens precisely when $\mu$ does not contain the box
$(a,b)$. By the discussion preceding the proposition, this is equivalent to
$\mu \subseteq \lambda_{a,b}$, and the first part follows. The second part is
an immediate consequence of the first part and the definition of the order
among minors (\cref{eq:minor-order}).
\end{proof}

\section{A decomposition of the arc space of the Grassmannian}

\label{sec:strata}

In this section we give a decomposition of the arc space of the Grassmannian
that resembles the Schubert cell decomposition of the Grassmannian itself. We
will call the pieces of this decomposition contact strata. Just as Schubert
cells are indexed by partitions, contact strata are indexed by plane
partitions.

Recall that we write $J_\infty G(k,n)$ and $J_m G(k,n)$ for the arc space and
jet schemes of $G(k,n)$. The universal property of the
Grassmannian~\cite[\S9.7]{GD71} tells us what the $\C$-valued points of
$J_\infty G(k,n)$ are: they correspond to $\psr$-submodules $\Lambda \subset
\psr^n$ for which the corresponding quotient $\psr^n/\Lambda$ is free of rank
$n-k$. Each such $\Lambda$ is itself free, and of rank $k$, so it can be
represented by a $k \times n$ matrix with coefficients in $\psr$:
\[
    \Lambda
    =
    \operatorname{row\ span}
    \begin{pmatrix}
        x_{11}(t) & x_{12}(t) & \hdots & x_{1n}(t) \\
        x_{21}(t) & x_{22}(t) & \hdots & x_{2n}(t) \\
        \vdots & \vdots & \ddots & \vdots \\
        x_{k1}(t) & x_{k2}(t) & \hdots & x_{kn}(t)
    \end{pmatrix},
    \qquad
    x_{ij}(t) \in \psr.
\]
The condition on the freeness of the quotient $\psr^n/\Lambda$ simply says that
one of the maximal minors of this matrix is a unit in $\psr$. Notice that the
matrix is only determined by $\Lambda$ up to multiplication on the right by an
element of $\GL_k(\psr)$. Despite this, we will often use the same symbol
$\Lambda$ to also denote any matrix representing $\Lambda$.

There is an analogous description for $\C$-valued points of the jet schemes
$J_m G(k,n)$.

\begin{definition}
    \label{cont-prof}
    Let $\Lambda$ be an arc in $J_\infty G(k,n)$. The collection
    \[
        \{ \ord_\Lambda(\Omega_\lambda) \}_{\lambda},
    \]
    where $\lambda$ ranges among all partitions with $\leq k$ parts of size
    $\leq n-k$, is called the \emph{contact profile} of $\Lambda$ (with respect
    to the Schubert varieties). Given a collection
    \[
        \alpha = \{ \alpha_\lambda \}_{\lambda},
    \]
    where $\alpha_\lambda \in [0,\infty]$ and $\lambda$ ranges as above, the
    \emph{contact stratum} of $J_\infty G(k,n)$ associated to $\alpha$ is the
    set of arcs that have contact profile $\alpha$. We define analogously
    contact profiles for jets in $J_m G(k,n)$, and contact strata in $J_m
    G(k,n)$.
\end{definition}

Not all collections of numbers $\{ \alpha_\lambda \}_\lambda$ appear as contact
profiles of arcs and jets. The following results address the issue of
enumerating all possible contact profiles. As we will see, contact profiles are
in bijection with certain plane partitions.

\begin{proposition}
    \label{ess-cont-prof}
    Let $\{\alpha_\lambda\}_\lambda$ be the contact profile of an arc or a jet
    in $G(k,n)$. Fix a partition $\lambda$, and let $\mu_1, \ldots, \mu_r$ be
    the Schubert conditions of $\lambda$ (as in \cref{schubert-ideals}). Then
    \[
        \alpha_\lambda = \min\{\alpha_{\mu_1}, \ldots, \alpha_{\mu_r}\}.
    \]
\end{proposition}

\begin{proof}
This follows immediately from \cref{schubert-ideals}. 
\end{proof}

\begin{definition}
    \label{def:ess-cont-prof}
    For positive integers $i$ and $j$, recall that $(j^i)$ denotes the
    rectangular partition with $i$ rows and $j$ columns. If $\Lambda$ is an arc
    or a jet in $G(k,n)$, the matrix $\alpha$ of size $k \times (n-k)$ with
    entries
    \[
        \alpha_{i,j} = \ord_\Lambda(\Omega_{(j^i)})
    \]
    is called the \emph{essential contact profile} of $\Lambda$. Notice that
    \cref{ess-cont-prof} guarantees that the essential contact profile determines
    the contact profile.
\end{definition}

\begin{proposition}
    \label{A-restr}
    Let $\alpha=(\alpha_{i,j})$ be the essential contact profile of an arc
    $G(k,n)$. Then
    \begin{align*}
        & \alpha_{i,j} \geq \alpha_{i',j'}
        \quad \text{if } i \leq i' \text{ and } j \leq j'
    \\
        & \alpha_{i,j} + \alpha_{i+2,j+1}
        \geq \alpha_{i+1,j} + \alpha_{i+1,j+1}
    \\
        & \alpha_{i,j} + \alpha_{i+1,j+2}
        \geq \alpha_{i,j+1} + \alpha_{i+1,j+1}
    \end{align*}
    The same statements are true for essential contact profiles of $m$-jets,
    provided one replaces sums $x+y$ with $\min\{x+y, m+1\}$. 
\end{proposition}

\begin{proof}
The first inequality follows from the inclusions $\Omega_{(j^i)} \supseteq
\Omega_{(j'^{i'})}$. We will prove the second inequality, the third one
following from similar arguments.

Let $\Lambda$ be an arc or a jet in $G(k,n)$ with essential contact profile
$(\alpha_{i,j})$. The Borel group $B$ has a right action on the arc space and
jet schemes, and all the elements in the orbit $\Lambda \cdot B$ have the same
contact profile.

Let $\mathcal I(i,j) = \mathcal I_{(j^i)}$ be the ideal of $\Omega_{(j^i)}$,
and let $I_{i,j}$ be distinguished Plücker coordinate of $\mathcal I(i,j)$, as
it appears in \cref{gens-rect}. The other generators $J$ of $\mathcal I(i,j)$
verify $J \leq I_{i,j}$, that is, the columns that appear in $J$ are to the
left of the columns that appear in $I_{i,j}$. Notice that $B$ acts by column
operations, in such a way that columns on the left affect columns on the right.
Therefore, after replacing $\Lambda$ by a generic $B$-translate, we can assume
that $\alpha_{i,j} = \ord_\Lambda (I_{i,j})$ for all $i,j$. The
\lcnamecref{A-restr} will be proven if we show that
\begin{equation}
    \label{eq:A-restr-goal}
    I_{i,j} \cdot I_{i+2,j+1}
    \in
    \mathcal I(i+1,j)
    \,
    \mathcal I(i+1,j+1).
\end{equation}

From \cref{eq:final-plucker} we see that:
\[
    \begin{array}{lccccccccccc}
        I_{i,j} & = &
        [\,\, j & j+1 & \dots & c-1 & c & \phantom{d} & \phantom{d+1} & d+2 & \dots & n \,\,],
    \\[.5em]
        I_{i+2,j+1} & = &
        [\,\, \phantom{j} & j+1 & \dots & c-1 & \phantom{c} & d & d+1 & d+2 & \dots & n \,\,],
    \\[.5em]
        I_{i+1,j} & = &
        [\,\, j & j+1 & \dots & c-1 & \phantom{c} & \phantom{d} & d+1 & d+2 & \dots & n \,\,],
    \\[.5em]
        I_{i+1,j+1} & = &
        [\,\, \phantom{j} & j+1 & \dots & c-1 & c & \phantom{d} & d+1 & d+2 & \dots & n \,\,],
    \end{array}
\]
where $c=j+k-i$ and $d=n-i$. We apply \cref{plucker} to the product $I_{i,j}
\cdot I_{i+2,j+1}$ using $u=k-i+1$. We obtain an expansion 
\[
    I_{i,j} \cdot I_{i+2,j+1}
    =
    \sum_v \pm\, I^v \cdot J^v,
\]
where $I^v$ and $J^v$ are obtained from $I_{i,j}$ and $I_{i+2,j+1}$ by swapping
the entry with value $c$ in $I_{i,j}$ with the $v$-th entry in $I_{i+2,j+1}$.
After the swap in most cases $I^v$ is zero (it has a repeated entry). The only
exceptions are $v=k-i$ and $v=k-i+1$, that is, when $c$ gets swapped with
either $d$ or $d+1$. In both cases it is easy to see that
\[
    I^v \leq I_{i+1,j}
    \qquad\text{and}\qquad
    J^v \leq I_{i+1,j+1}.
\]
Now \cref{eq:A-restr-goal} follows from \cref{gens-rect}, and the
\lcnamecref{A-restr} is proven.
\end{proof}

\begin{corollary}
    \label{ess-to-full}
    Let $\Lambda$ be an arc or a jet in $G(k,n)$, and let
    $\{\alpha_\lambda\}_\lambda$ and $(\alpha_{i,j})$ be the associated contact
    profile and essential contact profile. Then:
    \[
        \alpha_\lambda = \min\{ \alpha_{i,j} \mid (i,j) \in \lambda \}.
    \]
\end{corollary}

\begin{definition}
    \label{def:if-prof}
    Let $\Lambda$ be an arc in $G(k,n)$, and let $(\alpha_{i,j})$ be its
    essential contact profile. Consider the matrix $\beta$ of size $k \times
    (n-k)$ with entries
    \[
        \beta_{i,j} = \alpha_{i,j}-\alpha_{i+1,j+1},
    \]
    where we set $\alpha_{i',j'}=0$ when $i'>k$ or $j'>n-k$, and use the
    convention $\infty - x = \infty$. Then $\beta$ is called the
    \emph{invariant factor profile} of $\Lambda$. This choice of terminology
    will be justified in \cref{sec:schubert-cond}. Notice that we only define
    invariant factor profiles for arcs, not jets. For an extension of this
    notion to jets, see \cref{if-jets}.
\end{definition}

\begin{remark}[Plane partitions]
Recall that a \emph{plane partition} (sometimes also called a \emph{3d
partition}) is a matrix of non-negative integers whose entries are
non-increasing along each column and along each row. We slightly generalize
this notion and \emph{allow entries to be infinite}. Two plane partitions are
identified when they have the same non-zero entries, and therefore they are
often written by omitting the zero entries. The collection of non-zero entries
in a plane partition $\beta$ determines a (linear) partition $\lambda$, called
the \emph{base} (or \emph{shape}) of the plane partition, and gives the
\emph{number of columns} and the \emph{number of rows} of $\beta$. The biggest
entry of $\beta$ (the one in position $(1,1)$) is called the \emph{height} of
$\beta$. The \emph{number of boxes} in $\beta$ (or the \emph{sum} of $\beta$,
or the \emph{volume} of $\beta$) is the sum of the entries of $\beta$. A plane
partition $\beta = (\beta_{i,j})$ is often visualized via its
\emph{Ferrers-Young diagram}, a collection of boxes in space, with a pillar of
height $\beta_{i,j}$ on top of the square in the plane in position $(i,j)$. For
some examples see \cref{fig:plane-partition}.
\end{remark}

\begin{figure}
    \centering
    \includegraphics{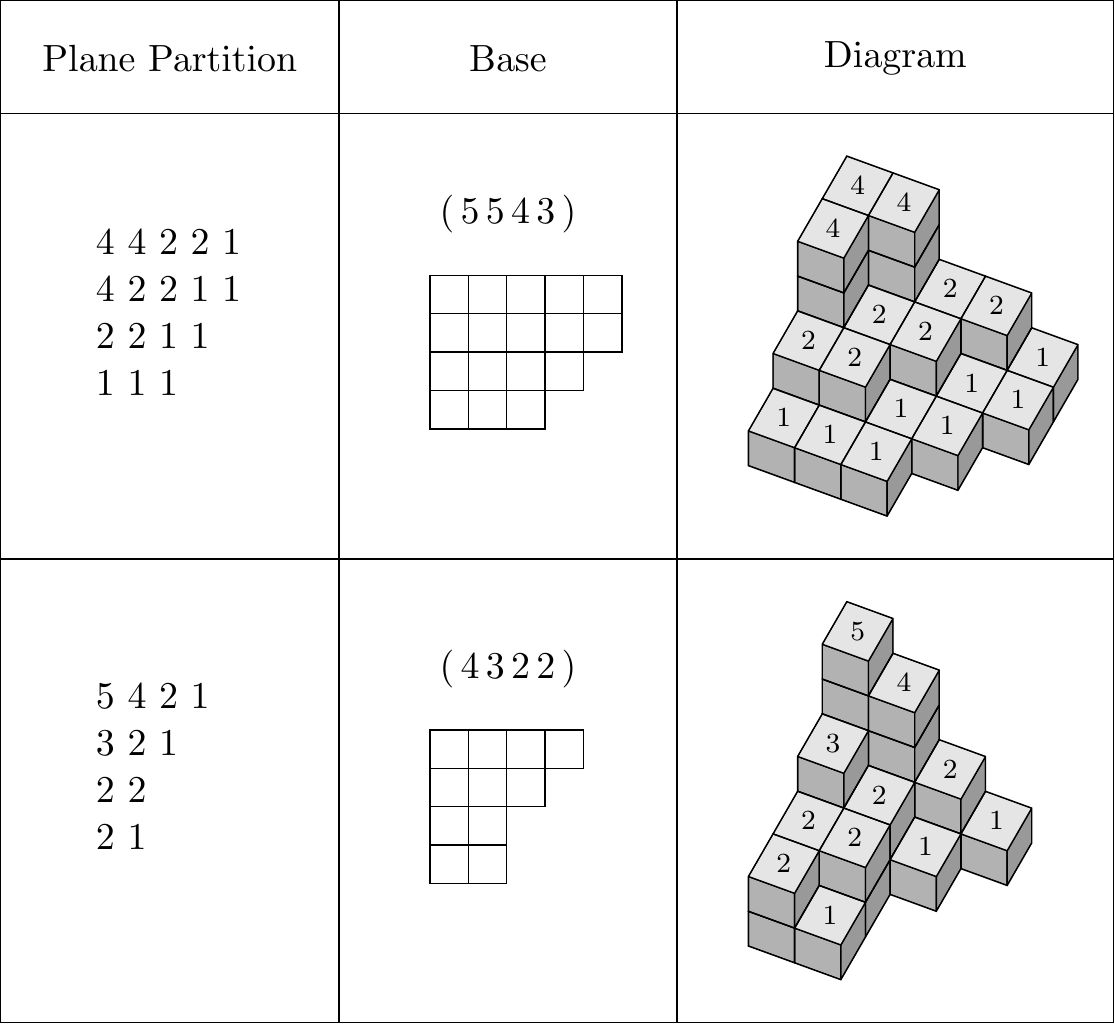}
    \caption{Some examples of plane partitions and their diagrams.}
    \label{fig:plane-partition}
\end{figure}

\begin{theorem}
    \label{if-prof}
    Let $\beta = (\beta_{i,j})$ be the invariant factor profile of an arc in
    $G(k,n)$. Then $\beta$ is a plane partition (possibly with infinite
    height), i.e.,
    \begin{align*}
        & \beta_{i,j} \geq 0,
        \\& \beta_{i,j} \geq \beta_{i+1,j},
        \\& \beta_{i,j} \geq \beta_{i,j+1}.
    \end{align*}
    Conversely, any plane partition with base contained in the rectangle of
    size $k \times (n-k)$ is the invariant factor profile of some arc in
    $G(k,n)$.
\end{theorem}

The first part of the previous \lcnamecref{if-prof} follows immediately from
\cref{A-restr}. The ``converse'' part, the fact that all plane partitions give
rise to a non-empty contact stratum is harder to prove, and requires some
preparation. The proof appears in \cref{sec:planar-networks} as a consequence
of \cref{c-works}.

\begin{notation}
    Given a plane partition $\beta$ (possibly with infinite height), we denote
    by $\mathcal C_\beta$ the contact stratum in $J_\infty G(k,n)$ whose arcs
    have invariant factor profile equal to $\beta$. The previous theorem
    guarantees that all contact strata in $J_\infty G(k,n)$ are of the form
    $\mathcal C_\beta$ for some $\beta$, and that $\mathcal C_\beta$ is
    non-empty precisely when $\beta$ has its base contained in the rectangle of
    size $k \times (n-k)$.
\end{notation}

\section{Contact strata and Schubert conditions}

\label{sec:schubert-cond}

Before finishing the proof of \cref{if-prof}, we discuss another interpretation
for the numbers that appear in the invariant factor profile of an arc. This
interpretation justifies our terminology.

The idea is to study ``Schubert conditions for lattices''. We start by
recalling how to do this for Schubert cells. As before, we consider the flag
$F_\bullet$ given by
\[
    0 = F_0 \subset F_1 \subset F_2 \subset \dots \subset F_n = \C^n,
\]
where $F_i$ is spanned by the last $i$ vectors in the standard basis of $\C^n$.
Then, given a point $V \in G(k,n)$, the numbers
\[
    \rho_s = \dim_{\C} ( V \cap F_s )
\]
determine, and are determined by, the Schubert cell $\Omega^\circ_I$ to which
$V$ belongs. Indeed, the $\rho_s$ are clearly invariant under $B$-translates,
so we can assume that $V=V_I = \langle e_{i_1}, \ldots, e_{i_k} \rangle$. In
this case it is easy to see that $k-\rho_s = \dim_\C (V_I+F_s/F_s)$ is the
number of entries in $I$ less than or equal to $n-s$.

We would like to characterize contact strata $\mathcal C_\beta$ in an analogous
way. But in order to do this, the above description using intersections $V\cap
F_\bullet$ is not convenient. It is better to consider the quotients $\C^n/(V +
F_\bullet)$. More precisely, for a point $V \in \Omega^\circ_I$ we consider the
partition $\mu = (\mu_0 \ldots \mu_n)$ given by
\begin{equation}
    \label{eq:schubert-mu}
    \dim_{\C}
    \left(
        \frac{\C^n}{V + F_s}
    \right)
    = \mu_s
\end{equation}
for $0 \leq s \leq n$. Then $n-s-\mu_s = \dim_{\C}(V+F_s/F_s) = k-\rho_s$,
and therefore the $\mu_s$ are determined by $I$ (and vice-versa). A direct
computation shows that $\mu_s$ can be computed more visually in the following
way. We consider the diagram of the partition $\lambda$ associated to $I$.
Above this diagram we place, upside-down, the diagram of $(n-k \ldots 1)$.
Then the entries in $\mu$ count the number of boxes in the diagonals of the
resulting arrangement of boxes. For examples of this computation see
\cref{fig:comp-mu}.

\begin{figure}[ht]
    \centering
    \includegraphics{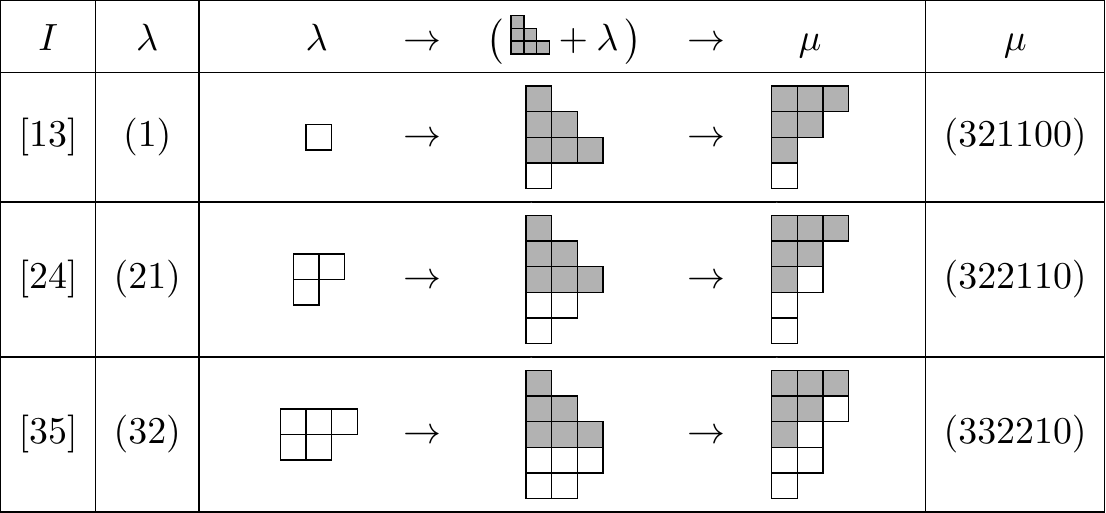}
    \caption{Computing $\mu$ in $G(2,5)$.}
    \label{fig:comp-mu}
\end{figure}

We interpret the above equations on dimensions as statements about the
isomorphism type of vector spaces. The elements of the arc space $J_\infty
G(k,n)$ are lattices, and the dimension (or rank) is no longer a complete
invariant of the isomorphism type of a $\psr$-module. Instead, using the
structure theory for finitely generated modules over a PID, we know that every
finitely generated $\psr$-module $\Gamma$ has a unique expression of the form
\[
    \Gamma \simeq
    \frac{\psr}{t^{\lambda_1}}
    \oplus
    \frac{\psr}{t^{\lambda_2}}
    \oplus
    \cdots
    \oplus
    \frac{\psr}{t^{\lambda_m}}
\]
where $\lambda_i \in \mathbb Z_{\geq 0} \cup \{\infty\}$ and $\lambda_1 \geq
\lambda_2 \geq \cdots \geq \lambda_m$. As usual, we are using the convention
$t^\infty = 0$, so the free rank of $\Gamma$ is the number of infinite terms
among the $\lambda_i$. The numbers $\lambda_i$ are called the \emph{invariant
factors} of $\Gamma$.

With a slight abuse of notation, we also denote by $F_\bullet$ the flag
\[
    0 = F_0 \subset F_1 \subset F_2 \subset \dots \subset F_n = \psr^n,
\]
where $F_i$ is the $\psr$-span of the last $i$ vectors in the standard
basis of $\psr^n$.

\begin{theorem}
    \label{schubert-cond-lattices}
    Let $\beta = (\beta_{i,j})$ be a plane partition (possibly with infinite
    height) whose base is contained in the rectangle of size $k \times (n-k)$.
    Extend $\beta$ by setting $\beta_{i,j} = \infty$ for $i<1$ and
    $\beta_{i,j}=0$ for $i>k$. Let $\Lambda$ be an arc in $G(k,n)$, thought as
    a $\psr$-submodule of $\psr^n$. Then $\Lambda$ belongs to the contact
    stratum $\mathcal C_\beta$ if and only if the quotient module
    \[
        \frac{\psr^n}{\Lambda + F_i}
    \]
    has invariant factors
    \[
        \beta_{i+k-n+1,1} \quad
        \beta_{i+k-n+2,2} \quad
        \cdots \quad
        \beta_{i-1,n-k-1} \quad
        \beta_{i,n-k}
    \]
    for $1 \leq i \leq n-1$.
\end{theorem}

See \cref{fig:schubert-cond-lattices} for a diagram explaining the content of
\cref{schubert-cond-lattices} in the particular case of $G(2,5)$. Notice how it
is a natural generalization of the description of Schubert cells given in
\cref{eq:schubert-mu}.

\begin{proof}
The action of the Borel group $B$ on $\C^n$ naturally induces an action on
$\psr^n$, and the flag $F_{\bullet}$ is fixed by $B$. In particular, for any
element $b \in B$ we obtain isomorphisms
\[
    \frac{\psr^n}{\Lambda + F_i}
    \simeq
    \frac{\psr^n}{(\Lambda \cdot b) + F_i}
\]
for all $i$. Also, both $\Lambda$ and $\Lambda \cdot b$ have the same contact
profile, and therefore the same invariant factor profile. Because of these
facts, in order to prove the \lcnamecref{schubert-cond-lattices} we are free to
replace $\Lambda$ with any of its $B$-translates.

Replace $\Lambda$ with a generic $B$-translate. This implies that $\Lambda$ is
contained in the opposite big cell, that is, it can be written as the row span
of a matrix $\Lambda_0$ of the form
\[
    \Lambda_0 = 
    \begin{pmatrix}
        x_{11} & x_{12} & \cdots & x_{1(n-k)} & 0      & \cdots  & 0      & 1 \\
        x_{21} & x_{22} & \cdots & x_{2(n-k)} & 0      & \cdots  & 1      & 0 \\
        \vdots & \vdots & \ddots & \vdots     & \vdots & \iddots & \vdots & \vdots \\
        x_{k1} & x_{k2} & \cdots & x_{k(n-k)} & 1      & \cdots  & 0      & 0 \\
    \end{pmatrix},
    \qquad
    x_{ij} \in \psr.
\]

We let $\Lambda_i$ be the matrix of size $k\times(n-i)$ obtained by removing
the last $i$ columns of $\Lambda_0$. After identifying $\psr^n/F_i$ with
$\psr^{n-i}$, the submodule $\Lambda+F_i/F_i$ can be obtained as the row span
of $\Lambda_i$. In particular, the invariant factors of $\psr^n/\Lambda+F_i$
can be computed by looking at ideals of minors of $\Lambda_i$.

More explicitly, for $1 \leq j \leq n-i$, let $\mathfrak d_{i,j}$ be the ideal
generated by the minors of size $n+1-i-j$ of $\Lambda_i$. Here we set a minor
equal to $0$ if its size is bigger than the size of $\Lambda_i$. We know that
$\mathfrak d_{i,j}$ is of the form $(t^{d_{i,j}})$ for some $d_{i,j} \in
[0,\infty]$. Then the invariant factors of $\psr^n/\Lambda+F_i$ are:
\[
    d_{i,1}-d_{i,2}
    \qquad
    d_{i,2}-d_{i,3}
    \qquad
    \cdots
    \qquad
    d_{i,n-i-1}-d_{i,n-i}
    \qquad
    d_{i,n-i}
\]

Let $(\alpha_{i,j})$ be the essential contact profile of $\Lambda$, extended by
setting $\alpha_{i,j} = \infty$ if $i<1$ and $\alpha_{i,j} = 0$ if $i>k$ or
$j>n-k$. From the previous discussion we see that the
\lcnamecref{schubert-cond-lattices} follows if we prove that
\[
    \alpha_{i+k-n+j,j} = d_{i,j}.
\]

Observe that when $k < n+1-i-j$ both $\alpha_{i+k-n+j,j}$ and $d_{i,j}$ are
equal to $\infty$. So we can assume that $i+k-n+j \geq 1$. Also, when $i<k$ the
matrix $\Lambda_i$ has minors of any size $\leq k-i$ that are equal to $1$.
Therefore, if $i < k$ and $j > n-k$ we see that $\mathfrak d_{i,j}$ must be the
unit ideal, since it contains all minors of size $n+1-i-j < k+1-i$. In
particular in this case $\alpha_{i+k-n+j,j}=d_{i,j}=0$. Since we always have $j
\leq n-i$, we can assume that $j\leq n-k$. Finally, also using that $j \leq
n-i$, we see that $i+k-n+j \leq k$.

We use \cref{gens-rect} with $a=i+k-n+j$ and $b=j$. Notice that the previous
paragraph guarantees that we can consider $1\leq a \leq k$ and $1\leq b \leq
n-k$. Let $\mathfrak a_{a,b}$ be the ideal of $\Omega_{(b^a)} \cap \mathcal U$.
We will show that $\mathfrak a_{a,b} = \mathfrak d_{i,j}$. Consider
\[
    r = \min \{ k-a, n-k-b \}
    = \min \{ n-i-j, n-k-j \} > 0,
\]
and recall that $\mathfrak a_{a,b}$ is generated by certain minors of size
$r+1$.

According to \cref{gens-rect}, we have two cases. If $i \geq k$, then $r=n-i-j$
and $\mathfrak a_{a,b}$ is generated by the minors of size $r+1$ in the first
$b + r$ columns. In this case $b+r = n-i$, and we see that $\mathfrak a_{a,b} =
\mathfrak d_{i,j}$.

The other possibility is $i<k$, which implies $r=n-k-j$. We can write
$\Lambda_i$ in block form:
\[
    \Lambda_i
    =
    \left(
    \begin{array}{c|c}
    A_i & 0 \\
    \hline
    B_i & C_i \\
    \end{array}
    \right),
\]
where $C_i$ is the anti-diagonal matrix of size $(k-i)\times(k-i)$. Then
$\mathfrak a_{a,b}$ is generated by the minors of $A_i$ of size $n+1-k-j$,
while $\mathfrak d_{i,j}$ contains all minors of $\Lambda_i$ of size $n+1-i-j$.
Since determinantal ideals are independent of the choice of basis, we can
perform column operations on $\Lambda_i$ until we obtain a matrix of the form
\[
    \widetilde\Lambda_i
    =
    \left(
        \begin{array}{c|c}
            A_i & 0 \\
            \hline
            0 & C_i \\
        \end{array}
    \right),
\]
and $\mathfrak d_{i,j}$ is still the ideal generated by all minors of
$\widetilde \Lambda_i$ of size $n+1-i-j$. But from the above block form we see
that $\mathfrak d_{i,j}$ must be generated by the minors of
$\widetilde\Lambda_i$ of size $n+1-i-j$ which contain the last $k-i$ columns,
and these are clearly equal to the minors of $A_i$ of size $n+1-k-j$. In other
words, $\mathfrak a_{a,b} = \mathfrak d_{i,j}$, and the
\lcnamecref{schubert-cond-lattices} follows.
\end{proof}

\begin{remark}[Invariant factor profiles for jets]
\label{if-jets}
We can use \cref{schubert-cond-lattices} to extend the definition of invariant
factor profile to jets. Given a jet $\Lambda \in J_m G(k,n)$ we define the
matrix $\beta = (\beta_{i,j})$ in such a way that each quotient module
\[
    \frac{\mathbb C[t]/(t^{m+1})}{\Lambda + F_i}
\]
has invariant factors
\[
    \beta_{i+k-n+1,1} \quad
    \beta_{i+k-n+2,2} \quad
    \cdots \quad
    \beta_{i-1,n-k-1} \quad
    \beta_{i,n-k}.
\]
Then $\beta$ is called the \emph{invariant factor profile} of $\Lambda$. The
same ideas of the proof of \cref{schubert-cond-lattices} show that the
essential contact profile $\alpha$ of $\Lambda$ can be recovered from the
invariant factor profile:
\[
    \alpha_{i,j} = \beta_{i,j} + \beta_{i+1,j+1} + \cdots.
\]
But notice that, in contrast with what happens for arcs, the invariant factor
profile of a jet cannot be recovered in general from its contact profile.

We still denote by $\mathcal C_\beta$ the set of jets with invariant factor
profile $\beta$. Each contact stratum is a union of the $\mathcal C_\beta$.
\end{remark}

\section{Orbits in the arc space}

\label{sec:orbits}

\cref{schubert-cond-lattices} provides a strong motivation for considering
contact strata, but there are other natural decompositions of the arc space of
the Grassmannian, mainly coming from groups actions.

Recall that the action of $\GL_n$ on $G(k,n)$ induces an action at the level of
arc spaces: $J_\infty\GL_n$ acts on $J_\infty G(k,n)$. The arc space
$J_\infty\GL_n$ is the group of invertible matrices with coefficients in
$\psr$, and its action on $J_\infty G(k,n)$ is by column operations (also with
coefficients in $\psr$). This identifies the arc space of the Grassmannian with
the homogeneous space
\[
    J_\infty G(k,n) 
    = \frac{J_\infty \GL_n}{J_\infty P_{k,n}}
    = \frac{\GL_n(\psr)}{P_{k,n}(\psr)},
\]
where $P_{k,n} \subset \GL_n$ is the parabolic subgroup described in
\cref{sec:grass}.

The above presentation of the arc space of the Grassmannian as a homogeneous
space should not be confused with other quotients of similar type that appear
frequently in the literature. The affine Grassmannian and the affine Flag
variety (in type A) are defined as
\[
    \mathcal Gr_n
    = \frac{\GL_n(\lsr)}{\GL_n(\psr)}
    \qquad\text{and}\qquad
    \mathcal Fl_n
    = \frac{\GL_n(\lsr)}{\mathcal B_n},
\]
where $\mathcal B_n = \Cont^{\geq 1}(B) \subset J_\infty\GL_n$ is the Iwahori
group. Denote by $Fl(n) = B \backslash\GL_n$ the Flag manifold. Then there are
natural fibrations relating the different quotients. We have
\[
    \xymatrix{
        \dfrac{\GL_n(\lsr)}{B(\psr)}
        \ar[rr]
        \ar@/_2pc/[rrrr]_{J_\infty Fl(n)}
        &&
        \mathcal Fl_n
        \ar[rr]^{Fl(n)}
        &&
        \mathcal Gr_n
    }
\]
for flags, and
\[
    \xymatrix{
        \dfrac{\GL_n(\lsr)}{P_{k,n}(\psr)}
        \ar[rr]
        \ar@/_2pc/[rrrr]_{J_\infty G(k,n)}
        &&
        \dfrac{\GL_n(\lsr)}{\Cont^{\geq 1}(P_{k,n})}
        \ar[rr]^-{G(k,n)}
        &&
        \mathcal Gr_n
    }
\]
for the Grassmannian. In the above two diagrams, the objects labeling the
arrows represent the fibers of the corresponding fibrations. Observe that, from
this point of view, the arc spaces $J_\infty G(k,n)$ and $J_\infty Fl(n)$ are
very different from $\mathcal Gr_n$ and $\mathcal Fl_n$. The ``affine'' objects
are well behaved representation theoretically: they are natural homogeneous
spaces associated to a Kac-Moody group. On the other hand, the group associated
to the arc spaces is $J_\infty \GL_n$, which is isomorphic to the product of a
reductive group, $\GL_n$, with an infinite dimensional solvable group,
$J_\infty\mathfrak {gl}_n = \Cont^{\geq 1}({\rm Id}) \subset J_\infty\GL_n$.

In principle, two subgroups of $J_\infty \GL_n$ should be relevant from the
point of view of the Grassmannian: the arc space of the Borel $J_\infty B$, and
the Iwahori subgroup $\mathcal B_n = \Cont^{\geq 1}(B) \subset J_\infty \GL_n$.
It is natural to consider the decomposition of $J_\infty G(k,n)$ into the
orbits of either of these two groups. Perhaps surprisingly, and in contrast
with what happens in $G(k,n)$, neither of these actions gives rise to contact
strata.

The Iwahori orbits are just $\Cont^{\geq 1}(\Omega^\circ_I)$, the inverse
images of the Schubert cells. They provide little insight into the structure of
$J_\infty G(k,n)$.

The orbits for $J_\infty B$ are more interesting, but we found them to be less
apt for our study than contact strata. The main difficulty is the lack of a
good combinatorial device parametrizing all the orbits. Notice that contact
strata are invariant under the action of $J_\infty B$, and therefore they are
unions of orbits. But there are a lot more orbits than contact strata. 

Even though we will not pursue a full study of the $J_\infty B$-orbits, we
would like to give an idea of their complexity. To simplify the discussion, we
restrict ourselves to $G(2,4)$, and we only consider arcs centered on the Borel
fixed point. Such arcs are contained in the opposite big cell, and are
represented by matrices of the form
\[
    \Lambda = 
    \begin{pmatrix}
        x_{11} & x_{12} & 0 & 1 \\
        x_{21} & x_{22} & 1 & 0 \\
    \end{pmatrix},
\]
where the coefficients are in the maximal ideal, $x_{ij} \in (t) \subset \psr$.
These arcs are characterized by having an invariant factor profile with base
the full rectangle of size $2\times 2$. Given an element $b \in J_\infty B$,
the translate $\Lambda \cdot b$ is an arc of the same type. In fact, there
exists a unique $g \in J_\infty \GL_2$ such that $\Lambda_b = g \cdot \Lambda
\cdot b$ is again a matrix of the same form as above.

The group $J_\infty B$ is generated by the arc spaces of the torus, $J_\infty
T$, and of the three unipotent subgroups corresponding to the positive roots,
$J_\infty U_{12}$, $J_\infty U_{23}$, and $J_\infty U_{34}$. After
straightforward computations, we can determine the explicit effect of the
action on the matrix $\Lambda$ for each of these generators. The results are
summarized in \cref{fig:borel-gens}, where we have used a particular torus
twist of $U_{23}$ (denoted $\widetilde U_{23}$) so that the expression is more
readable.

\begin{figure}[ht]
    \centering
    \includegraphics{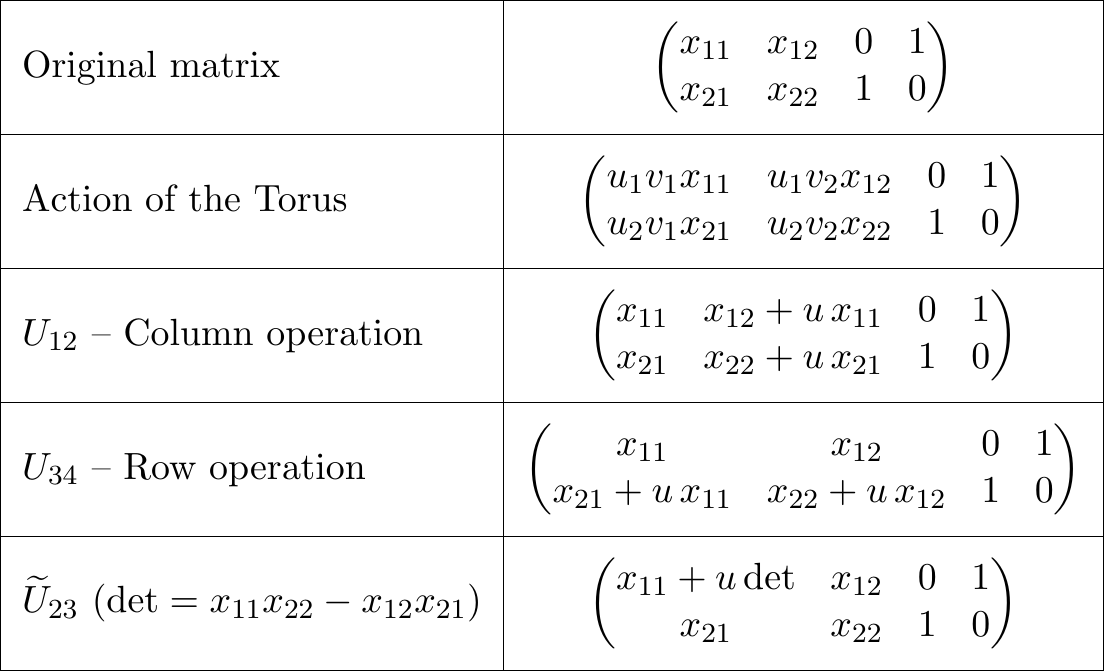}
    \caption{The generators of the action of $J_\infty B$ on $J_\infty G(2,4)$
    over the Borel fixed point. In the table $u_i$ and $v_i$ are units in
    $\psr$, and $u$ is an arbitrary power series.}
    \label{fig:borel-gens}
\end{figure}

After these remarks, it is easy to construct matrices in different orbits but
in the same contact stratum. An example is given by the arcs
\[
    \Lambda_1 =
    \begin{pmatrix}
        0 & t^2 & 0 & 1 \\
        t^2 & 0 & 1 & 0 \\
    \end{pmatrix}
    \quad\text{and}\quad
    \Lambda_2 =
    \begin{pmatrix}
        t^3 & t^2 & 0 & 1 \\
        t^2 & 0 & 1 & 0 \\
    \end{pmatrix},
\]
whose invariant factor profile is
\[
    \beta = 
    \begin{pmatrix}
        2 & 2 \\
        2 & 2 \\
    \end{pmatrix}.
\]
The moves in \cref{fig:borel-gens} do not allow to transform $\Lambda_1$ into
$\Lambda_2$. In order to change the $0$ in position $(1,1)$ with a $t^3$, the
only operation that could help would be $\widetilde U_{23}$. But the
determinant $\det$, even after any other move, is a multiple of $t^4$, so we
can never obtain the desired $t^3$.

Other examples display even more pathological behavior. We consider the
matrices
\[
    \Lambda_u = 
    \begin{pmatrix}
        u t^3 & t^2 & 0 & 1 \\
        t^2 & t & 1 & 0 \\
    \end{pmatrix},
    \qquad
    u \in \C \setminus \{0,1\},
\]
which clearly belong to the same contact stratum. We will show that they give
rise to different orbits. Notice that in this way we obtain a continuous family
of orbits, whereas the family of contact strata is countable. 

To see that no two among the $\Lambda_u$ belong to the same orbit, we could
argue as above, using the generators of $J_\infty B$. More directly, we proceed
as follows. Consider
\begin{equation}
    \label{eq:lambda_u_v}
    \Lambda_v = g \cdot \Lambda_u \cdot b,
\end{equation}
where $g \in J_\infty\GL_2$ and $b = (b_{ij}) \in J_\infty B$. We apply all the
Plücker coordinates $[ij]$ to both sides of \cref{eq:lambda_u_v}. After
explicit computations we get:
\[
    \begin{array}{l@{\colon\qquad}r@{\,}l@{\quad}c@{\quad}r@{\,}ll}
        [12]
        & t^4 & (v-1)
        & =
        & t^4 & (\det g)\, b_{11}\, b_{22}\, (u-1)
    \\[.5em]
        [13]
        & t^3 & v
        & \equiv
        & t^3 & (\det g)\, b_{11}\, b_{33}\, u
        & \mod t^4
    \\[.5em]
        [23]
        & t^2 &
        & \equiv
        & t^2 & (\det g)\, b_{22}\, b_{33}
        & \mod t^3
    \\[.5em]
        [14]
        & -t^2 &
        & \equiv
        & -t^2 & (\det g)\, b_{11}\, b_{44}
        & \mod t^3
    \\[.5em]
        [24]
        & -t\phantom{^1} &
        & \equiv
        & -t\phantom{^1} & (\det g)\, b_{22}\, b_{44}
        & \mod t^2
    \\[.5em]
        [34]
        & -1\phantom{^1} &
        & \equiv
        & -\phantom{t^0} & (\det g)\, b_{33}\, b_{44}
        & \mod t
    \\
    \end{array}
\]
Focusing on the terms of lowest degree, the above equations imply:
\[
    \begin{array}{l@{\qquad}l@{\qquad}l}
        b^0_{11} \, b^0_{22} \, \delta = \dfrac{v-1}{u-1},
        &
        b^0_{22} \, b^0_{33} \, \delta = 1,
        &
        b^0_{22} \, b^0_{44} \, \delta = 1,
    \\[1.5em]
        b^0_{11} \, b^0_{33} \, \delta = \dfrac{v}{u},
        &
        b^0_{11} \, b^0_{44} \, \delta = 1,
        &
        b^0_{33} \, b^0_{44} \, \delta = 1,
    \\
    \end{array}
\]
where $b_{ii}^0 \ne 0$ is the constant coefficient of $b_{ii}$, and $\delta \ne
0$ is the constant coefficient of $\det g$. But these equations imply that
$u=v$, as required.

We know very little about the $J_\infty B$-orbit decomposition of $J_\infty
G(k,n)$ in general. It would be interesting to understand these orbits better,
and to study contact strata from this point of view.

\section{Constructing arcs using planar networks}

\label{sec:planar-networks}

The goal of this section is to finish the proof of \cref{if-prof}, that is, we
want to show that all plane partitions appear as the invariant factor profile
of an arc in the Grassmannian. In order to do this, we need to find a way of
producing arcs with prescribed order of contact with respect to all the
(single-condition) Schubert varieties. This task seems to be highly non-trivial.
We have borrowed ideas form the theory of the totally positive Grassmannian: we
use planar networks to construct matrices, and use Lindström's Lemma to control
the behavior of the minors. Most of the results in this section are adaptations
to the case of arcs of the ideas in \cite{FZ00}.

\subsection*{Planar networks}

A \emph{planar network} $\Gamma$ is an acyclic directed finite graph with a
fixed embedding in the closed disk. We allow multiple edges, but no loops. We
identify planar networks when they are homotopy equivalent (respecting the
boundary of the disk). The vertices of $\Gamma$ can be naturally classified
into four types: \emph{sources}, \emph{sinks}, \emph{internal vertices}, and
\emph{isolated vertices}. In all the cases that we consider, $\Gamma$ has no
isolated vertices, and can be drawn inside of the disc in such a way that the
set of \emph{boundary vertices} is the union of the sources and the sinks.
Furthermore, $\Gamma$ will have $k$ sources, labeled from $1$ to $k$ clockwise
along the boundary, and $(n-k)$ sinks, labeled counterclockwise.

When discussing \emph{paths} on a planar network, we always assume that they
are directed. A path is called \emph{maximal} if it connects a source with a
sink. A collection of paths is said to be \emph{non-intersecting} if no two
paths in the collection share a vertex.

A \emph{chamber} of a planar network $\Gamma$ is a connected component of the
complement of $\Gamma$ in the closed disk. Let $p$ be a maximal path in
$\Gamma$. Then $p$ splits the disk in two connected components, which, taking
into account the natural orientation of $p$, are called the \emph{left} and
\emph{right} sides of $p$. Every chamber of $\Gamma$ is either to the left or
to the right of $p$.

A \emph{weighting} of a planar graph $\Gamma$ is a collection $w = \{ w_v,
w_e\}$, where $v$ ranges among the internal vertices of $\Gamma$ and $e$ ranges
among the edges of $\Gamma$. The elements of $w$ belong to some ring fixed in
advance, which in our case it will always be $\psr$. Using such a weighting
$w$, we define the \emph{weight of a path} $p$ in $\Gamma$ as the product of
the weights of all the vertices and all the edges in $p$,
\[
    w_p = \left( \prod_{v \in p} w_v \right)
          \left( \prod_{e \in p} w_e \right).
\]
The \emph{weight of a collection of paths} is the product of the weights of the
paths in the collection.

The \emph{weight matrix} $X(\Gamma,w)$ is the matrix of size $k \times (n-k)$
whose entry in position $(i,j)$ is the sum of the weights of all (maximal)
paths with source $i$ and sink $j$,
\[
    x_{ij} 
    = \sum_{p \in {\rm Paths}(i \to j)} w_p.
\]

The following result makes calculations with weight matrices particularly
convenient.

\begin{lemma}[Lindström Lemma]
    \label{lindstrom}
    Let $X=X(\Gamma, w)$ be the weight matrix of a weighted planar network, and
    let $[I|J]$ be the minor of $X$ with row set $I$ and column set $J$. Then
    $[I|J]$ is the sum of the weights of the collections of non-intersecting
    paths that connect the sources in $I$ with the sinks in $J$.
\end{lemma}

For the proof we refer the reader to \cite[Lemma~1]{FZ00}. There it can be
found the proof for the case where only edge weightings are used, but the same
idea works for arbitrary weightings.

\subsection*{A particular network}

Given a plane partition $\beta$, we will produce arcs in the contact stratum
$\mathcal C_\beta$ using a particular weight matrix. We now describe the
corresponding planar network. The cases of $G(3,8)$ (with $k=3$ and $n-k=5$)
and $G(2,5)$ (with $k=2$ and $n-k=3$) are given in \cref{fig:main-network}. 

\begin{figure}[ht]
    \centering
    \includegraphics{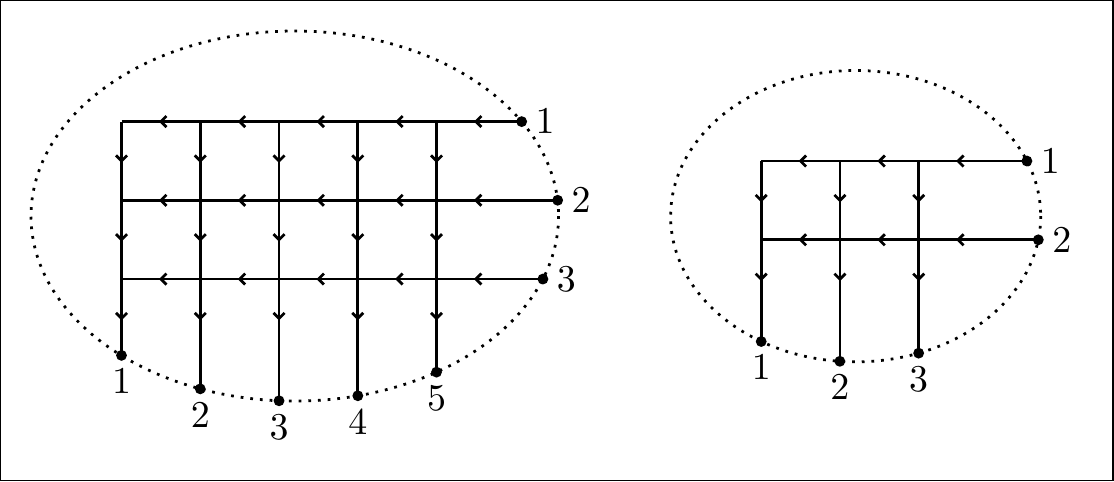}
    \caption{The network $\Gamma_0$.}
    \label{fig:main-network}
\end{figure}

The network we describe is denoted $\Gamma_0$. It has $k$ sources, $n-k$ sinks,
and $k \times (n-k)$ internal vertices. The internal vertices are arranged
using $k$ rows and $(n-k)$ columns, and $\Gamma_0$ has edges connecting the
internal edges to form a grid. The resulting $k$ horizontal lines are extended
towards the right until the boundary of the disk, where we place the $k$
sources. The $n-k$ vertical lines are extended towards the bottom, until the
$n-k$ sinks. The horizontal edges are oriented from right to left, and the
vertical ones from top to bottom. We get $k(n-k)+1$ chambers, one on top, and
the rest arranged in a grid of size $k \times (n-k)$. We label these last
chambers $C_{ij}$, using matrix indexing. 

\begin{figure}[ht]
    \centering
    \includegraphics{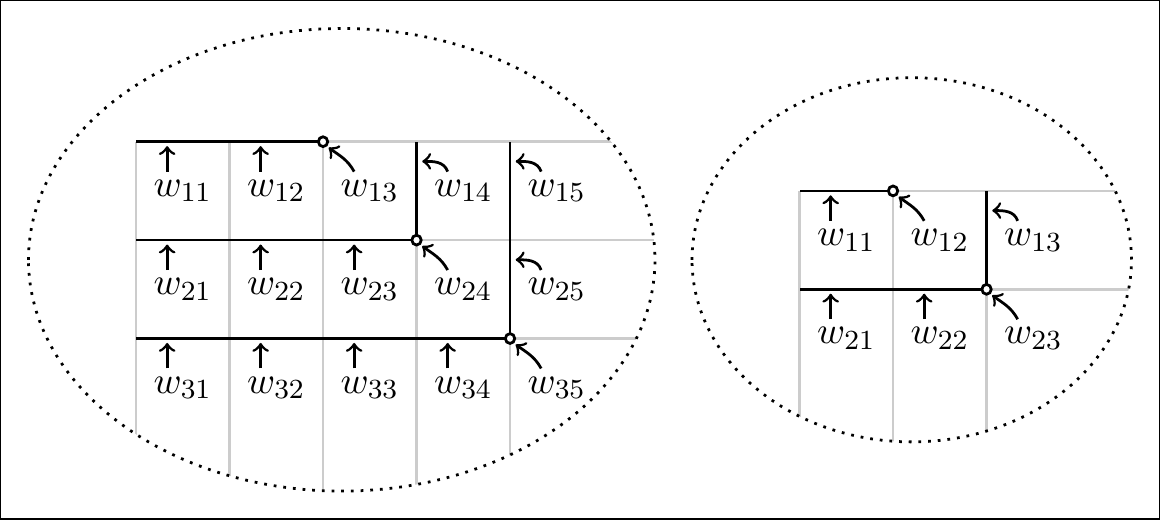}
    \caption{The weights of $\Gamma_0$. Only the dark edges and the marked
    vertices get a weight different from $1$.}
    \label{fig:network-weights}
\end{figure}

To assign weights to $\Gamma_0$, we consider a matrix $(w_{ij})$ of size $k
\times (n-k)$. If $k-i < n-k-j$ (resp.\ $k-i > n-k-j$), we assign weight
$w_{ij}$ to the edge on the top (resp.\ to the left) of the chamber $C_{ij}$.
If $k-i = n-k-j$, we assign weight $w_{ij}$ to the vertex in the top-left
corner of $C_{ij}$. All other edges and vertices get weight $1$. See
\cref{fig:network-weights} for some examples. Following \cite{FZ00}, we call a
weighting $w$ of $\Gamma_0$ obtained this way an \emph{essential weighting}.

To compute the weight matrix $X(\Gamma_0,w)$, only the marked edges and vertices
in \cref{fig:network-weights} contribute some weight. For example, in
$G(2,5)$ we obtain:
\[
    X(\Gamma_0,w) =
    \begin{pmatrix}
        w_{12}w_{11} + w_{12}w_{21} + w_{13}w_{23}w_{22}w_{21}
        & w_{12} + w_{13}w_{23}w_{22}
        & w_{13}w_{23}
    \\
        w_{23}w_{22}w_{21}
        & w_{23}w_{22}
        & w_{23}
    \end{pmatrix}.
\]

We remark that our network $\Gamma_0$ is slightly different from the one used
in \cite[Figure~2]{FZ00}. The reason is that the authors of \cite{FZ00} prefer
to use only edge weightings, while we decided to allow arbitrary weightings.
The network in \cite{FZ00} can be obtained from ours by replacing each weighted
vertex with a diagonal edge (oriented from top-right to bottom-left), and
moving weight from the vertices to these new edges. This process is explained
visually in \cref{fig:compareFZ}.

\begin{figure}[ht]
    \centering
    \includegraphics{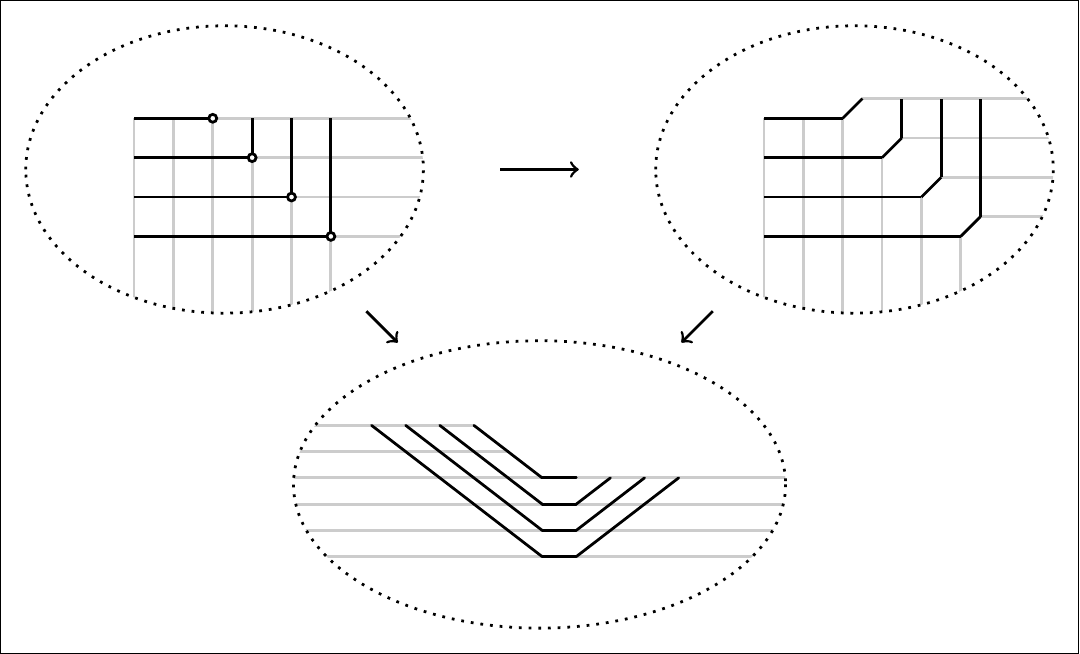}
    \caption{The relationship between our network and the one in \cite{FZ00}.}
    \label{fig:compareFZ}
\end{figure}

\subsection*{Final minors in the weight matrix}

From now on we assume that the weights $w_{ij}$ belong to a power series ring,
either $\psr$ or $K\llbracket t\rrbracket$, where $K$ is some field. We can
think of $X(\Gamma_0,w)$ as giving a ($K$-valued) arc in the Grassmannian. More
precisely, we consider
\[
    \Lambda(\Gamma_0,w) = 
        \Big( \, X(\Gamma_0,w) \, \Big| \,\Delta' \,\Big),
    \qquad\text{where}\qquad
    \Delta' = 
    \begin{pmatrix}
        0      & \cdots  & 0      & 1 \\
        0      & \cdots  & 1      & 0 \\
        \vdots & \iddots & \vdots & \vdots \\
        1      & \cdots  & 0      & 0 \\
    \end{pmatrix}.
\]
Notice that $\Lambda(\Gamma_0,w)$ is an arc in the opposite big cell $\mathcal U$.

We are interested in understanding the invariant factor profile of
$\Lambda(\Gamma_0,w)$. From \cref{gens-rect}, this involves studying the order
of $t$ in all the minors of $X(\Gamma_0,w)$. We start with the final minors
(recall our terminology from \cref{sec:grass}).

\begin{lemma}
    \label{explicit-minors}
    Let $M_{i,j}$ be a final minor of $X(\Gamma_0,w)$. Then there is a
    unique collection of non-intersecting paths of $\Gamma_0$ whose weight
    gives $M_{i,j}$. Moreover, if $(k-i) \leq (n-k-j)$ we have
    \[
        M_{i,j} = 
        \prod_{u=0}^{k-i} \prod_{v=0}^{v_{\rm max}}
        w_{i+u,\, j+u+v},
        \qquad
        v_{\rm max} = (n-k-j) - (k-i),
    \]
    and if $(k-i) \geq (n-k-j)$ we have
    \[
        M_{i,j} = 
        \prod_{u=0}^{u_{\rm max}} \prod_{v=0}^{n-k-j}
        w_{i+u+v,\, j+v},
        \qquad
        u_{\rm max} = (k-i) - (n-k-j).
    \]
\end{lemma}

The unique collection of paths is described in the proof. For an example, see
\cref{fig:explicit-minors-example}.

\begin{figure}[ht]
    \centering
    \includegraphics{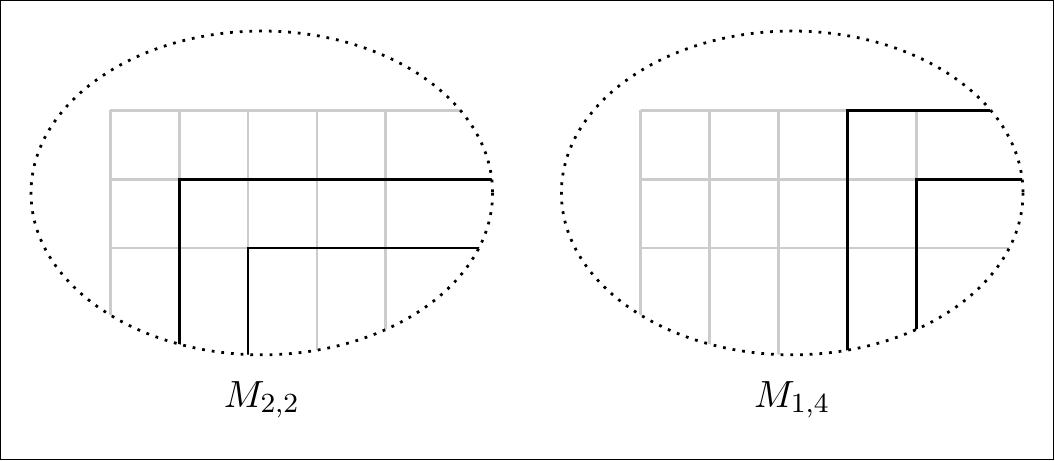}
    \caption{The unique collections of paths in $\Gamma_0$ realizing the final
    minors $M_{2,2}$ and $M_{1,4}$ in $G(3,8)$.}
    \label{fig:explicit-minors-example}
\end{figure}

\begin{proof}
We use \cref{lindstrom} to compute $M_{a,b}$. Recall that $M_{a,b} = [a \ldots
a+r | b \ldots b+r]$, where $r = \min\{k-a, n-k-b\}$. In particular, $M_{a,b}$
involves either the last $r+1$ rows (if $b \leq a+n-2k$) or the last $r+1$
columns (if $b \geq a+n-2k$). In both cases there is a unique collection of
non-intersecting paths of $\Gamma_0$ connecting the sources $\{a,\ldots,a+r\}$
with the sinks $\{b,\ldots,b+r\}$.

For example, assume $b \leq a+n-2k$. We need to connect $\{a, \ldots, k\}$ with
$\{b, \ldots, b+r\}$. There is a unique path $p_r$ in $\Gamma_0$ connecting the
source $k$ with the sink $b+r$. After removing $p_r$ from $\Gamma_0$, there is
a unique path $p_{r-1}$ from $k-1$ to $b+r-1$. Inductively, we produce the
unique collection of paths $\{p_0, \ldots, p_r\}$. To the weight of $p_r$ there
are only contributions from one marked vertex (the one in position $(k,n-k)$),
and the marked horizontal edges from $(k,n-k)$ to $(k,b+r)$. We get that
$w_{p_r} = \prod_{s=b+r}^{n-k} w_{k,s} = \prod_{v=0}^{v_{\rm max}}w_{k,b+r+v}$.
Similarly $w_{p_u} = \prod_{v=0}^{v_{\rm max}}w_{a+u,b+u+v}$, and the formula
for $M_{a,b}$ follows.

The case $b\geq a+n-2k$ is completely analogous.
\end{proof}

\subsection*{Weight exponents}

Given a plane partition $\beta = (\beta_{i,j})$ we define the \emph{weight
exponents} associated to $\beta$ as the matrix $c(\beta) = (c_{i,j})$ given by:
\begin{align*}
    c_{i,j} &= \beta_{i,j} 
    & \text{if $(k-i) = (n-k-j)$,}
\\
    c_{i,j} &= \beta_{i,j} - \beta_{i,j+1}
    & \text{if $(k-i) < (n-k-j)$,}
\\
    c_{i,j} &= \beta_{i,j} - \beta_{i+1,j}
    & \text{if $(k-i) > (n-k-j)$.}
\end{align*}
Here we use the convention that $\infty - x = \infty$. For example, in the case
of $G(3,6)$ we have:
\[
    c(\beta) =
    \begin{pmatrix}
        \beta_{1,1} & \beta_{1,2} - \beta_{2,2} & \beta_{1,3} - \beta_{2,3} \\
        \beta_{2,1} - \beta_{2,2} & \beta_{2,2} & \beta_{2,3} - \beta_{3,3} \\
        \beta_{3,1} - \beta_{3,2} & \beta_{3,2} - \beta_{3,3} & \beta_{3,3} \\
    \end{pmatrix}
\]

For a given essential weighting $w$ of $\Gamma_0$ associated to a matrix
$(w_{i,j})$ whose coefficients are power series, we define the \emph{weight
exponents} of $w$ as the matrix $c(w) = (c_{i,j})$ given by
\[
    c_{i,j} = \ord_{t} (w_{i,j}).
\]

The main goal of this section is to prove the following theorem, which
immediately concludes the proof of \cref{if-prof}.

\begin{theorem}
    \label{c-works}
    Let $\beta$ be a plane partition with associated weight exponents
    $c(\beta)$, and let $w$ be an essential weighting of\/ $\Gamma_0$ with
    weight exponents $c(w)$. If $c(\beta) = c(w)$, then the arc
    $\Lambda(\Gamma_0,w)$ belongs to the contact stratum $\mathcal C_\beta$.
\end{theorem}

\begin{figure}[ht]
    \centering
    \includegraphics{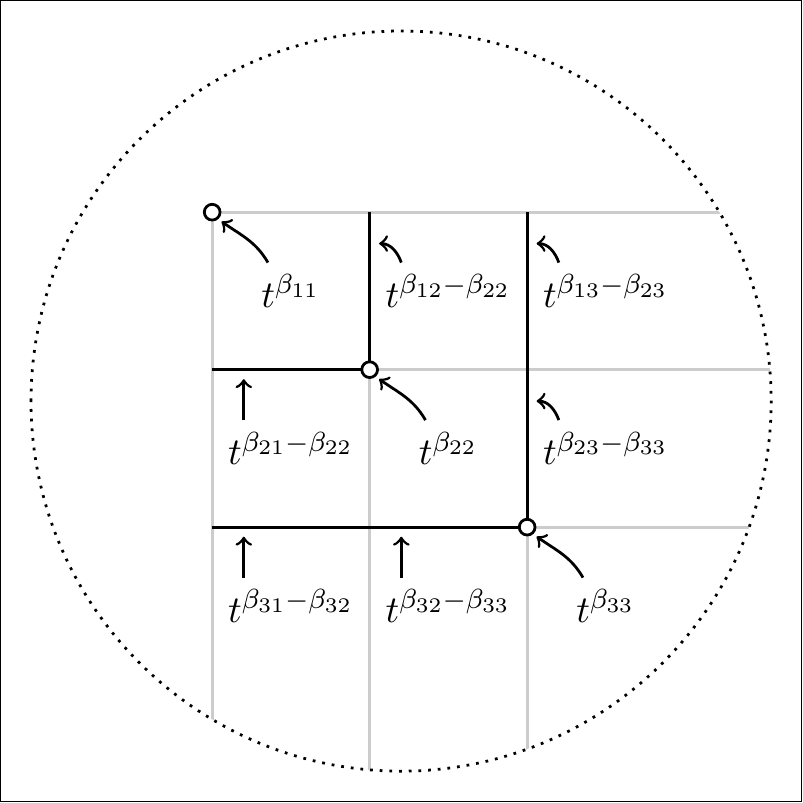}
    \caption{The weighted network $(\Gamma_0,t^{c_{i,j}})$ for $G(3,6)$.}
    \label{fig:network-g36}
\end{figure}

\begin{figure}[ht]
    \centering
    \includegraphics{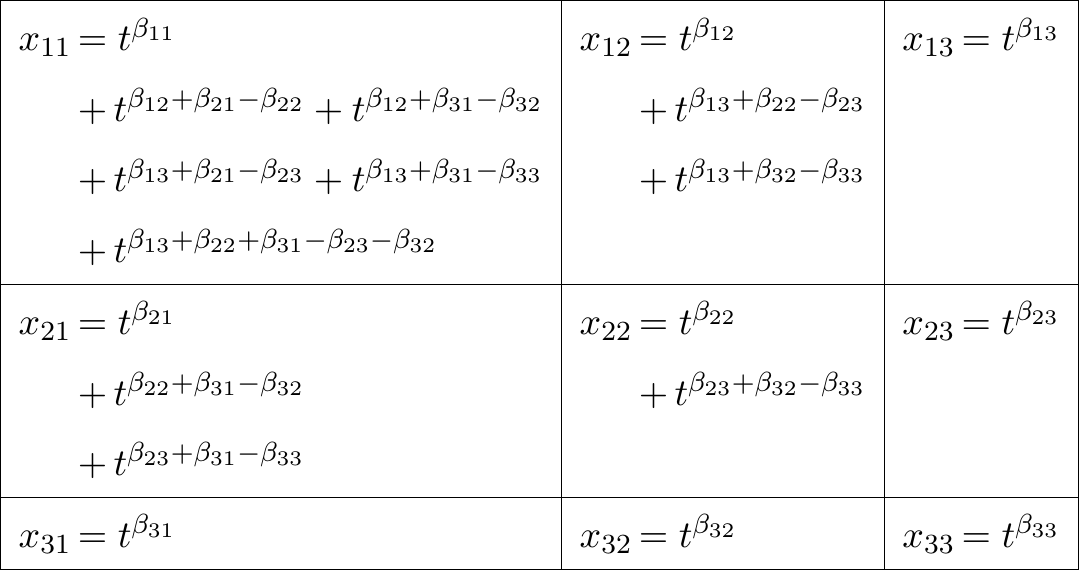}
    \caption{The matrix $X(\Gamma_0,t^{c_{i,j}})$ in $G(3,6)$ for a generic
    plane partition $\beta = (\beta_{i,j})$. It gives an arc in the contact
    stratum $\mathcal C_\beta$.}
    \label{fig:generic-g36}
\end{figure}

The simplest arc provided by \cref{c-works} is obtained by setting $w_{ij} =
t^{c_{i,j}}$. For example, the weighted network $(\Gamma_0,t^{c_{i,j}})$ for
$G(3,6)$ is given in \cref{fig:network-g36}, and the matrix
$X(\Gamma_0,t^{c_{i,j}}) = (x_{i,j})$ is given in \cref{fig:generic-g36}.
Notice how the entries of the resulting matrix have many terms. In some cases,
it is possible to exhibit simpler arcs in a contact stratum, in the sense that
their matrices are more sparse, have more zeros. For example, for the plane
partition
\[
    \beta = 
    \begin{pmatrix}
        2 & 2 \\
        2 & 1 \\
    \end{pmatrix}
\]
our method constructs the arc
\[
    \begin{pmatrix}
        t^2 + t^3 & t^2 & 0 & 1 \\ 
        t^2       & t   & 1 & 0 \\ 
    \end{pmatrix},
\]
but the following arc also belongs to $\mathcal C_\beta$:
\[
    \begin{pmatrix}
        t^2 & 0 & 0 & 1 \\ 
        0   & t & 1 & 0 \\ 
    \end{pmatrix}.
\]
In general, it seems hard to give an algorithm producing sparse examples.

We proceed now to prove \cref{c-works}. For the rest of the section we assume
the hypotheses of the \lcnamecref{c-works}: we fix a plane partition $\beta$
and an essential weighting $w$ with $c(\beta) = c(w)$. We start with the
following easy statement, which in fact motivated our definition of $c(\beta)$.

\begin{lemma}
    \label{beta-final}
    Let $M_{i,j}$ be a final minor of $X(\Gamma_0,w)$, and set
    $r=\min\{k-i,n-k-j\}$ Then:
    \[
        \ord_{t}(M_{i,j}) = 
        \beta_{i,j} + \beta_{i+1,j+1} + \cdots + \beta_{i+r,j+r}.
    \]
\end{lemma}

\begin{proof}
This is an immediate consequence of \cref{explicit-minors}.
\end{proof}

Let $M = [a_0 \ldots a_r|b_0 \ldots b_r]$ be a minor of $X(\Gamma_0,w)$ of size
$r+1$. We define
\[
    M' = M_{k-r,b_{r}-r} = [k-r \ldots k | b_r-r \ldots b_r]
\]
and
\[
    M'' = M_{a_r-r,n-k-r} = [a_r-r \ldots a_r| n-k-r \ldots n-k].
\]
Notice that both $M'$ and $M''$ are final minors, and that $M \leq M'$ and $M
\leq M''$ (according to the order on minors introduced in \cref{sec:grass}).

\begin{lemma}
    \label{minor-projections}
    Let $M$, $M'$, and $M''$ be as above. Then:
    \[
        \ord_t (M) \geq \ord_t(M')
        \qquad\text{and}\qquad
        \ord_t (M) \geq \ord_t(M'').
    \]
\end{lemma}

\begin{proof}
We focus on the first inequality, the proof of the second one is analogous.
Using \cref{lindstrom} we get an expansion
\[
    M = \sum w_{p_0} \cdots w_{p_r},
\]
the sum ranging among all non-intersecting collections of paths
$\{p_0,\ldots,p_r\}$ where $p_s$ connects the source $a_s$ with the sink $b_s$.
From \cref{explicit-minors} we know that
\[
    M' = w_{q_0} \cdots w_{q_r},
\]
where $q_s$ is the path that starts at the source $k-r+s$, then moves
horizontally to the left until the column $b_r-s$, and then moves vertically
down until the sink $b_r-s$. Observe that $\ord_t(w_{q_s}) =
\beta_{k-r+s,b_r-s}$. The \lcnamecref{minor-projections} will follow if we show
that $\ord_t(w_{p_s}) \geq \ord_t(w_{q_s})$ for all the possible collections of
paths $\{p_s\}$ and all $s$.

Fix one such collection $\{p_s\}$, and $0 \leq s \leq r$. In $\Gamma_0$, the
vertices that have weights are disposed along a diagonal, in such a way that
any maximal path must pass though one of them. Let $v^0$ be the (only) weighted
vertex in the path $p_s$, and let $\tilde p_s$ be the final part of the path
$p_s$ that connects $v^0$ with the sink $b_s$. Let $NW$ be the collection of
north-west corners of $\tilde p_s$; these are the vertices $v$ of $\tilde p_s$
for which no other vertex of $\tilde p_s$ is immediately above or immediately
to the left of $v$. We write $NW = \{v_{i_1,j_1}, \ldots v_{i_\ell,j_\ell}\}$,
where the vertices are ordered using the orientation of $\tilde p_s$, and
$v_{i,j}$ denotes the vertex of $\Gamma_0$ in the grid position $(i,j)$.
Observe that it could happen that $v^0 = v_{i_1,j_1}$. From the definitions it
follows that
\[
    \ord_t(w_{\tilde p_s})
    =
    \beta_{i_1,j_1}
    + \beta_{i_2,j_2} - \beta_{i_2,j_1}
    + \cdots
    + \beta_{i_\ell,j_\ell} - \beta_{i_\ell,j_{(\ell-1)}},
\]
and, using the fact that $\beta$ is a plane partition,
\[
    \ord_t(w_{\tilde p_s})
    \geq
    \beta_{i_\ell,j_\ell}.
\]

By construction, the path $p_s$ must be above the path $q_s$, and therefore the
lowest north-west corner of $p_s$ (with is $v_{i_\ell,j_\ell}$) must be to the
north-west of the (only) north-west corner of $q_s$ (which is
$v_{k-r+s,b_r-s}$). In particular
\[
    i_\ell \leq k-r+s
    \qquad\text{and}\qquad
    j_\ell \leq b_r-s.
\]
Using again that $\beta$ is a plane partition we see that
\[
    \ord_t(w_{p_s})
    \geq
    \ord_t(w_{\tilde p_s})
    \geq
    \beta_{i_\ell,j_\ell} 
    \geq 
    \beta_{k-r+s,b_r-s}
    =
    \ord_t(w_{q_s}),
\]
as required.
\end{proof}

\begin{proof}[Proof of \cref{c-works}]
The \lcnamecref{c-works} follows if we show that
\[
    \ord_\Lambda(\Omega_{j^i}) = \ord_t(M_{i,j}),
\]
where $\Lambda = \Lambda(\Gamma_0,w)$ and $M_{i,j}$ is a final minor of
$X(\Gamma_0,w)$. From \cref{gens-rect} it is enough to show that
\begin{equation}
    \label{eq:c-works-goal}
    \ord_t(M) \geq \ord_t(M_{i,j}),
\end{equation}
for all minors $M$ such that $M \leq M_{i,j}$. Consider $M'$ and $M''$ as
\cref{minor-projections}. Then, when $(k-i) \leq (n-k-j)$ we have that $M \leq M'
\leq M_{i,j}$, and when $(k-i) \geq (n-k-j)$ we have $M \leq M'' \leq M_{i,j}$.
Therefore, using \cref{minor-projections}, we see that it is enough to prove
\cref{eq:c-works-goal} in the case where $M$ is a final minor. But this case is
a consequence of \cref{beta-final} and the fact that $\beta$ is a plane
partition.
\end{proof}

\section{Schubert valuations}

\label{sec:valuations}

The main goal of this section is to prove the following statement.

\begin{theorem}
    \label{irred}
    Every contact stratum $\mathcal C_\beta$ is an irreducible subset of
    $J_\infty G(k,n)$.
\end{theorem}

\begin{definition}
    From \cref{irred} it follows that the closure $\overline{\mathcal C}_\beta$
    is the maximal arc set in $J_\infty G(k,n)$ associated to the
    semi-valuation $\ord_\beta$. These semi-valuations are called
    \emph{Schubert semi-valuations}.
\end{definition}

Observe that $\mathcal C_\beta$ is a contact locus precisely when $\beta$ is a
plane partition with finite height (no infinities allowed). In this case
$\ord_\beta$ is a valuation (and not just a semi-valuation). As we will see,
Schubert valuations are the most relevant from the point of view of the study
of the singularities of Schubert varieties. 

For a Schubert semi-valuation $\ord_\beta$, we can easily determine its home
and its center form the plane partition $\beta$. We let $\beta^1$ be the base
of $\beta$: the linear partition whose diagram contains the positions $(i,j)$
where $\beta_{i,j} \geq 1$. Analogously, $\beta^\infty$ is the partition
corresponding to the condition $\beta_{i,j} = \infty$. Then the home of
$\ord_\beta$ is the Schubert variety $\Omega_{\beta^\infty}$, and its center is
$\Omega_{\beta^1}$.

To prove \cref{irred} we will use the techniques developed in
\cref{sec:planar-networks} and produce explicitly the generic point of
$\mathcal C_\beta$. We consider the torus $(\mathbb C^\times)^{k(n-k)}$, and
its arc space $J_\infty (\mathbb C^\times)^{k(n-k)} = (\psr^\times)^{k(n-k)}$.
Notice that this arc space is a connected algebraic group, and in particular it
is irreducible. Its generic point is a matrix that we denote $(u_{i,j})$. Its
entries are of the form
\[
    u_{i,j} = 
    u^{[0]}_{i,j} +
    u^{[1]}_{i,j}\, t +
    u^{[2]}_{i,j}\, t^2 +
    \cdots +
    u^{[p]}_{i,j}\, t^p +
    \cdots
\]
where the coefficients $u^{[p]}_{i,j}$ are transcendentals generating the
function field of the arc space of the torus:
\[
    \mathbb C\left(J_\infty (\mathbb C^\times)^{k(n-k)}\right)
    =
    \mathbb C\left(~
        u_{i,j}^{[p]} 
    ~\middle|~
        {1 \leq i \leq k,} \quad
        {1 \leq j \leq n-k,} \quad
        {0 \leq p \leq \infty}
    ~\right).
\]

Fix a plane partition $\beta$, possibly with infinite height. With the
notations of \cref{sec:planar-networks}, we define an essential weighting $w =
w(\beta,u)$ on $\Gamma_0$ given by
\[
    w_{i,j} = t^{c_{i,j}} \, u_{i,j},
\]
where the $c_{i,j}$ are the weight exponents associated to $\beta$. Recall that
we use the notations
\[
    X(\Gamma_0, w(u,\beta))
    \qquad\text{and}\qquad
    \Lambda(\Gamma_0, w(u,\beta))
\]
for the weight matrix and arc associated to this weighting. We think of them as
giving a morphism between arc spaces:
\[
    \Phi_\beta \colon 
    J_\infty (\mathbb C^\times)^{k(n-k)}
    \to
    J_\infty G(k,n),
    \qquad 
    u \mapsto \Lambda(\Gamma_0,w(u,\beta)).
\]
Notice that it follows form \cref{c-works} that the image of $\Phi_\beta$ is
contained in $\mathcal C_\beta$. Also, from \cref{explicit-minors} it is easy
to see that the minors $M_{i,j}$ determine the weights $w_{i,j}$, and therefore
the morphism $\Phi_\beta$ is injective.

\begin{lemma}
    \label{precise-irred}
    Let $\Lambda$ be an arc in the contact stratum $\mathcal C_\beta$.
    \begin{enumerate}
        \item \label{precise-irred:1}
            $\Lambda$ belongs to the image of $\Phi_\beta$ if and only if it is
            contained in the opposite big cell and
            $\ord_\Lambda(\Omega_{(j^i)}) = \ord_\Lambda(M_{i,j})$ for each
            final minor $M_{i,j}$.
        \item \label{precise-irred:2}
            The Borel orbit $\Lambda \cdot B$ has non-empty intersection with
            the image of $\Phi_\beta$.
    \end{enumerate}
\end{lemma}

\begin{proof}[Proof of \cref{irred}]
Consider the morphism:
\[
    \Psi_\beta \colon 
    B \times J_\infty (\mathbb C^\times)^{k(n-k)}
    \to
    J_\infty G(k,n),
    \qquad 
    (b,u) \mapsto \Lambda(\Gamma_0,w(u,\beta)) \cdot b.
\]
From \cref{precise-irred}, part \ref{precise-irred:2}, we see that the image of
$\Psi_\beta$ is the whole contact stratum $\mathcal C_\beta$. Since the domain
of $\Psi_\beta$ is irreducible, the \lcnamecref{irred} follows.
\end{proof}

\begin{proof}[Proof of \cref{precise-irred}, part \ref{precise-irred:2}]
Let $\Lambda'$ be a generic $B$-translate of $\Lambda$. Since all Schubert
varieties intersect the opposite big cell, we know that $\Lambda'$ is contained
in the opposite big cell. Let $M \leq M_{i,j}$ be the minor of $\Lambda'$ for
which $\ord_{\Lambda'}(\Omega_{(j^i)}) = \ord_{\Lambda'}(M)$. The action of $B$
on $\Lambda'$ transforms $M_{i,j}$ into a linear combination of minors $\tilde
M$ of the same size verifying $\tilde M \leq M_{i,j}$. Moreover, if the action
is by a generic element of $B$, all such minors appear in the linear
combination. In particular, since $\Lambda'$ is already a generic translate, we
see that $\ord_{\Lambda'}(M) = \ord_{\Lambda'}(M_{i,j})$. Now the result
follows from part \ref{precise-irred:1}.
\end{proof}

\begin{proof}[Proof of \cref{precise-irred}, part \ref{precise-irred:1}]
The necessary condition is an immediate consequence of \cref{minor-projections}
and \cref{beta-final}. In fact this was already shown during the proof of
\cref{c-works}, as \cref{eq:c-works-goal}.

For the sufficient condition, let $\Lambda$ be an arc satisfying the
hypothesis, and consider the final minors $M_{i,j}$ of $\Lambda$. It follows
form \cref{explicit-minors} that we can find a $u$ for which
$\Lambda(\Gamma_0,w(u,\beta))$ also has the $M_{i,j}$ as its final minors. The
result follows if we show that $\Lambda = \Lambda(\Gamma_0,w(u,\beta))$. For
this we use the argument of \cite[Lemma~7]{FZ00}, adapted to allow for the
possibility of some final minors being zero.

Let $X = (x_{i,j})$ and $X(\Gamma_0,w(u,\beta)) = (\tilde x_{i,j})$ be the
matrices determining $\Lambda$ and $\Lambda(\Gamma_0,w(u,\beta))$ in the
opposite big cell. We know that the final minors of these two matrices agree.
Also, notice that both $(x_{i,j})$ and $(\tilde x_{i,j})$ verify the
hypothesis, which we rewrite as
\begin{equation}
    \label{eq:hyp-rewrite}
    \begin{aligned}
        \ord_t(M_{i,j}) 
        &= 
        \min\big\{ 
            \ord_{t}(M) 
            \,\big|\,
            \text{$M$ a minor of $(x_{i,j})$, $M\leq M_{i,j}$}
        \big\}
    \\
        &= 
        \min\big\{ 
            \ord_{t}(\tilde M) 
            \,\big|\,
            \text{$\tilde M$ a minor of $(\tilde x_{i,j})$, $\tilde M\leq M_{i,j}$}
        \big\}.
    \end{aligned}
\end{equation}

We prove by induction that $x_{i,j} = \tilde x_{i,j}$. The base case is when
$i=k$ or $j=n-k$, which clearly implies $M_{i,j} = x_{i,j} = \tilde x_{i,j}$.
Otherwise we have an expansion
\[
    M_{i,j} = x_{i,j} M_{i+1,j+1} + P(x)
\]
where $P(x)$ is a polynomial in entries $x_{i',j'}$ where $i'\geq i$, $j'\geq
j$, and $(i',j') \ne (i,j)$. Analogously,
\[
    M_{i,j} = \tilde x_{i,j} M_{i+1,j+1} + P(\tilde x)
\]
where $P(\tilde x)$ is obtained from $P(x)$ by replacing each $x_{i',j'}$ with
$\tilde x_{i',j'}$. By induction we see that
\[
    x_{i,j} M_{i+1,j+1}
    = 
    \tilde x_{i,j} M_{i+1,j+1}.
\]
If $M_{i+1,j+1} \ne 0$ we conclude. Assume $M_{i+1,j+1} = 0$, and consider the
minor $M'$ obtained from $M_{i,j}$ by removing the row $i+1$ and the column
$j+1$ in $X$. Notice that $M' \leq M_{i+1,j+1}$, and by \cref{eq:hyp-rewrite}
this implies that $M'=0$. We construct similarly $\tilde M'$ from
$X(\Gamma_0,w(u,\beta))$, and we also get $\tilde M'=0$. If $M$ and $M'$ have
size $1\times 1$, we get $x_{i,j} = \tilde x_{i,j}= M = 0$ and we conclude.
Otherwise, expanding $M$ and $M'$ we get
\[
    x_{i,j} M_{i+2,j+2} + P'(x)
    = 
    \tilde x_{i,j} M_{i+2,j+2} + P'(\tilde x).
\]
Again the induction hypothesis implies that $P'(x) = P'(\tilde x)$. Now we can
repeat the same argument: we conclude if $M_{i+2,j+2} \ne 0$, and otherwise we
consider $M'' = \tilde M'' = 0$ by removing the row $i+2$ and the column $j+2$.
Eventually this process must stop, showing that $x_{i,j}=\tilde x_{i,j}$.
\end{proof}

\section{The generalized Nash problem for contact strata}

\label{sec:gen-nash}

In this section we start the analysis of the finer geometric structure of
contact strata. We consider the closures $\overline{\mathcal C}_\beta$ of
contact strata inside of the arc space of the Grassmannian, which we simply
call the \emph{closed contact strata}. We are mainly interested in the
following version of the \emph{generalized Nash problem}: 

\begin{problem}
    \label{problem:nash}
    Determine all possible containments among closed contact strata.
\end{problem}

As it often happens with Nash-type questions, this problems seems to be very
difficult. Nevertheless, we are able to show several types of containments
among closed contact strata, and that will be enough for the applications in
the rest of the paper.

\subsection*{The Plücker order}

We start with the easier direction: a necessary condition for a containment to
exist. Given a plane partition $\beta$, we consider the weighting $w(u,\beta)$
as in \cref{sec:valuations}, so that $\Lambda(\Gamma_0,w(u,\beta))$ is the
generic point of the contact stratum $\mathcal C_\beta$.

Now we consider a Plücker coordinate $[i_1 \ldots i_k]$. Then the number
\[
    \ord_\beta([i_1 \ldots i_k])
    =
    \ord_{\Lambda(\Gamma_0,w(u,\beta))}([i_1 \ldots i_k])
\]
is well-defined, in the sense that it only depends on $\beta$ and $[i_1 \ldots
i_k]$. Notice that there are combinatorial descriptions of $\Gamma_0$ and
$w(u,\beta)$, so one could implement an algorithm to compute these orders. We
use these numbers to define an order among plane partitions.

\begin{definition}
    We say that $\beta$ is less than or equal to $\beta'$ in the \emph{Plücker
    order}, written $\beta \trianglelefteq \beta'$, if
    \[
        \ord_\beta([i_1 \ldots i_k])
        \leq
        \ord_{\beta'}([i_1 \ldots i_k])
    \]
    for all Plücker coordinates $[i_1 \ldots i_k]$.
\end{definition}

The
following is obvious from the definitions.

\begin{lemma}
    \label{nash-trivial}
    Consider two plane partitions $\beta$ and $\beta'$. Then:
    \[
        \overline{\mathcal C}_{\beta} 
        \supseteq 
        \overline{\mathcal C}_{\beta'}
        \quad\Rightarrow\quad
        \beta \trianglelefteq \beta'.
    \]
    In particular, if $\alpha$ (resp.\ $\alpha'$) is the contact profile of any
    arc in $\mathcal C_{\beta}$ (resp.\ of any arc in $\mathcal C_{\beta'}$),
    then we have that
    \[
        \overline{\mathcal C}_{\beta} 
        \supseteq 
        \overline{\mathcal C}_{\beta'}
        \quad\Rightarrow\quad
        \alpha \leq \alpha'.
    \]
    Here $\alpha \leq \alpha'$ means that $\alpha_\lambda \leq \alpha'_\lambda$
    for all partitions $\lambda$.
\end{lemma}

The condition in \cref{nash-trivial} is not sufficient in general to guarantee
a containment of closed contact strata. An example of this is given in $G(3,6)$
by the following plane partitions:
\[
    \beta =
    \begin{pmatrix}
        3 & 2 & 1 \\
        2 & 1 & 1 \\
        1 & 1 & 0 \\
    \end{pmatrix},
    \qquad
    \beta' =
    \begin{pmatrix}
        2 & 2 & 1 \\
        2 & 2 & 1 \\
        1 & 1 & 0 \\
    \end{pmatrix}.
\]
It is possible to check (although quite tedious) that $\beta \trianglelefteq
\beta'$. But the partitions have the same number of boxes, and we will see in
\cref{codim} that this prevents the existence of a containment.

On the other hand, in the special case of $G(2,4)$ the Plücker order
completely characterizes containments.

\begin{proposition}
    \label{nash-2}
    Consider two contact strata $\mathcal C_\beta$ and $\mathcal C_{\beta'}$ in
    $J_\infty G(2,4)$. Then:
    \[
        \overline{\mathcal C}_{\beta} 
        \supseteq 
        \overline{\mathcal C}_{\beta'}
        \quad\Leftrightarrow\quad
        \beta \trianglelefteq \beta'.
    \]
\end{proposition}

\begin{proof}
Using the notation of \cref{sec:valuations}, we consider the matrix
$X = X(\Gamma_0,w(u,\beta))$. A direct computation gives the following:
\[
    X =
    \begin{pmatrix}
        u_{11} t^{\beta_{11}} + u_{12} u_{21} u_{22}
        t^{\beta_{12}+\beta_{21}-\beta_{22}}
        & u_{12} u_{22} t^{\beta_{12}}
    \\
        u_{21} u_{22} t^{\beta_{21}}
        & u_{22} t^{\beta_{22}}
    \end{pmatrix}.
\]
Therefore, the orders of contact with respect to the Plücker coordinates are:
\[
    \renewcommand*{\arraystretch}{1.6}
    \begin{array}{|l||l|l|l|l|l|l|}
    \hline
    [ij] & [12] & [13] & [14] & [23] & [24] & [34]
    \\[.1em] \hline
    \ord_\beta([ij])
    & \beta_{11} + \beta_{22}
    & \min\{\beta_{11}, \beta_{12} + \beta_{21} - \beta_{22}\}
    & \beta_{21}
    & \beta_{12}
    & \beta_{22}
    & 0
    \\ \hline
    \end{array}
\]
and the Plücker order is given by:
\[
    \beta \trianglelefteq \beta'
    \quad\Leftrightarrow\quad
    \bigg\{
    \begin{array}{c}
        \min\{\beta_{11}, \beta_{12} + \beta_{21} - \beta_{22}\}
        \leq \min\{\beta'_{11}, \beta'_{12} + \beta'_{21} - \beta'_{22}\},
        \\[.5em] \beta_{11} + \beta_{22} \leq \beta'_{11} + \beta'_{22},
        \quad \beta_{21} \leq \beta'_{21},
        \quad \beta_{12} \leq \beta'_{12},
        \quad \beta_{22} \leq \beta'_{22}.
    \end{array}
\]

From this explicit description, it is easy to determine the covers for the
Plücker order. In precise terms, if $\beta \vartriangleleft \beta'$, then we can
find a partition $\beta^*$ verifying $\beta \vartriangleleft \beta^*
\trianglelefteq \beta'$, and such that the difference $\beta^* - \beta$ is one
of the following matrices:
\[
    \begin{pmatrix}
        1 & 0 \\
        0 & 0
    \end{pmatrix},
    \quad
    \begin{pmatrix}
        0 & 1 \\
        0 & 0
    \end{pmatrix},
    \quad
    \begin{pmatrix}
        0 & 0 \\
        1 & 0
    \end{pmatrix},
    \quad
    \begin{pmatrix}
        0 & 0 \\
        0 & 1
    \end{pmatrix},
    \quad
    \begin{pmatrix}
       -1 & 1 \\
        0 & 1
    \end{pmatrix},
    \quad
    \begin{pmatrix}
       -1 & 0 \\
        1 & 1
    \end{pmatrix}.
\]
It is therefore enough to prove the \lcnamecref{nash-2} assuming $\beta'-\beta$
is one of the six matrices above. The first four cases are easy, and the last
two are transposed of each other, so we will focus on the last case:
\[
    \begin{array}{l@{\qquad\quad}l}
    \beta'_{11} = \beta_{11}-1,
    & \beta'_{12} = \beta_{12},
    \\[.5em] \beta'_{21} = \beta_{21} + 1,
    & \beta'_{22} = \beta_{22} + 1.
    \end{array}
\]
Observe that since $\beta \vartriangleleft \beta'$, we must have that
$\beta_{11} > \beta_{12}+\beta_{21}-\beta_{22}$.

Consider the matrix $X' = X(\Gamma_0,w(u',\beta'))$ giving the generic point of
$\mathcal C_{\beta'}$:
\[
    X' =
    \begin{pmatrix}
        x'_{11} & x'_{12} \\
        x'_{21} & x'_{22}
    \end{pmatrix}
    =
    \begin{pmatrix}
        u'_{11} t^{\beta_{11}-1} + u'_{12} u'_{21} u'_{22}
        t^{\beta_{12}+\beta_{21}-\beta_{22}}
        & u'_{12} u'_{22} t^{\beta_{12}}
    \\[.5em]
        u'_{21} u'_{22} t^{\beta_{21}+1}
        & u'_{22} t^{\beta_{22}+1}
    \end{pmatrix},
\]
and let $X(s)$ be the one-parameter family of matrices given by
\[
    X(s) = 
    \begin{pmatrix}
        x'_{11} & x'_{12} \\
        x'_{21} + s x'_{11} t^v
        & x'_{22} + s x'_{12} t^v
    \end{pmatrix}
\]
where $v = \beta_{22}-\beta_{12}$. Clearly $X(0) = X'$, and direct computation
shows that $X(s) \in \mathcal C_{\beta}$ for generic $s$. For this, the only
non-obvious thing to check is that
\begin{equation}
    \label{eq:nash-2-goal}
    \ord_t(x'_{21} + s x'_{11} t^v) = \beta_{21}.
\end{equation}
But we have that
\[
    \ord_t(x'_{21} + s x'_{11} t^v) =
    \min\{ \beta_{21}+1, \beta_{11}+\beta_{22}-\beta_{12}-1, \beta_{21} \},
\]
and hence \cref{eq:nash-2-goal} follows form the inequality
$\beta_{11} > \beta_{12}+\beta_{21}-\beta_{22}$.

We have constructed a family $X(s)$ whose special point is the generic point of
$\mathcal C_{\beta'}$ and whose generic point is in $\mathcal C_{\beta}$.
Therefore 
$\overline{\mathcal C}_{\beta} \supseteq {\mathcal C}_{\beta'}$, as required.
\end{proof}

\subsection*{Nash containments via weight exponents}

The easiest way to give a general sufficient condition for a containment
$\overline{\mathcal C}_{\beta} \supseteq \overline{\mathcal C}_{\beta'}$ is via
an analysis of the weight exponents.

\begin{theorem}
    \label{nash-c}
    Let $\beta$ and $\beta'$ be plane partitions, and let $c(\beta) =
    (c_{i,j})$ and $c(\beta') = (c'_{i,j})$ be the corresponding weight
    exponents (as in \cref{sec:planar-networks}). Assume that $c(\beta) \leq
    c(\beta')$ (i.e., that $c_{i,j} \leq c'_{i,j}$ for all $i,j$). Then
    $\overline{\mathcal C}_{\beta} \supseteq \overline{\mathcal C}_{\beta'}$.
\end{theorem}

\begin{proof}
In the notation of \cref{sec:valuations}, we write the generic points of
$\mathcal C_\beta$ and $\mathcal C_{\beta'}$ as
\[
    \Lambda(\Gamma_0,w)
    \qquad\text{and}\qquad
    \Lambda(\Gamma_0,w')
\]
where $w = w(u,\beta)$ and $w' = w(u',\beta')$ are the essential weightings of
$\Gamma_0$ given by
\[
    w_{i,j} = t^{c_{i,j}} u_{i,j}
    \qquad\text{and}\qquad
    w'_{i,j} = t^{c'_{i,j}} u'_{i,j}.
\]
Also, recall that each $u_{i,j}$ is a power series of the form
\[
    u_{i,j} = 
    u^{[0]}_{i,j} +
    u^{[1]}_{i,j}\, t +
    u^{[2]}_{i,j}\, t^2 +
    \cdots +
    u^{[p]}_{i,j}\, t^p +
    \cdots
\]
where the $u^{[p]}_{i,j}$ are variables. We have an analogous description of
$u'_{i,j}$.

The \lcnamecref{nash-c} will follow if we write a specialization of
$\Lambda(\Gamma_0,w)$ to $\Lambda(\Gamma_0,w')$. To do this, it is enough to
describe how each variable $u^{[p]}_{i,j}$ specializes to a function of the
$u'^{[p]}_{i,j}$. We consider the specialization given by
\[
    u^{[p]}_{i,j} 
    \mapsto
    \begin{cases}
        0 & \text{if $p < c'_{i,j} - c_{i,j}$,}\\
        u'^{[p-q]}_{i,j} & \text{if $p \geq q = c'_{i,j} - c_{i,j}$.}
    \end{cases}
\]
Observe that under this specialization we have
\[
    w_{i,j} = t^{c_{i,j}} u_{i,j}
    \quad\mapsto\quad
    t^{c'_{i,j}} u'_{i,j} = w'_{i,j},
\]
and therefore $\Lambda(\Gamma_0,w)$ specializes to $\Lambda(\Gamma_0,w')$, as
required.
\end{proof}

\subsection*{Nash containments via plateaux}

Given plane partitions $\beta \vartriangleleft \beta'$, we are going to give
some subtle conditions on the shapes of $\beta$ and $\beta'$ that guarantee the
existence of a containment $\overline{\mathcal C}_{\beta} \supsetneq
\overline{\mathcal C}_{\beta'}$. For this we need some definitions.

\begin{definition}
    We say that a plane partition $\beta$ has a \emph{plateau} up to position
    $(a,b)$ if there is a number $h$ such that $\beta_{i,j} = h$ when $i \leq
    a$, $j \leq b$, and $(i,j) \ne (a,b)$. The position $(a,b)$ is called the
    \emph{corner} of the plateau, and $h$ is called the \emph{height} of the
    plateau. If $(a,b) = (1,1)$ we set $h = \infty$. If the height is finite,
    the difference $h - \beta_{a,b}$ is called the \emph{fall}. If the height
    is infinite, we say that the fall is $0$ if $\beta_{a,b} = \infty$, and
    that it is $\infty$ if $\beta_{a,b} < \infty$.
\end{definition}

Notice that the above definition does not impose any condition on the entry in
the corner (the one in position $(a,b)$). Also, all plane partitions with
finite height have a plateau with infinite fall and corner at $(1,1)$. For some
examples see \cref{fig:plateau}.

\begin{figure}[ht]
    \centering
    \includegraphics{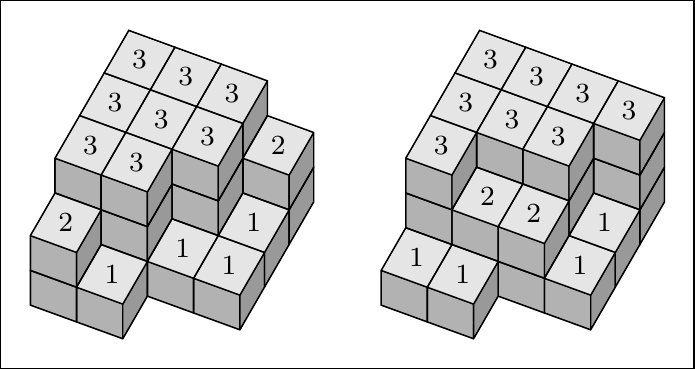}
    \caption{Some examples of plateaux in $G(4,8)$. The plane partition on the
    left has three plateaux with positive finite fall: their corners are in
    positions $(1,4)$, $(3,3)$, and $(4,1)$. On the right, $(2,4)$, $(3,2)$,
    and $(4,1)$ are the corners of the plateaux with positive finite fall. All
    of these plateaux have height $3$.}
    \label{fig:plateau}
\end{figure}

Let $\beta$ be a plane partition with a plateau with corner at $(a,b)$. Let $h$
and $f$ be the height and fall of the plateau. We want to understand the orders
of the weights in the network $(\Gamma_0, w(u,\beta))$. These are determined by
the weight exponents $c(\beta) = (c_{i,j})$. From the definitions we see that
the plateau imposes some conditions on $c(\beta)$. More precisely, if $(i,j)$
verifies
\[
    i \leq a,
    \qquad
    j \leq b,
    \qquad\text{and}\qquad
    (i,j) \not\in \{(a,b),(a-1,b),(a,b-1)\},
\]
then we have
\begin{align*}
    c_{i,j} &= 0 \qquad \text{if $(k-i) \ne (n-k-j)$, and} \\
    c_{i,j} &= h \qquad \text{if $(k-i) = (n-k-j)$.}
\end{align*}
We also have some information on the values of $c_{a-1,b}$ and $c_{a,b-1}$. We
have three possibilities:
\begin{equation}
    \label{eq:c-3pos}
    \begin{aligned}
    &
        (k-a) = (n-k-b)
        \quad \Rightarrow  \quad
        c_{a-1,b} = c_{a,b-1} = f,
    \\ &
        (k-a) < (n-k-b)
        \quad \Rightarrow  \quad
        c_{a-1,b} = f,
    \\ &
        (k-a) > (n-k-b)
        \quad \Rightarrow  \quad
        c_{a,b-1} = f.
    \end{aligned}
\end{equation}
In the cases $a=1$ or $b=1$, the above equations that involve $c_{a-1,b}$ and
$c_{a,b-1}$ are to be ignored.

We want to explain the consequences of having the plateau in terms of the
weighted planar network. We denote by $NW_{a,b}$ the part of the grid in
$\Gamma_0$ to the north-west of the vertex in position $(a,b)$. We let
$L_{a,b}$ be the open south-east corner of $NW_{a,b}$, containing the vertex in
position $(a,b)$, the vertical edge above $(a,b)$, and the horizontal edge to
the left of $(a,b)$. $L_{a,b}$ only contains the interior of the two edges, so
there is only one vertex in $L_{a,b}$. We set $NW_{a,b}^\circ = NW_{a,b}
\setminus L_{a,b}$.

The conditions found above on the weight exponents $c_{i,j}$ can be translated
in terms of the network as follows. Inside $NW_{a,b}^\circ$ all the edges have
weights of order $0$, and all the vertices have weights of order $0$ or $h$. The
edges of $L_{a,b}$ have weights of order either $0$ or $f$, corresponding to
the different cases of \cref{eq:c-3pos}. See \cref{fig:plateau-cons}.

\begin{figure}[ht]
    \centering
    \includegraphics{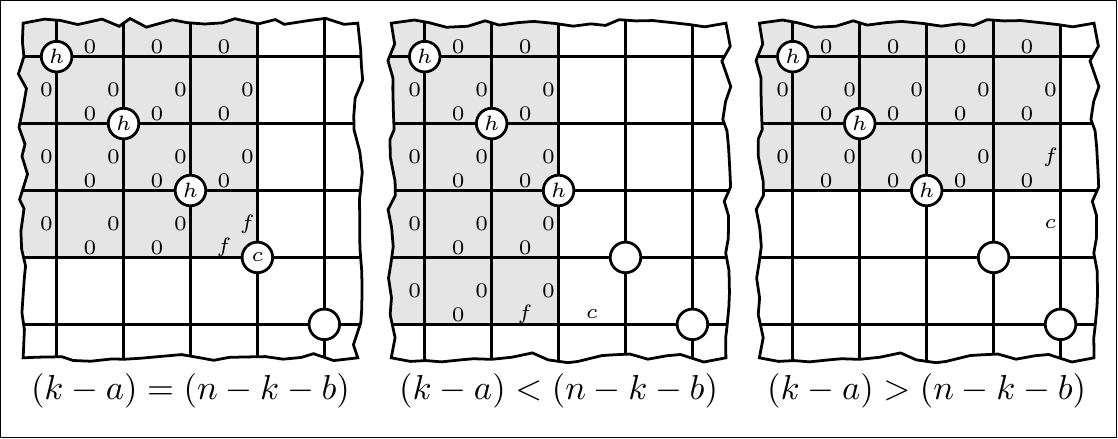}
    \caption{The consequences of having a plateau in $\beta$. The shaded region
    corresponds to $NW_{a,b}$. The edges and weighted vertices are labeled with
    the order of the weight. $h$ and $f$ are the height and fall of the
    plateau, and $c = c_{a,b}$.}
    \label{fig:plateau-cons}
\end{figure}

\begin{theorem}
    \label{nash-hard}
    Let $\beta$ be a plane partition having a plateau with corner at $(a,b)$.
    Assume the plateau has positive fall, and let $\beta'$ be the plane
    partition obtained by adding one box in position $(a,b)$. Then there is a
    containment $\overline{\mathcal C}_{\beta} \supsetneq \overline{\mathcal
    C}_{\beta'}$.
\end{theorem}

\begin{proof}
We will produce a wedge realizing the containment. That is, we will construct a
one-parameter family of arcs $\Lambda_s$ such that the generic point of the
family belongs to $\mathcal C_\beta$ and the special point of the family is the
generic point of $\mathcal C_{\beta'}$.

Let $s$ be a transcendental, meant to be the parameter of the family we are
about to construct.

We let $\Gamma_1$ be the planar network obtained from $\Gamma_0$ by adding one
diagonal edge joining $v_{a,b+1}$ with $v_{a+1,b}$. Here $v_{i,j}$ denotes the
internal vertex of $\Gamma_0$ in position $(i,j)$, $v_{i,n-k+1}$ is the source
labeled $i$, and $v_{k+1,j}$ is the sink labeled $j$. We construct a weighting
$w(s)$ of $\Gamma_1$ in the following way. The new diagonal edge of $\Gamma_1$
gets weight $s t^{c_{a,b}}$, where $c(\beta) = (c_{i,j})$ are the weight
exponents of $\beta$. The rest of edges of $\Gamma_1$, and all the vertices,
get the same weight as in $(\Gamma_0,w(u,\beta'))$. See \cref{fig:add-edge} for
the local structure of $(\Gamma_1,w(s))$ around the added diagonal edge.

\begin{figure}[ht]
    \centering
    \includegraphics{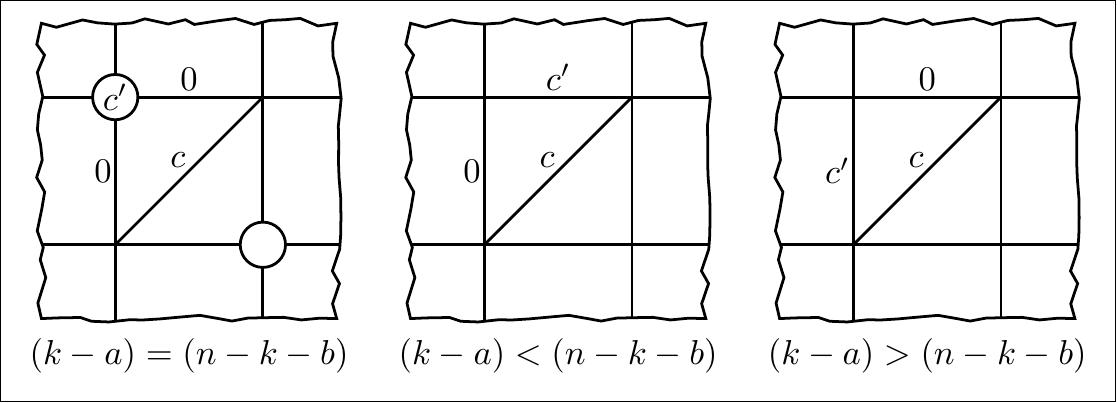}
    \caption{The local structure of $(\Gamma_1,w(s))$. Edges and weighted
    vertices are labeled with the order of the weight. $c = c_{a,b}$ and $c' =
    c'_{a,b}$ are weight exponents for $\beta$ and $\beta'$.}
    \label{fig:add-edge}
\end{figure}

We have a morphism
\[
    \mathbb A^1
    \longrightarrow
    J_\infty G(k,n),
    \qquad
    s \mapsto \Lambda(\Gamma_1, w(s)).
\]
When $s=0$, the added diagonal edge in $\Gamma_1$ gets weight $0$, and
therefore it does not affect the weight matrix. Hence $\Lambda(\Gamma_1,w(0)) =
\Lambda(\Gamma_0,w(u,\beta'))$. The \lcnamecref{nash-hard} follows if we show
that $\Lambda(\Gamma_1,w(s))$ belongs to $\mathcal C_\beta$ for generic $s$.

To lighten notation, we will simply denote by $w_p$ the weight of a path $p$ in
the weighted network $(\Gamma_0,w(u,\beta))$, and by $w'_q$ the weight of a
path in $(\Gamma_1,w(s))$. Notice that if the path $p$ is contained in
$\Gamma_0$, then $w'_p$ coincides with the weight of $p$ with respect to
$(\Gamma_0,w(u,\beta'))$.

To determine the invariant factor profile of $\Lambda(\Gamma_1,w(s))$ we first
compute the orders of its final minors. Let $M = M_{i_0,j_0}$ be a final minor
of $\Lambda(\Gamma_1,w(s))$ with row set $I=[i_0 \ldots i_r]$ and column set
$J=[j_0 \ldots j_r]$. Our goal is to show the following:
\begin{equation}
    \label{eq:nash-hard-goal}
    \ord_t(M)
    = \beta_{i_0,j_0} + \cdots + \beta_{i_r,j_r}.
\end{equation}

We know that there is a unique collection of paths $P = \{p_0, \ldots, p_r\}$
in $\Gamma_0$ connecting $I$ with $J$. These paths verify:
\[
    \ord_t(w'_{p_0}) = \beta'_{i_0,j_0},
    \quad
    \ord_t(w'_{p_1}) = \beta'_{i_1,j_1},
    \quad\ldots\quad
    \ord_t(w'_{p_r}) = \beta'_{i_r,j_r}.
\]
Also, recall that:
\begin{align*}
    & \beta'_{i,j} = \beta_{i,j} \quad\qquad \text{if $(i,j) \ne (a,b)$, and}
    \\ & \beta'_{a,b} = \beta_{a,b} + 1 .
\end{align*}

If none of the pairs $(i_0, j_0)$ equals $(a,b)$, then $P$ is also the unique
collection of paths in $\Gamma_1$ connecting $I$ with $J$. In this case it is
easy to show that \cref{eq:nash-hard-goal} holds:
\[
    \ord_t(M) = \ord_t(w'_P)
    = \beta'_{i_0,j_0} + \cdots + \beta'_{i_r,j_r}
    = \beta_{i_0,j_0} + \cdots + \beta_{i_r,j_r}.
\]

Assume that $(i_\ell,j_\ell) = (a,b)$ for some $0 \leq \ell \leq r$. Write $M =
\sum w'_Q$, where $Q = \{q_0, \ldots, q_r\}$ ranges among the collections of
paths in $\Gamma_1$ connecting $I$ with $J$. For each such $Q$ the last paths
agree with the ones of $P$, more precisely:
\begin{equation}
    \label{eq:nash-hard-step-1}
    p_{\ell+1} = q_{\ell+1}, \qquad
    p_{\ell+2} = q_{\ell+2}, \qquad
    \ldots\qquad
    p_r = q_r.
\end{equation}
There are two possibilities for $q_\ell$. It could be $q_\ell = p_\ell$, in
which case we also have $P=Q$. Otherwise $q_\ell$ is the path $\hat p_\ell$
obtained from $p_\ell$ by removing the north-west corner and adding the
diagonal edge of $\Gamma_1$. From the computation in \cref{fig:add-edge} we see
that the weight of $\hat p_\ell$ is given by:
\[
    \ord_t(w'_{\hat p_\ell}) = \ord_t(w'_{p_\ell}) - c'_{a,b} + c_{a,b} =
    \beta'_{a,b} - c'_{a,b} + c_{a,b} = \beta_{a,b}.
\]
We consider the collection $\hat P$ obtained from $P$ by replacing $p_\ell$
with $\hat p_\ell$. Observe that:
\[
    \ord_t(w'_{\hat P})
    = \ord_t(w'_P) - \beta'_{a,b} + \beta_{a,b}
    = \beta_{i_0,j_0} + \cdots + \beta_{i_r,j_r}.
\]
Therefore, in order to finish the proof of \cref{eq:nash-hard-goal}, it is
enough to show that
\begin{equation}
    \label{eq:nash-hard-goal-2}
    \ord_t(w'_Q) \geq \ord_t(w'_{\hat P}).
\end{equation}

If $q_\ell = p_\ell$, then $Q = P$ and \cref{eq:nash-hard-goal-2} follows from
the fact that $\beta'_{a,b} = \beta_{a,b} + 1$. We can therefore assume that
$q_\ell = \hat p_\ell$.

The paths $q_0,\ldots,q_{\ell-1}$ are contained in $\Gamma_0$, and agree with
the paths $p_0,\ldots,p_{\ell-1}$ away from the region $NW_{a,b}$. From the
discussion summarized in \cref{fig:plateau-cons}, we see that we have
\begin{equation}
    \label{eq:nash-hard-step-2}
    \ord_t(w'_{q_0}) = \ord_t(w'_{p_0}),\quad
    \ldots\quad
    \ord_t(w'_{q_{\ell-2}}) = \ord_t(w'_{p_{\ell-2}}).
\end{equation}
For $q_{\ell-1}$ there are two possibilities. We can have $q_{\ell-1} =
p_{\ell-1}$, which implies $Q = \hat P$. Otherwise $q_{\ell-1}$ is obtained
from $p_{\ell-1}$ by removing the north-west corner, and replacing it with the
corresponding south-west corner. Again using \cref{fig:plateau-cons}, we see
that
\[
    \ord_t(w'_{q_{\ell-1}}) =
    \ord_t(w'_{p_{\ell-1}}) + f'
\]
if $(k-a) \ne (n-k-b)$, and
\[
    \ord_t(w'_{q_{\ell-1}}) =
    \ord_t(w'_{p_{\ell-1}}) - h' + c'_{a,b} + 2f' 
\]
if $(k-a) = (n-k-b)$. Here $h' \geq 0$ and $f' \geq 0$ are the height and fall
of the plateau in the plane partition $\beta'$. Observe that when $(k-a) =
(n-k-b)$ we have that $h' = \beta'_{a-1,b-1}$, $c'_{a,b} = \beta'_{a,b}$, and
$f' = \beta'_{a-1,j-1} - \beta'_{a,b}$. In all instances, we see that
\begin{equation}
    \label{eq:nash-hard-step-3}
    \ord_t(w'_{q_{\ell-1}}) \geq
    \ord_t(w'_{p_{\ell-1}}).
\end{equation}
Combining \cref{eq:nash-hard-step-1,eq:nash-hard-step-2,eq:nash-hard-step-3} we
get \cref{eq:nash-hard-goal-2}, and therefore \cref{eq:nash-hard-goal} is
proven.

Recall that we need to prove that $\Lambda(\Gamma_1,w(s))$ has invariant factor
profile $\beta$. Using \cref{eq:nash-hard-goal}, this is amounts to showing
that
\[
    \ord_{\Lambda(\Gamma_1,w(s))}(\Omega_{(j^i)})
    =
    \ord_t(M_{i,j}),
\]
where $M_{i,j}$ is a final minor of $\Lambda(\Gamma_1,w(s))$. Equivalently, we
need to show that
\begin{equation}
    \label{eq:nash-hard-goal-3}
    \ord_t(M)
    \geq
    \ord_t(M_{i,j}),
\end{equation}
for all minors $M$ of $\Lambda(\Gamma_1,w(s))$ such that $M \leq M_{i,j}$ (in
the order of minors defined in \cref{sec:grass}). Notice that
\cref{eq:nash-hard-goal-3} is implied by a version of \cref{explicit-minors}
for the weighed network $(\Gamma_1,w(s))$. But the networks $\Gamma_0$ and
$\Gamma_1$ are very similar, and the proof of \cref{explicit-minors} can be
adapted easily to give the result that we need.
\end{proof}

\section{Log discrepancies}

In this section we compute log discrepancies for Schubert valuations. This will
be essential for later sections, where we study log canonical thresholds of
pairs involving Schubert varieties. The following definitions will be useful.

\begin{definition}
    Let $\beta$ be a plane partition. The \emph{floor} at level $s$ of $\beta$
    is the plane partition of height $1$ containing the boxes of $\beta$ at
    height $s$. More precisely, if we denote by $\mu^s$ such floor, we have
    \[
        \mu^s_{i,j} = 1
        \qquad\Leftrightarrow\qquad
        \beta_{i,j} \geq s.
    \]
    Notice that a plane partition of height $1$ is determined by its base, so
    we can also think of the $\mu^s$ as linear partitions. A plane partition is
    determined by its floors. In fact, given a nested sequence of plane
    partitions of height $1$ (or a nested sequence of linear partitions)
    \[
        \mu^1 \supseteq \mu^2 \supseteq \mu^2 \supseteq \cdots
        \supseteq \mu^h
    \]
    there is a unique plane partition $\beta$ having $\mu^s$ as the floor at
    level $s$. Notice that
    \[
        \beta = \mu^1 + \mu^2 + \mu^3 + \cdots
        + \mu^h
    \]
    Geometrically, we think of $\beta$ as obtained by stacking the floors on
    top of each other, with $\mu^1$ on the bottom and $\mu^h$ on top.
\end{definition}

\begin{definition}
    Let $\beta$ be a plane partition. The \emph{pillar} of $\beta$ in position
    $(i,j)$ is the collection of boxes which lay above the position $(i,j)$ in
    the plane. Notice that the number of boxes of the pillar in position
    $(i,j)$ is $\beta_{i,j}$, so a plane partition is determined by its
    pillars.
\end{definition}

\begin{proposition}
    \label{codim}
    Let $\beta = (\beta_{i,j})$ be a plane partition, possibly with infinite
    height. Then the codimension of $\mathcal C_\beta$ in $J_\infty G(k,n)$ is
    the number of boxes in $\beta$:
    \[
    \codim(\mathcal C_\beta, J_\infty G(k,n)) = \sum_{i,j} \beta_{i,j}
    \]
\end{proposition}

\begin{proof}
If $\beta$ has infinite height, then $\mathcal C_\beta \subset J_\infty
\Omega_{\lambda}$, where $\lambda$ is the (linear) partition marking the
infinite pillars of $\beta$ (i.e., the diagram of $\lambda$ contains the box in
position $(i,j)$ if and only if $\beta_{i,j} = \infty$). Since $J_\infty
\Omega_\lambda$ has infinite codimension in $J_\infty G(k,n)$, the
\lcnamecref{codim} follows in this case.

Assume that $\beta$ has finite height, let $h = \beta_{1,1}$ be this height,
and let $c = \sum \beta_{i,j}$ be the number of boxes in $\beta$. We let $N = h
k (n-k)$. We denote by $\beta^0$ the empty plane partition, and by $\beta^N$
the constant plane partition with height $h$ (i.e., $\beta^0_{i,j} = 0$ and
$\beta_{i,j}^N = h$ for all $i$,$j$). We will give an algorithm to construct a
sequence of nested plane partitions
\[
    \beta^0 \subset \beta^1 \subset \cdots \subset \beta^{c-1}
    \subset
    \beta^{c} \subset \beta^{c+1} \subset \cdots \subset \beta^N,
\]
such that $\beta^c = \beta$, and such that there are containments
$\overline{\mathcal C}_{\beta^r} \supsetneq \overline{\mathcal
C}_{\beta^{r+1}}$ for all $0 \leq r < N$.

The sequence $\beta^0 \subset \cdots \subset \beta^c$ corresponds to a process
of building the plane partition $\beta$ by adding one box at a time, and in
such a way that the boxes in lower floors are added before the boxes in higher
floors. To make this explicit, we order the pillars of a plane partition
lexicographically according to their positions. Assuming $\beta^r$ has been
constructed, we consider the boxes of $\beta$ which are not contained in
$\beta^r$. Among these, we select the one box which is in the lowest possible
floor and in the lexicographically smallest pillar. We add this box to
$\beta^r$ to get $\beta^{r+1}$. Notice that in this way the pairs $\beta^r
\subset \beta^{r+1}$ satisfy the hypotheses of \cref{nash-hard}: $\beta^{r+1}$
is obtained from $\beta^r$ by adding one box in the corner of a plateau with
fall $1$. In particular we get containments $\overline{\mathcal C}_{\beta^r}
\supsetneq \overline{\mathcal C}_{\beta^{r+1}}$.

The sequence $\beta^c \subset \cdots \subset \beta^N$ corresponds to a process
of building the plane partition $\beta^N$ starting from $\beta^c$ by adding one
box at a time. This can also be done compatibly with \cref{nash-hard}. For
example, we order the pillars of $\beta$ lexicographically, and at each step we
add one box to the pillar which is lexicographically smallest among those
having height less than $h$. Again, we get containments $\overline{\mathcal
C}_{\beta^r} \supsetneq \overline{\mathcal C}_{\beta^{r+1}}$.

For an example of how this algorithm works, see \cref{fig:codim-example}.

\newsavebox{\smlmat}
\savebox{\smlmat}{$\beta = \left(\begin{smallmatrix}3&2\\1&1\end{smallmatrix}\right)$}
\begin{figure}[ht]
    \centering
    \includegraphics{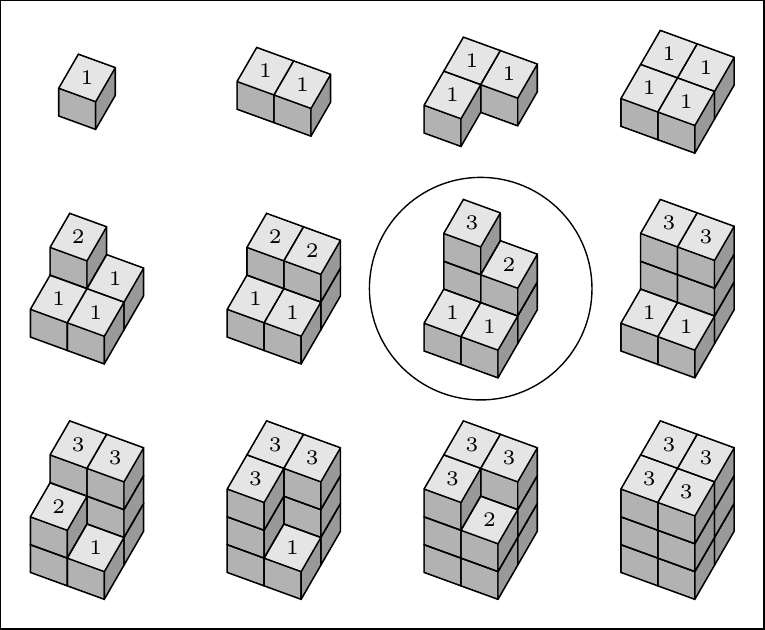}
    \caption{The algorithm of the proof of \cref{codim} for the plane partition
    \usebox\smlmat\ in $G(2,4)$.}
    \label{fig:codim-example}
\end{figure}

We have produced a nested sequence of closed contact strata:
\[
    \overline{\mathcal C}_{\beta^{0}}
    \supsetneq
    \overline{\mathcal C}_{\beta^{1}}
    \supsetneq
    \cdots
    \supsetneq
    \overline{\mathcal C}_{\beta^{N}}
\]
Since closed contact strata are irreducible (by \cref{irred}), we get lower
bounds
\[
    \codim \mathcal C_{\beta^r} \geq r.
\]

Let $\mu = ((n-k)^{k})$ be the linear partition whose diagram is the rectangle
of size $k\times(n-k)$. Then the Schubert variety $\Omega_\mu$ is just a point
(the Borel-fixed point), and it is easy to check that $\Cont^{\geq
h}(\Omega_\mu)$ is irreducible and
\[
    \codim \Cont^{\geq h}(\Omega_\mu) = h k (n-k) = N.
\]

Let $\Lambda$ be the generic point of $\Cont^{\geq h}(\Omega_\mu)$, and let
$\alpha$ be the contact profile of $\Lambda$. Notice that $\alpha$ is the
smallest possible contact profile of all arcs in $\Cont^{\geq h}(\Omega_\mu)$,
that is, the smallest contact profile for which $\alpha_\mu \geq h$. This
implies that $\alpha$ corresponds to the invariant factor profile $\beta^N$,
and therefore $\Cont^{\geq h}(\Omega_\mu) \subseteq \overline{\mathcal
C}_{\beta^N}$. In fact, from the above dimensional considerations, we get the
equality $\Cont^{\geq h}(\Omega_\mu) = \overline{\mathcal C}_{\beta^N}$.
Finally, this implies that $\codim \mathcal C_{\beta^r} = r$, and the
\lcnamecref{codim} follows.
\end{proof}

The following result is an immediate application of
\cref{eq:discrepancy-formula}.

\begin{corollary}
    \label{discrepancy}
    Let $\beta = (\beta_{i,j})$ be a plane partition with finite height. Let
    $\ord_\beta$ denote the corresponding Schubert valuation, $q_\beta$ its
    multiplicity, and $k_\beta = k_{\ord_\beta}(G(k,n))$ its discrepancy. Then
    \[
        k_\beta + q_\beta = \sum_{i,j} \beta_{i,j}.
    \]
\end{corollary}

\section{Log canonical thresholds}

\label{sec:lct}

In this section we study log canonical thresholds of pairs involving Schubert
varieties. As mentioned in the introduction, for this we need to introduce a
polytope, ${\rm SV}(k,n)$, which we call the \emph{polytope of normalized
Schubert valuations}. We start by expanding the discussion of the introduction,
and describe ${\rm SV}(k,n)$ in detail.

\subsection*{The cone of plane partitions}

A \emph{plane $\mathbb R$-partition} is a matrix $\beta = (\beta_{i,j})$ of
size $k \times (n-k)$, with real coefficients, and verifying the inequalities
\[
    \beta_{i,j} \geq 0,
    \qquad
    \beta_{i,j} \geq \beta_{i+1,j},
    \qquad\text{and}\qquad
    \beta_{i,j} \geq \beta_{i,j+1},
\]
whenever they make sense. We denote by $\mathbb R{\rm PP}(k,n) \subset \mathbb
R^{k(n-k)}$ the set of plane $\mathbb R$-partitions, and we write ${\rm
PP}(k,n) = \mathbb R{\rm PP}(k,n) \cap \mathbb Z^{k(n-k)}$ for its intersection
with the standard lattice. Notice that the elements of ${\rm PP}(k,n)$ are the
(usual) plane partitions. 

By definition, $\mathbb R{\rm PP}(k,n)$ is a pointed rational convex polyhedral
cone, with vertex at the origin (corresponding to the empty plane partition).
It is the convex hull of ${\rm PP}(k,n)$.

Recall that given a linear partition $\mu$, there is a unique plane partition
of height $1$ having $\mu$ as its base. We will use the same notation, $\mu$,
for this one-floor plane partition. We consider what we call \emph{chains of
floors}: these are sets of non-empty floors which are totally ordered with
respect to inclusion. We write them in the form $\mu^\bullet = \{ \mu^1
\supsetneq \mu^2 \supsetneq \dots \supsetneq \mu^h \}$, and call
$\ell(\mu^\bullet) = h$ the \emph{length} of the chain $\mu^\bullet$.

Given a chain of floors $\mu^\bullet$ of length $h$ and non-negative integers
$a_1, a_2, \ldots, a_h$, we have a plane partition $\beta \in {\rm PP}(k,n)$
given by
\[
    \beta = a_1 \mu^1 + a_2 \mu^2 + \cdots + a_h \mu^h.
\]
Moreover, any plane partition whose set of floors is contained in $\mu^\bullet$
has a unique expression of the above form. We denote by ${\rm PP}(\mu^\bullet)$
the set of plane partitions obtained in this way, and by $\mathbb R{\rm
PP}(\mu^\bullet)$ the corresponding convex cone.

The following facts follow from straightforward computations:
\begin{enumerate}
    \item A chain of floors $\mu^\bullet$ can always be completed to a basis of
        the lattice $\mathbb Z^{k(n-k)}$. In particular the cones $\mathbb
        R{\rm PP}(\mu^\bullet)$ are non-singular (in the sense of
        \cite[\S2.1]{Ful93}), and of dimension $\ell(\mu^\bullet)$.
   \item Intersections of cones correspond to intersections of chains: given
       two chains $\mu^\bullet$ and $\mu^\bullet$, we have $\mathbb R{\rm
       PP}(\mu^\bullet) \cap \mathbb R{\rm PP}(\nu^\bullet) = \mathbb R{\rm
       PP}(\mu^\bullet \cap \nu^\bullet)$. In particular, $\mathbb R{\rm
       PP}(\nu^\bullet) \subseteq \mathbb R{\rm PP}(\mu^\bullet)$ if and only
       if $\nu^\bullet \subseteq \mu^\bullet$.
    \item Given a floor $\mu$, the one-dimensional cone $\mathbb R{\rm
        PP}(\{\mu\}) = \mathbb R_{\geq 0}\cdot \mu$ is an extremal ray of
        $\mathbb R{\rm PP}(k,n)$.
\end{enumerate}

The collection of cones $\mathbb R{\rm PP}(\mu^\bullet)$, where $\mu^\bullet$
ranges among all chains of floors, gives a non-singular fan whose support is
$\mathbb R{\rm PP}(k,n)$. The one-dimensional cones in this fan are exactly the
extremal rays of $\mathbb R{\rm PP}(k,n)$.

\subsection*{The polytope of normalized Schubert valuations}

Given a plane $\mathbb R$-partition $\beta \in \mathbb R{\rm PP}(k,n)$, we
denote by $|\beta| = \sum \beta_{i,j}$ the sum of the entries in $\beta$, and
call it the \emph{volume} of $\beta$. Notice that the volume of an element in
${\rm PP}(k,n)$ agrees with the log discrepancy of the corresponding valuation.
We set
\[
    {\rm SV}(k,n)
    =
    \{\,\, 
    \beta \in \mathbb R{\rm PP}(k,n)
    \,\,\mid\,\,
    |\beta|=1
    \,\,\},
\]
and call it the \emph{polytope of normalized Schubert valuations}. Analogously,
if $\mu^\bullet \ne \emptyset$ is a non-empty chain of floors, we set
\[
    {\rm SV}(\mu^\bullet)
    =
    \{\,\, 
    \beta \in \mathbb R{\rm PP}(\mu^\bullet)
    \,\,\mid\,\,
    |\beta|=1
    \,\,\}.
\]

Using the above description of $\mathbb R{\rm PP}(k,n)$ as the support of a
non-singular fan, we immediately get a simplicial structure on ${\rm SV}(k,n)$.
More precisely, ${\rm SV}(k,n)$ is a bounded rational convex polytope whose
extremal points are of the form $\mu/|\mu|$, where $\mu$ ranges among all
non-empty linear partitions with at most $k$ parts of size at most $n-k$. For a
non-empty chain of floors $\mu^\bullet$, the polytope ${\rm SV}(\mu^\bullet)$
is a simplex of dimension $\ell(\mu^\bullet)-1$, and its faces correspond to
sub-chains $\nu^\bullet \subseteq \mu^\bullet$. This collection of simplices
endows ${\rm SV}(k,n)$ with the structure of a simplicial complex. See
\cref{fig:polytope} for an example.

We let ${\rm Bru}^*(k,n)$ denote the poset of partitions with at most $k$ parts
of size $n-k$ endowed with the Bruhat order (containment of partitions), and
consider ${\rm Bru}(k,n) = {\rm Bru}^*(k,n) \setminus \{\emptyset\}$. Then the
simplicial complex ${\rm SV}(k,n)$ coincides with the nerve (in the sense of
category theory) of the poset ${\rm Bru}(k,n)$. But notice that ${\rm SV}(k,n)$
is not just an abstract simplicial complex, it has a natural geometric
realization embedded in $\mathbb R^{k(n-k)}$.

\subsection*{The Arnold multiplicity}

Fix a Schubert variety $\Omega_\lambda$, and let $(b_1^{a_1})$, \ldots,
$(b_r^{a_r})$ be the Schubert conditions of $\lambda$ (as defined before
\cref{schubert-ideals}). Notice that $(a_1,b_1)$, \ldots, $(a_r,b_r)$ are the
South-East corners of the diagram of $\lambda$. It follows from
\cref{ess-cont-prof} and \cref{def:if-prof} that
\begin{equation}
    \label{eq:ord-lambda}
    \ord_\beta(\Omega_\lambda)
    =
    \min_{s=1 \ldots r}
    \{
        \beta_{a_s,b_s}
        + \beta_{a_s+1,b_s+1}
        + \beta_{a_s+2,b_s+2}
        + \cdots
    \}
\end{equation}
for any plane partition $\beta$. In the above formula, for each $s$ we are
summing the entries of $\beta$ corresponding to the half-diagonal emanating
from the corner $(a_s,b_s)$. See \cref{fig:diagonals} for some examples.

We use \cref{eq:ord-lambda} to define $\ord_\beta(\Omega_\lambda)$ when $\beta$
is a plane $\mathbb R$-partition. We obtain a function on $\mathbb R{\rm
PP}(k,n)$, which we denote $\ord(\lambda)$:
\[
    \ord(\lambda)\colon \mathbb R{\rm PP}(k,n) \to \mathbb R,
    \qquad
    \beta \mapsto \ord(\lambda)(\beta) = \ord_\beta(\Omega_\lambda).
\]
By construction $\ord(\lambda)$ is a concave piecewise-linear function on
$\mathbb R{\rm PP}(k,n)$. We denote by $H_\lambda$ the biggest linear subspace
of $\mathbb R^{k(n-k)}$ where $\ord(\lambda)$ is linear. This is the biggest
linear subspace contained in the corner locus of $\ord(\lambda)$, and its
equations are
\[
    \beta_{a_s,b_s}
    + \beta_{a_s+1,b_s+1}
    + \beta_{a_s+2,b_s+2}
    + \cdots
    =
    \beta_{a_{s'},b_{s'}}
    + \beta_{a_{s'}+1,b_{s'}+1}
    + \beta_{a_{s'}+2,b_{s'}+2}
    + \cdots
\]
where $s$ and $s'$ range in $\{1, \ldots, r\}$. Observe that when $\lambda$ is
rectangular, $\ord(\lambda)$ is linear, and therefore $H_\lambda = \mathbb
R^{k(n-k)}$. See \cref{fig:diagonals} for some examples.

\begin{theorem}
    \label{arnold}
    Let $\Omega_\lambda$ be a Schubert variety in $G(k,n)$. Then the Arnold
    multiplicity of the pair $(G(k,n),\Omega_\lambda)$ is the maximum of
    $\ord(\lambda)$ in ${\rm SV}(k,n) \cap H_\lambda$.
\end{theorem}

Notice that $\ord(\lambda)$ is a linear function on the convex polytope ${\rm
SV}(k,n) \cap H_\lambda$, and therefore the Arnold multiplicity is achieved at
one of the extremal points of the polytope.

\begin{proof}
Notice that the contact loci $\Cont^{\geq p}(\Omega_\lambda)$ are unions of
contact strata, and in particular their irreducible components are closed
contact strata. Therefore, using \cref{eq:arnold-formula} we see that the
Arnold multiplicity is given by
\[
    \operatorname{Arnold-mult}(G(k,n), \Omega_\lambda) = 
    \max_\beta \left\{ \frac{\ord_\beta(\Omega_\lambda)}{\codim(\mathcal C_\beta,J_\infty G(k,n))} \right\}
    =
    \max_\beta \left\{ \frac{\ord(\lambda)(\beta)}{|\beta|} \right\},
\]
where $\beta$ ranges among all plane partitions in ${\rm PP}(k,n)$. If $\beta
\notin H_\lambda$, it is possible to decrease the number of boxes in $\beta$
without changing $\ord(\lambda)(\beta)$, and we see that the maximum is
achieved in ${\rm PP}(k,n) \cap H_\lambda$. Observe that $\ord(\lambda)$ is
homogeneous:
\[
    \frac{\ord(\lambda)(\beta)}{|\beta|}
    =
    \ord(\lambda)\left(\frac{\beta}{|\beta|}\right).
\]
Therefore the Arnold multiplicity is the maximum of $\ord(\lambda)$ on ${\rm
SV}(k,n) \cap H_\lambda \cap \mathbb Q^{k(n-k)}$. The \lcnamecref{arnold} now
follows from the fact that ${\rm SV}(k,n) \cap H_\lambda$ is a rational
polytope.
\end{proof}

\subsection*{The rectangular case}

\cref{arnold} can be improved slightly when $\lambda$ is rectangular. That is
the content of \cref{intro-lct-sq}.

\begin{proof}[Proof of \cref{intro-lct-sq}]
Let $\lambda = (b^a)$ be a rectangular partition. Then $H_\lambda = \mathbb
R^{k(n-k)}$, and \cref{arnold} says that the Arnold multiplicity (and therefore
the log canonical threshold) is achieved at one of the extremal points of ${\rm
SV}(k,n)$. These extremal points are of the form $\mu/|\mu|$, where $\mu$ is a
floor (determined by a linear partition). The \lcnamecref{intro-lct-sq} follows
if we show that we can restrict $\mu$ to be in the set $\lambda^0$, \ldots,
$\lambda^r$, where $\lambda^s = ((b+s)^{a+s})$ and $r=\min\{k-a,n-k-b\}$.

Let $\mu$ be arbitrary, and let $s$ be the biggest index such that $\lambda^s
\subset \mu$. Then $\ord(\lambda)(\mu) = \ord(\lambda)(\lambda^s)$ and $|\mu|
\geq |\lambda^s|$, and the \lcnamecref{intro-lct-sq} follows.
\end{proof}

As we saw in \cref{gens-rect}, the ideal of a Schubert variety of rectangular
shape is essentially equivalent to the ideal of a generic determinantal
variety. The log canonical thresholds of determinantal varieties were first
computed in \cite{Joh03}. For an approach to the singularities of generic
determinantal varieties using arcs, see \cite{Doc13}.

\begin{figure}[ht]
    \centering
    \parbox{11cm}{
    \begin{framed}
        $\begin{array}{lll}
            \text{Maximize:} 
            & \beta_{1,4} + \beta_{2,5}
            &
        \\[.6em]
            \text{Subject to:} 
            & \beta_{1,4} + \beta_{2,5} = \beta_{2,2} + \beta_{3,3}
            &
        \\
            & \beta_{1,4} + \beta_{2,5} = \beta_{3,1}
            &
        \\[.6em]
            & \sum_{i=1 \ldots 3 ~ j=1 \ldots 5} \beta_{i,j} = 1
            &
        \\[.6em]
            & \beta_{i,j} \geq \beta_{i+1,j} 
            \quad \forall i \in \{1,2\}
            & \forall j \in \{1,2,3,4,5\}
        \\
            & \beta_{i,j} \geq \beta_{i,j+1} 
            \quad \forall i \in \{1,2,3\}
            & \forall j \in \{1,2,3,4\}
        \\[.6em]
            &\multicolumn{2}{l}{\begin{array}{@{}lllll}
                  \beta_{1,1} \geq 1
                & \beta_{1,2} \geq 1
                & \beta_{1,3} \geq 1
                & \beta_{1,4} \geq 1
                & \beta_{1,5} \geq 0
            \\
                  \beta_{2,1} \geq 1
                & \beta_{2,2} \geq 1
                & \beta_{2,3} \geq 0
                & \beta_{2,4} \geq 0
                & \beta_{2,5} \geq 0
            \\
                  \beta_{3,1} \geq 1
                & \beta_{3,2} \geq 0
                & \beta_{3,3} \geq 0
                & \beta_{3,4} \geq 0
                & \beta_{3,5} \geq 0
            \end{array}}
        \end{array}$
    \end{framed}%
    \vspace{-.75em}%
    }
    \caption{The linear program for $\lambda = (421)$ in $G(3,8)$.}
    \label{fig:linear-program}
\end{figure}

\subsection*{Linear programming} \cref{arnold} does not give a closed formula
for the Arnold multiplicity. To get an actual value, one needs to solve a
linear programming problem: maximize the linear function $\ord(\lambda)$ on the
polytope ${\rm SV}(k,n) \cap H_\lambda$. In the present case, we believe this
is a task better left to a computer. The equations defining ${\rm SV}(k,n) \cap
H_\lambda$ and $\ord(\lambda)$ are easy to describe to a computer, and the
complexity of the resulting linear program, which is high when approached
manually, is perfectly manageable by modern machines. See
\cref{fig:linear-program} for an example of the input that would be fed to a
linear programming solver. Notice that in \cref{fig:linear-program} we have
added the constraints $\beta_{i,j} \geq 1$ for $(i,j) \in \lambda$. This is
done so the computer does not need to spend time searching for an initial
extremal point of the polytope, and can focus on just maximizing the objective
function.

In small cases, it is possible to run a linear programming solver by hand. We
sketch the idea of the standard algorithm (the simplex method), which is
straightforward. We start with $\beta^1$, the one-floor plane partition with
base $\lambda$. Notice that $\beta^1/|\beta^1|$ is an extremal point of ${\rm
SV}(k,n) \cap H_\lambda$. Assuming we have constructed $\beta^s$, we try to
find $\beta^{s+1}$ verifying
\begin{equation}
    \label{eq:candidate-cond}
   \frac{|\beta^{s+1}|}{s+1}
   <
   \frac{|\beta^{s}|}{s}.
\end{equation}
The possible candidates $\beta^{s+1}$ are obtained from $\beta^s$ by first
adding one box to each of the half-diagonals determined by $\lambda$ (see
\cref{fig:diagonals}), and then completing with more boxes away from the
half-diagonals in order to obtain a plane partition. If none of the candidates
verifies \cref{eq:candidate-cond}, we stop the algorithm and the log canonical
threshold is $|\beta^s|/s$. Otherwise we choose $\beta^{s+1}$ such that
$|\beta^{s+1}|/(s+1)$ is minimal among all the candidates. In this process,
$\beta^s/|\beta^s|$ iterates among extremal points of the polytope ${\rm
SV}(k,n) \cap H_\lambda$, and therefore the algorithm finishes after a finite
number of steps. An example of the execution of this algorithm appears in
\cref{fig:lp-example}.

\begin{figure}[ht]
    \centering
    \includegraphics{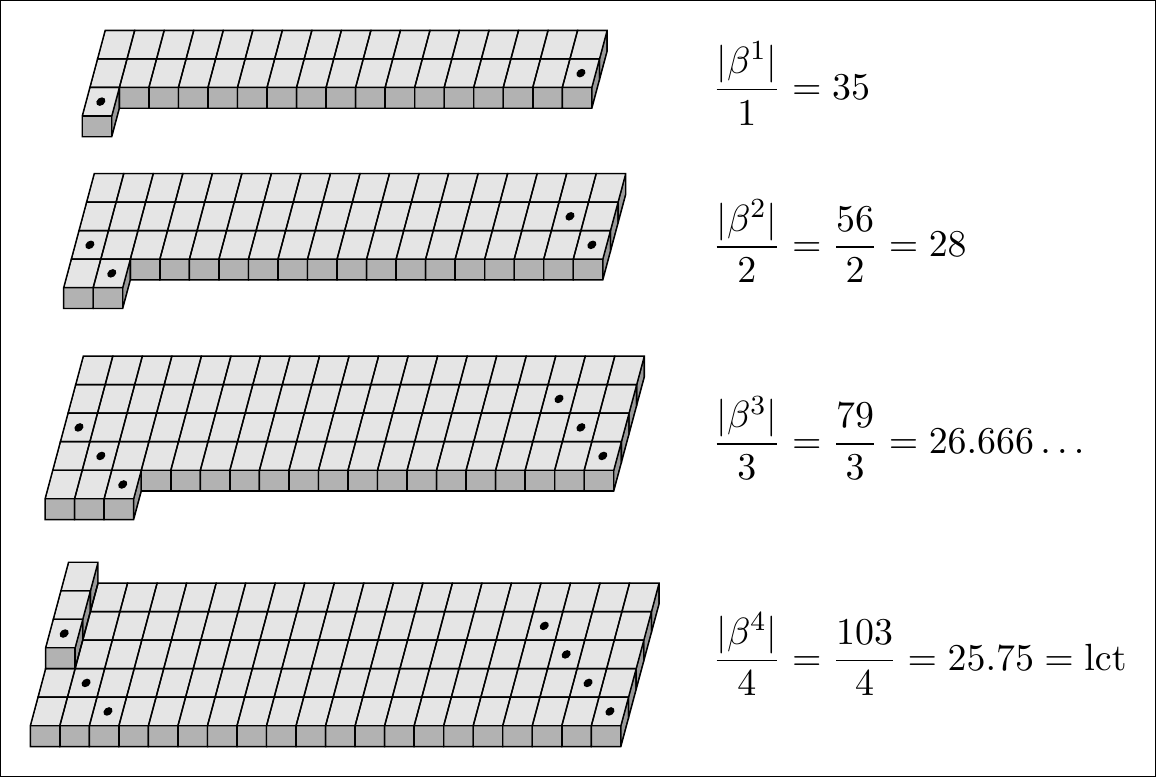}
    \caption{Plane partitions visited by the linear programming algorithm for
    $\lambda = (17,17,1)$ in any $G(k,n)$ with $k\geq 5$ and $n-k \geq 20$.}
    \label{fig:lp-example}
\end{figure}

From the above description of the algorithm, we immediately get log canonical
thresholds of many partitions ``with small number of boxes''. To make this
precise, consider a partition $\lambda$ with at most $k-1$ parts of size at
most $n-k-1$. Then let $r(\lambda)$ be the \emph{rim size} of $\lambda$: the
number of boxes in the rectangle $k\times(n-k)$ which touch $\lambda$ (possibly
just in a vertex) but are not contained in $\lambda$. Then, if $|\lambda| \leq
r(\lambda)$, one can see that the algorithm stops at $\beta^1$. Therefore: 
\[
    \lct(G(k,n),\Omega_\lambda) = |\lambda|
    \quad
    \Leftrightarrow
    \quad
    |\lambda| \leq r(\lambda).
\]
Notice that $|\lambda| = \codim(\Omega_\lambda,G(k,n))$ is the maximal possible
value for the log canonical threshold. In a sense, the Schubert varieties
$\Omega_\lambda$ for which $|\lambda| \leq r(\lambda)$ are the least singular.

\subsection*{Log resolutions} An alternative, more direct, approach for the
computation of log canonical thresholds would be to use the definition with log
resolutions. Unfortunately, we do not know a usable description of log
resolutions for all pairs $(G(k,n),\Omega_\lambda)$.\footnote{Notice that log
resolutions for the Schubert varieties themselves are well-known: they are
given by a construction of Bott and Samelson (for a nice description in modern
language, see \cite{AM09}). But it seems that the case of log resolutions for
pairs $(G(k,n),\Omega_\lambda)$ has not been studied.}

A natural candidate would be the one provided in \cite{Bou93}. The construction
resembles the one for generic determinantal varieties. We start with the
Grassmannian $X_0 = G(k,n)$ and we let $X_1$ be the blowing-up of $X_0$ along
the Schubert point (the Borel-fixed point). Then $X_2$ is the blowing-up of
$X_1$ along the strict transform of the one-dimensional Schubert variety. In
$X_2$, the strict transforms of the two-dimensional Schubert varieties are
smooth and disjoint, so we can blow them up (in any order) and obtain $X_3$.
Recursively, $X_{s+1}$ is obtained from $X_s$ by blowing up the strict
transforms of the $s$-dimensional Schubert varieties (which, as is shown in
\cite{Bou93}, are smooth and disjoint in $X_s$). At the end we obtain $X =
X_{k(n-k)-1}$. The variety $X$ has $\binom{n}{k}-1$ exceptional divisors, each
one corresponding to a Schubert variety in $G(k,n)$ (except $G(k,n)$ itself).
We have not worked this out in complete detail, but there is strong evidence
suggesting that the polytope ${\rm SV}(k,n)$ is the dual complex of the
resolution $X \to G(k,n)$.

In \cite{Bou93} it is not studied whether $X$ is a log resolution for all
possible pairs: we do not know if the exceptional locus is a simple normal
crossings divisor, and if the scheme-theoretic inverse image of
$\Omega_\lambda$ is a divisor. In fact, it seems that $X$ is not a log
resolution, at least for some of Schubert varieties. If it were, the log
canonical threshold would be computed by one of the valuations corresponding to
the exceptional divisors in $X$. In terms of the arc space, these valuations
correspond to one-floor plane partitions. But there are Schubert varieties for
which the log canonical threshold is computed by a plane partition with more
than one floor. We already saw one example in \cref{fig:lp-example}:
$\Omega_{(17,17,1)}$ in $G(5,25)$. A smaller example is $\Omega_{(54441)}$ in
$G(5,10)$.

It would be interesting to construct a resolution of $G(k,n)$ which is a
simultaneous log resolution for all pairs $(G(k,n), \Omega_\lambda)$. The dual
complex of such resolution would be a simplicial subdivision of ${\rm
SV}(k,n)$, in such a way that all the slices ${\rm SV}(k,n) \cap H_\lambda$
would be simplicial sub-complexes.

\begin{remark}[dlt models] As pointed out by the referee, it is plausible that
the above $X$ gives a dlt model for all pairs $(G(k,n), \Omega_\lambda)$. In
particular, ${\rm SV}(k,n)$ should be related to the polytopes constructed in
\cite{NX} or \cite{dFKX}, and the techniques of those papers might help for the
computations of log canonical thresholds. It would be very interesting to
explore this possible connection. 
\end{remark}

\section{The Nash problem for Schubert varieties}

In this section we study the Nash problem for Schubert varieties in the
Grassmannian. The results here are largely independent from the rest of the
paper, as it turns out that the structure of the arc space plays a small role
towards the solution of the Nash problem. 

The main tool that we use are certain resolutions of singularities of Schubert
varieties. We show that their exceptional components are in bijection with the
irreducible components of the exceptional locus. 

\begin{figure}[ht]
    \centering
    \includegraphics{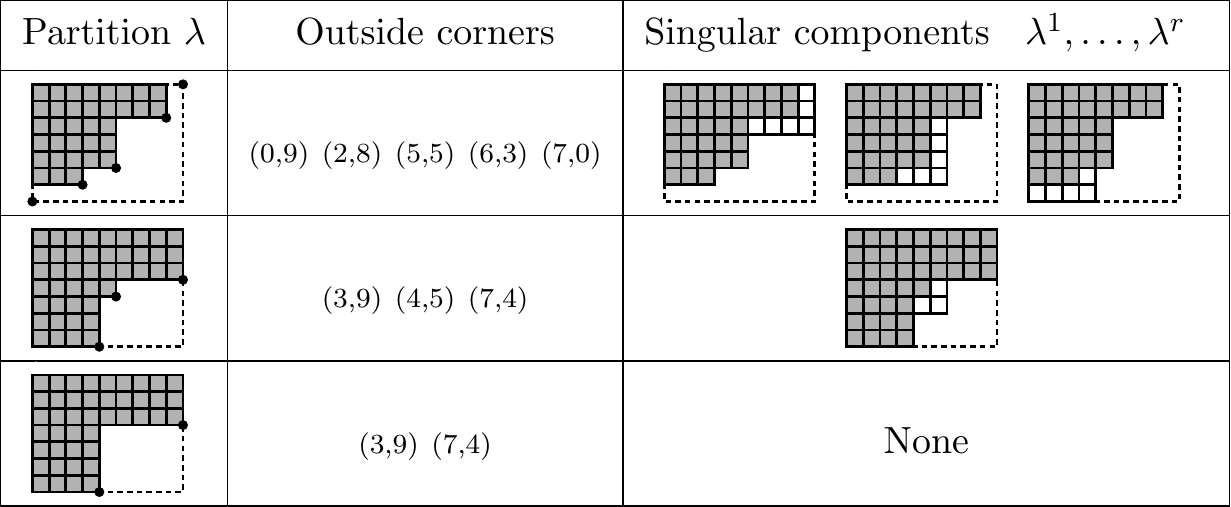}
    \caption{Some examples of outside corners of partitions in $G(7,16)$.}
    \label{fig:corners}
\end{figure}

\subsection*{The singular locus} 

Before studying resolutions, we need to recall what is the singular locus of a
Schubert variety, and fix some notations. We let $\Omega_\lambda$ be a Schubert
variety in $G(k,n)$. A \emph{proper outside corner} of $\lambda$ is a pair
$(a,b)$, such that the partition $(b^a)$ is a Schubert condition of $\lambda$
(as it was discussed in \cref{sec:grass}). In other words, we say that $(a,b)$
is a proper outside corner of $\lambda$ if the rectangle of size $a \times b$
is a maximal sub-rectangle of the diagram of $\lambda$. If no proper outside
corner of $\lambda$ is of the form $(a,n-k)$, we say that $(0,n-k)$ is a
\emph{virtual outside corner} of $\lambda$. Analogously, $(k,0)$ is a
\emph{virtual outside corner} of $\lambda$ if no proper outside corner is of
the form $(k,b)$. An \emph{outside corner} of $\lambda$ is either a proper
outside corner or a virtual outside corner of $\lambda$.

With $\lambda$ fixed, we denote the outside corners of $\lambda$ by:
\[
    (a_0, b_0),
    \quad
    (a_1, b_1),
    \quad
    \ldots,
    \quad
    (a_r, b_r),
    \quad
    (a_{r+1}, b_{r+1}).
\]
Here we assume that the corners are ordered from North-East to South-West, that
is:
\[
    a_0 < a_1 < \cdots < a_r < a_{r+1},
    \quad\text{and}\quad
    b_0 > b_1 > \cdots > b_r > b_{r+1}.
\]
These outside corners determine completely the partition $\lambda$, and
therefore the Schubert variety $\Omega_\lambda$. In fact, as we saw in
\cref{schubert-ideals}, a point $V \in G(k,n)$ belongs to $\Omega_\lambda$
precisely when
\[
    \dim V \cap F_{n-k+a_s-b_s} \geq a_s
\]
for all $0 \leq s \leq r+1$. Here $F_\bullet$ is the complete flag fixed by the
Borel subgroup, as usual.

For $1 \leq s \leq r$, we construct a partition $\lambda^s$ from $\lambda$ by
adding a rim of boxes around the corner $(a_s, b_s)$. More precisely, the
diagram of $\lambda^s$ is the union of the diagram of $\lambda$ and the diagram
of the rectangular partition with proper outside corner $(a_s+1, b_s+1)$.
Notice that when $r=-1$ (which only happens when $\lambda = ((n-k)^k)$) and
when $r=0$ we do not construct any partition $\lambda^s$. For examples, see
\cref{fig:corners}.

Notice that the Borel subgroup acts on $\Omega_\lambda$, and therefore the
singular locus ${\rm Sing}(\Omega_\lambda)$ must be Borel-invariant, i.e., it
is a union of Schubert varieties. The next theorem identifies the irreducible
components of ${\rm Sing}(\Omega_\lambda)$ as the Schubert varieties given by
the partitions $\lambda^1, \ldots, \lambda^r$ defined above. For a proof we
refer the reader to \cite[Thm.~5.3]{LW90}, or to \cite[Sec.~6.B]{BV88}.

\begin{theorem}
    \label{sing-locus}
    With the notations introduced above, the singular locus of the Schubert
    variety $\Omega_\lambda$ is given by:
    \[
        {\rm Sing}(\Omega_\lambda) 
        = \Omega_{\lambda^1} \cup \cdots \cup \Omega_{\lambda^r}.
    \]
    In particular, $\lambda$ is smooth if and only if $r \leq 0$, and otherwise
    ${\rm Sing}(\Omega_\lambda)$ has $r$ irreducible components. In general, a
    point $V \in G(k,n)$ belongs to the smooth locus of $\Omega_\lambda$ if and
    only if
    \[
        \dim V \cap F_{n-k+a_s-b_s} = a_s
    \]
    for each outside corner $(a_s, b_s)$ of $\lambda$.
\end{theorem}

\subsection*{Resolution of singularities}

We now describe a resolution of singularities of $\Omega_\lambda$. The
construction is well-known to the experts, and it appears for example in
\cite{Zel83}. But notice that the main goal of \cite{Zel83} is to show that
$\Omega_\lambda$ admits small resolutions in the sense of intersection
homology. For this, Zelevinsky gives a construction of several resolutions of
singularities, including the one that we discuss below. For
our purposes, the smallness of the resolution has no relevance, so we will
focus in a particular case, which is easy to describe.

We consider the manifold ${\rm Fl}(a_1, a_2, \ldots, a_r, k; n)$ of partial
flags in $\mathbb C^n$ of the form
\[
    U_1 \subset U_2 \subset \cdots \subset U_r \subset V \subset \mathbb C^n,
\]
where
\[
    \dim U_s = a_s,
    \qquad\text{and}\qquad
    \dim V = k.
\]
We let $Y \subset {\rm Fl}(a_1, a_2, \ldots, a_r, k; n)$ be the subvariety
corresponding to those flags that verify
\[
    U_s \subseteq F_{n-k+a_s-b_s}
\]
for all $1 \leq s \leq r$. It is straightforward to check that $Y$ is a tower
of Grassmann bundles, and in particular it is smooth and irreducible.

There is a natural projection $f \colon Y \to G(k,n)$ obtained by sending a
flag $U_1 \subset \cdots U_r \subset V$ to just $V$. For any flag in $Y$, we
have that $U_s \subseteq V \cap F_{n-k+a_s-b_s}$, and therefore $\dim V \cap
F_{n-k+a_s-b_s} \geq a_s$. This shows that the image of $f$ is exactly $f(Y) =
\Omega_\lambda$. Moreover, for $V$ in the smooth locus of $\Omega_\lambda$,
\cref{sing-locus} shows that there is exactly one flag in $Y$ mapping to $V$,
the one for which $U_s = V \cap F_{n-k+a_s-b_s}$. In other words, the morphism
$f \colon Y \to \Omega_\lambda$ is a resolution of singularities, and $f$ is an
isomorphism over the smooth locus of $\Omega_\lambda$.

We are now ready to prove the main result of this section.

\begin{proposition}
    \label{zelevinsky-exceptional}
    Let $\Omega_\lambda$ be a Schubert variety in $G(k,n)$, and let $f \colon Y
    \to \Omega_\lambda$ be the resolution of singularities described above.
    Then the exceptional components of $f$ are in bijection with the
    irreducible components of ${\rm Sing}(\Omega_\lambda)$.
\end{proposition}

\begin{proof}
Fix $1 \leq s \leq r$. We consider the subvariety
\[
    Z_s \subset 
    {\rm Fl}(a_1, \ldots, a_s, a_s+1, a_{s+1}, \ldots, a_r, k; n)
\]
given by those flags
\[
    U_1 
    \subset \cdots 
    \subset U_s \subset U_s^* \subseteq U_{s+1}
    \subset \cdots 
    U_r \subset V \subset \mathbb C^n
\]
that verify
\[
    \dim U_1 = a_1,
    \quad\ldots,\quad
    \dim U_r = a_r,
    \quad
    \dim V = k,
    \quad
    \dim U_s^* = a_s+1,
\]
and
\[
    U_1 \subseteq F_{n-k+a_1-b_1},
    \quad\ldots,\quad
    U_r \subseteq F_{n-k+a_r-b_r},
    \quad
    U_s^* \subseteq F_{n-k+a_s-b_s}.
\]
Similar considerations as the ones discussed above for $Y$ show that $Z_s$ is
smooth and, more importantly for us, irreducible. There is a natural morphism
$g_s \colon Z_s \to Y$, and the image of the composition $f \circ g_s$ is
$\Omega_{\lambda^s}$, one of the irreducible components of the singular locus
of $\Omega_\lambda$. We denote by $E_s \subset Y$ the image of the morphism
$g_s$. Notice that each of the $E_s$ is irreducible. We have shown that $f(E_1
\cup \cdots \cup E_r) = {\rm Sing}(\Omega_\lambda)$.

Consider now a flag $U_1 \subset \cdots \subset U_r \subset V$ in $Y$, and
assume that $V \in {\rm Sing}(\Omega_\lambda)$. Let $s$ be the largest index
such that $\dim V \cap F_{n-k+a_s-b_s} > a_s$. We know that such $s$ exists,
because $V \in {\rm Sing}(\Omega_\lambda)$. Let $U_s^*$ be any vector space of
dimension $a_s+1$ containing $U_s$ and contained in $V \cap F_{n-k+a_s-b_s}$.
From the construction of $s$, we must have that either $s=r$ or otherwise
\[
    U_{s+1} = V \cap F_{n-k+a_{s+1}-b_{s+1}} \supseteq U_s^*.
\]
Therefore $U_s^*$ can be used to define a flag in $Z_s$, and we see that the
original flag $U_1 \subset \cdots \subset U_r \subset V$ belongs to $E_s$. We have
shown that $f^{-1}({\rm Sing}(\Omega_\lambda)) = E_1 \cup \cdots \cup E_r$, as
required.
\end{proof}

\begin{corollary}
    \label{nash-zelevinsky}
    Let $\Omega_\lambda$ be a Schubert variety in $G(k,n)$. Then the Nash map
    for $\Omega_\lambda$ is bijective.
\end{corollary}

\begin{proof}
Recall from \cref{sec:arcs} that the Nash families are the irreducible
components of $\Cont^{\geq 1}({\rm Sing}(\Omega_\lambda))$, and that the Nash
map is an injection that associates to each Nash family an essential valuation.
Since ${\rm Sing}(\Omega_\lambda)$ has $r$ irreducible components, we see that
there are at least $r$ Nash families. \cref{zelevinsky-exceptional} shows that
there are at most $r$ essential valuations. The \lcnamecref{nash-zelevinsky}
follows from the fact that the Nash map is injective.
\end{proof}

\subsection*{Nash valuations}

For a Schubert variety $\Omega_\lambda$ in $G(k,n)$, the Nash valuations
(which, as we saw above, agree with the essential valuations) can be described
using contact strata. Recall that the arc space of $\Omega_\lambda$ is the
union of contact strata:
\[
    J_\infty \Omega_\lambda = \bigcup_\beta \mathcal C_\beta,
\]
where $\beta$ ranges among all plane partitions that have infinite height on
$\lambda$:
\[
    \beta_{i,j} = \infty
    \quad
    \forall (i,j) \in \lambda.
\]
As above, consider the partitions $\lambda^1, \ldots, \lambda^r$ giving the
irreducible components of the singular locus:
\[
    {\rm Sing}(\Omega_\lambda) 
    = \Omega_{\lambda_1} \cup \cdots \cup \Omega_{\lambda^r}.
\]
For each $1 \leq s \leq r$, we let $\beta^s$ be the plane partition with bottom
floor equal to $\lambda^s$ and an infinite number of floors equal to $\lambda$.
In other words, $\beta^s$ is given by
\[
    \beta_{i,j}^s
    =
    \begin{cases}
        \infty & \text{if $(i,j) \in \lambda$},\\
        1 & \text{if $(i,j) \in \lambda^s$ and $(i,j) \notin \lambda$},\\
        0 & \text{otherwise}.\\
    \end{cases}
\]
Then it follows from the above discussion on the Nash problem that the
Nash/essential valuations for $\Omega_\lambda$ are precisely
\[
    \ord_{\beta^1},\quad
    \ord_{\beta^2},\quad
    \ldots,\quad
    \ord_{\beta^r}.
\]
Notice that these are only semi-valuations on $G(k,n)$ (they have infinite
terms), but they are valuations on $\Omega_\lambda$.

We would like to remark that it is also possible to show directly that
\[
    \Cont^{\geq 1}({\rm Sing}(\Omega_\lambda))
    =
    \overline{\mathcal C}_{\beta^1}
    \cup \cdots \cup
    \overline{\mathcal C}_{\beta^r},
\]
without using resolutions. This is in fact an easy consequence of
\cref{nash-c}.


\vspace{2em}

\bibliographystyle{bib-style}
\bibliography{bibliography}

\end{document}

%% file: article-style.tex
\makeatletter

\def\@myMR[#1 #2]{\relax\ifhmode\unskip\spacefactor3000 \space\fi
  \MRhref{#1}{MR\,#1}}
\renewcommand\MR[1]{\@myMR[#1 ]}
\renewcommand{\MRhref}[2]{{\tiny%
  \href{http://www.ams.org/mathscinet-getitem?mr=#1}{#2}}%
}
\renewcommand*{\backref}[1]{}
\renewcommand*{\backrefalt}[4]{%
    \tiny%
    ({
    \ifcase #1 not cited%
          \or cit.\ on p.~#2%
          \else cit.\ on pp.~#2%
    \fi%
    })\\[-.6em]}

\def\maketitle{\par
  \@topnum\z@ 
  \@setcopyright
  \thispagestyle{empty}
  \ifx\@empty\shortauthors \let\shortauthors\shorttitle
  \else \andify\shortauthors
  \fi
  \@maketitle@hook
  \begingroup
  \@maketitle
  \toks@\@xp{\shortauthors}\@temptokena\@xp{\shorttitle}%
  \toks4{\def\\{ \ignorespaces}}
  \edef\@tempa{%
    \@nx\markboth{\the\toks4
      \@nx\MakeUppercase{\the\toks@}}{\the\@temptokena}}%
  \@tempa
  \endgroup
  \c@footnote\z@
    \renewcommand{\footnoterule}{%
      \kern -3pt
      \hrule width \textwidth height .5pt
      \kern 2pt
    }
  {
    \renewcommand\thefootnote{}
    \vspace{-2em}
    \footnote{
      \par\vspace{-1.2em}\noindent
      \def\@footnotetext##1{\noindent{\footnotesize##1}\par}%
      \let\@makefnmark\relax  \let\@thefnmark\relax
      \ifx\@empty\@date\else \@footnotetext{\@setdate}\fi
      \ifx\@empty\@subjclass\else \@footnotetext{\@setsubjclass}\fi
      \ifx\@empty\@keywords\else \@footnotetext{\@setkeywords}\fi
      \ifx\@empty\thankses\else \@footnotetext{%
        \def\par{\let\par\@par}\@setthanks}%
      \fi
    }
    \addtocounter{footnote}{-1}
  }
  \@cleartopmattertags
}

\def\@adminfootnotes{\@empty}

\def\@settitle{\begin{center}%
  \baselineskip14\p@\relax
    \bfseries
\Large
  \@title
  \end{center}%
}

\def\@setauthors{%
  \begingroup
  \def\thanks{\protect\thanks@warning}%
  \trivlist
  \centering\footnotesize \@topsep30\p@\relax
  \advance\@topsep by -\baselineskip
  \item\relax
  \author@andify\authors
  \def\\{\protect\linebreak}%
  \large{\authors}%
  \ifx\@empty\contribs
  \else
    ,\penalty-3 \space \@setcontribs
    \@closetoccontribs
  \fi
  \endtrivlist
  \endgroup
}

\def\@setaddresses{\par
  \nobreak \begingroup
\footnotesize
  \def\author##1{\end{minipage}\hskip 1sp \begin{minipage}{.5\textwidth}\raggedright%
    ~\\[2em]{\bf##1}\\[.5em]%
  }%
  \interlinepenalty\@M
  \def\address##1##2{\begingroup
    {\ignorespaces##2}\endgroup\\[.5em]}%
  \def\curraddr##1##2{\begingroup
    \@ifnotempty{##2}{\nobreak\indent\curraddrname
      \@ifnotempty{##1}{, \ignorespaces##1\unskip}\/:\space
      ##2\par}\endgroup}%
  \def\email##1##2{\begingroup
    \@ifnotempty{##2}{\nobreak\indent
      \@ifnotempty{##1}{, \ignorespaces##1\unskip}
      \ttfamily##2\par}\endgroup}%
  \def\urladdr##1##2{\begingroup
    \def~{\char`\~}%
    \@ifnotempty{##2}{\nobreak\indent\urladdrname
      \@ifnotempty{##1}{, \ignorespaces##1\unskip}\/:\space
      \ttfamily##2\par}\endgroup}%
  \setlength{\parindent}{0pt}%
  \vfill%
  {
  \begin{minipage}{0mm}
  \addresses
  \end{minipage}
  }
  \endgroup
}

\renewcommand{\author}[2][]{%
  \ifx\@empty\authors
    \gdef\authors{#2}%
    \g@addto@macro\addresses{\author{#2}}%
  \else
    \g@addto@macro\authors{\and#2}%
    \g@addto@macro\addresses{\author{#2}}%
  \fi
  \@ifnotempty{#1}{%
    \ifx\@empty\shortauthors
      \gdef\shortauthors{#1}%
    \else
      \g@addto@macro\shortauthors{\and#1}%
    \fi
  }%
}
\edef\author{\@nx\@dblarg
  \@xp\@nx\csname\string\author\endcsname}

\def\@secnumfont{\@empty}

\def\section{\@startsection{section}{1}%
  \z@{.7\linespacing\@plus\linespacing}{.5\linespacing}%
  {\large\bfseries\centering}}

\makeatother